\documentclass{amsart}

\usepackage{array}
\usepackage{tabularx}
\usepackage{amsmath, amssymb, amscd, calligra, mathrsfs}
\usepackage{mathrsfs}
\usepackage[all]{xy}
\usepackage{comment}
\usepackage{xcolor, graphicx}
\usepackage[
bookmarks=false,
colorlinks=true,
debug=true,
naturalnames=true,
pdfnewwindow=true,
citecolor=blue,
linkcolor=blue]{hyperref}
\usepackage{colonequals}
\newcommand{\seq}{\colonequals}

\theoremstyle{definition}
\newtheorem{Def}{Definition}[section]
\newtheorem{Rem}[Def]{Remark}
\newtheorem{Ex}[Def]{Example}
\theoremstyle{plain}
\newtheorem{Thm}[Def]{Theorem}
\newtheorem{Lem}[Def]{Lemma}
\newtheorem{Prop}[Def]{Proposition}
\newtheorem{Cor}[Def]{Corollary}

\numberwithin{equation}{section}

\newcommand{\C}{\mathbb{C}}
\newcommand{\kk}{\Bbbk}
\newcommand{\Cc}{\mathcal{C}}
\newcommand{\g}{\mathfrak{g}}

\newcommand{\Z}{\mathbb{Z}}
\newcommand{\modcat}{\text{-}\mathsf{mod}}
\newcommand{\udmod}{\text{-}\underline{\mathsf{mod}}}
\newcommand{\gmod}{\text{-}\mathsf{gmod}}
\newcommand{\vdim}{\underline{\dim}\,}
\newcommand{\Ext}{\mathop{\mathrm{Ext}}\nolimits}
\newcommand{\End}{\mathop{\mathrm{End}}\nolimits}
\newcommand{\Hom}{\mathop{\mathrm{Hom}}\nolimits}

\newcommand{\Ker}{\mathop{\mathrm{Ker}}}
\newcommand{\Ima}{\mathop{\mathrm{Im}}}

\newcommand{\Q}{\mathbb{Q}}

\newcommand{\KP}{\mathsf{KP}}
\newcommand{\pprime}{\prime \prime}
\newcommand{\M}{\mathfrak{M}}

\newcommand{\Mg}{\mathfrak{M}^{\bullet}}
\newcommand{\Mreg}{\mathfrak{M}_{0}^{\bullet \mathrm{reg}}}

\newcommand{\G}{\mathbb{G}}

\newcommand{\cP}{\mathsf{P}}
\newcommand{\cQ}{\mathsf{Q}}

\newcommand{\cR}{\mathsf{R}}

\newcommand{\KK}{\mathbb{K}}

\newcommand{\hR}{\widehat{R}}
\newcommand{\hH}{\widehat{H}}

\newcommand{\tZ}{\widetilde{\mathfrak{Z}}}
\newcommand{\B}{\mathcal{B}}
\newcommand{\F}{\mathcal{F}}

\newcommand{\mm}{\mathfrak{m}}

\newcommand{\rr}{\mathfrak{r}}

\newcommand{\hV}{\widehat{V}}
\newcommand{\Zg}{Z^{\bullet}}
\newcommand{\cZ}{\mathcal{Z}}
\newcommand{\OO}{\mathbb{O}}
\newcommand{\Oo}{\mathcal{O}}

\newcommand{\hK}{\widehat{K}}

\newcommand{\dD}{\mathcal{D}}
\newcommand{\ind}{\mathop{\mathsf{ind}}}
\newcommand{\tC}{\widetilde{C}}
\newcommand{\tc}{\widetilde{c}}
\newcommand{\cH}{\mathtt{H}}
\newcommand{\tDel}{\widetilde{\Delta}}
\newcommand{\rep}{\mathop{\mathsf{rep}}\nolimits}

\newcommand{\Mm}{\mathcal{M}}
\newcommand{\tLam}{\widetilde{\Lambda}}
\newcommand{\sV}{\mathsf{V}}
\newcommand{\sa}{\mathsf{a}}

\newcommand{\sx}{\mathsf{x}}

\newcommand{\fO}{\mathfrak{O}}

\newcommand{\ud}{\underline}

\numberwithin{equation}{section}

\title[Graded quiver varieties and 
normalized R-matrices]
    {Graded quiver varieties and singularities of 
normalized R-matrices for fundamental modules}

\date{\today}

 \author{Ryo Fujita}
 \address[R.~Fujita]{Research Institute for Mathematical Sciences, Kyoto University, Oiwake-Kitashirakawa, Sakyo, Kyoto, 606-8502, Japan \& Institut de Math\'{e}matiques de Jussieu-Paris Rive Gauche, Universit\'{e} de Paris, F-75013, Paris, France}
\email{rfujita@kurims.kyoto-u.ac.jp}

\begin{document}

\begin{abstract}
We present a simple unified formula expressing the denominators of the normalized 
$R$-matrices between the fundamental modules 
over the quantum loop algebras of type $\mathsf{ADE}$.
It has an interpretation in terms of representations of Dynkin quivers
and can be proved in a unified way using geometry of the graded quiver varieties. 
As a by-product, we obtain a geometric interpretation of 
Kang-Kashiwara-Kim's generalized quantum affine Schur-Weyl duality
functor when it arises from a family of the fundamental modules. 
We also study several cases 
when the graded quiver varieties are
isomorphic to unions of the graded nilpotent orbits of type $\mathsf{A}$.
\end{abstract}

\maketitle

\tableofcontents

\section{Introduction}
\label{Sec:intro}

\subsection{}
For a complex finite-dimensional simple Lie algebra $\g$, 
we can consider its (untwisted) quantum loop algebra $U_{q}(L\g)$ as a certain quantum affinization of 
the universal enveloping algebra $U(\g)$.
It is a Hopf algebra defined over the field $\kk = \overline{\Q(q)}$, where $q$ is the generic quantum parameter.
The structure of the monoidal abelian category $\Cc$ 
of finite-dimensional $U_{q}(L\g)$-modules
is much more complicated than that of $U(\g)$.
Indeed, the category $\Cc$ is neither semisimple as an abelian category,
nor braided as a monoidal category.
It has been studied by many researchers in connection with various research topics such as quantum integrable systems, 
combinatorics and cluster algebras.

The normalized $R$-matrices are constructed as intertwining operators
between tensor products of 
(relatively generic) simple objects of the category $\Cc$,
satisfying the quantum Yang-Baxter equation.
They can be seen as matrix-valued rational functions
in the spectral parameters, whose singularities strongly reflect the structure of tensor product modules (cf.~\cite{AK97, Kashiwara02}).
Thus, the singularities of normalized $R$-matrices
carry some important information
on the monoidal structure of $\Cc$.  

\subsection{A unified denominator formula}
In this paper, 
we focus on the normalized $R$-matrices 
between the fundamental modules 
over $U_{q}(L\g)$ 
associated with $\g$ of type $\mathsf{ADE}$.
Note that every simple object of $\Cc$
is obtained as a head of a suitably ordered tensor product of the fundamental modules. 
Thus, studying tensor products of the fundamental modules
can be thought of a first step toward a better understanding
of the monoidal structure of the whole category $\Cc$.

From now on, we assume that $\g$ is of type $\mathsf{ADE}$. Let $I$ be the set of Dynkin indices 
and $(c_{ij})_{i,j \in I}$ the Cartan matrix of $\g$.
For each $i \in I$, the $i$-th fundamental module
$V_{i}(a)$ is a simple object of $\Cc$,
which has
 a canonical highest weight vector $v_{i}$
and depends on a non-zero scalar $a \in \kk^{\times}$ called the spectral parameter.
Making the spectral parameters formal, for each $(i,j) \in I^{2}$, the normalized $R$-matrix $R_{ij}(z_2/z_1)$ 
is defined to be the unique
$U_{q}(L\g) \otimes \kk(z_{1}, z_{2})$-linear isomorphism 
$$
R_{ij}(z_{2}/z_{1}) \colon V_{i}(z_{1}) \otimes V_{j}(z_{2}) \to V_{j}(z_{2}) \otimes V_{i}(z_{1})
$$ 
satisfying the condition $R_{ij}(z_{2}/z_{1})(v_{i}\otimes v_{j}) = v_{j}\otimes v_{i}$.
Since the normalized $R$-matrix $R_{ij}(z_2/z_1)$ only rationally depends on the ratio $u = z_{2}/z_{1}$ of the spectral parameters,
one can consider its denominator $d_{ij}(u) \in \kk[u]$.
Explicit computations of these denominators $d_{ij}(u)$ have been accomplished in the separate works 
by Date-Okado~\cite{DO94} for type $\mathsf{A}$, 
by Kang-Kashiwara-Kim~\cite{KKK15} for type $\mathsf{D}$,
and by Oh-Scrimshaw~\cite{OS19, OS19c} for type $\mathsf{E}$. 
Note that
these computations relied on case-by-case arguments, which 
also required a use of computer  particularly for type $\mathsf{E}$.

The main theorem of this paper asserts that
these denominators $d_{ij}(u)$ can be expressed in a simple unified formula. 

\begin{Thm}[= Theorem~\ref{Thm:main}]
\label{Thm:imain}
For each $(i,j) \in I^{2}$, we have
$$
d_{ij}(u)=
\prod_{\ell=1}^{h-1}(u-q^{\ell+1})^{\tc_{ij}(\ell)},
$$
where $h$ is the Coxeter number of $\g$ and 
$\tc_{ij}(\ell)$ is the coefficient of $z^{\ell}$
in the formal expansion at $z=0$
of the $(i,j)$-entry of the inverse of 
the quantum Cartan matrix 
$\left( \frac{z^{c_{ij}} - z^{-c_{ij}}}{z-z^{-1}}\right)_{i,j \in I}$.
\end{Thm}
Note that the quantum Cartan matrix has appeared several times as a key combinatorial ingredient in the study of the category $\Cc$.
For example, it already appeared in 
the work of 
Frenkel-Reshetikhin~\cite{FR99}, which introduced 
the notion of $q$-characters
for finite-dimensional $U_{q}(L\g)$-modules.   

\subsection{An interpretation by quiver representations}
An advantage of our denominator formula is that
it admits an interpretation in terms of representations of 
a Dynkin quiver $Q$ of type $\g$. 
To describe it, we need additional notation.
Let us choose an $I$-tuple $(\epsilon_{i})_{i \in I} \in \{ 0,1\}^{I}$
such that $\epsilon_{i} \neq \epsilon_{j}$
whenever 
$c_{ij}=-1$.
Then we define
an infinite quiver 
$\Delta = (\Delta_{0}, \Delta_{1})$
by
\begin{align*}
\Delta_{0} &\seq \{(i,p) \in I \times \Z \mid p- \epsilon_{i} \in 2\Z \}, \\
\Delta_{1} &\seq \{(i,p) \to (j, p+1) \mid 
(i, p), (j, p+1) \in \Delta_{0}, \; c_{ij}=-1
\}.
\end{align*}
For instance, when $\g$ is of type $\mathsf{D}_{5}$,
the quiver $\Delta$ looks like: 
$$
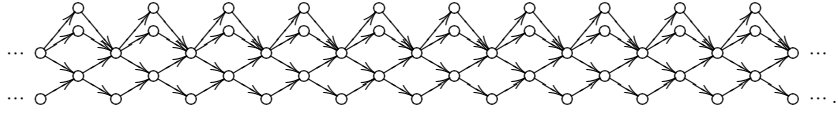
\begin{xy}
\def\objectstyle{\scriptstyle}
\ar@{->} (0,0) *++!R{\cdots} *\cir<2pt>{}="20";
(5,3) *\cir<2pt>{}="31", 
\ar@{->} "20";(5,-3) *\cir<2pt>{}="11", 
\ar@{->} "20";(5,6) *\cir<2pt>{}="41", 
\ar@{->} (0,-6) *++!R{\cdots} *\cir<2pt>{}="00";"11",
\ar@{->} "11"; (10,-6) *\cir<2pt>{}="02",
\ar@{->} "31"; (10,0) *\cir<2pt>{}="22",
\ar@{->} "11"; "22",
\ar@{->} "41"; "22",
\ar@{->} "22";(15,-3) *\cir<2pt>{}="13",
\ar@{->} "02"; "13",
\ar@{->} "22";(15,3) *\cir<2pt>{}="33",
\ar@{->} "22";(15,6) *\cir<2pt>{}="43",
\ar@{->} "13"; (20,0) *\cir<2pt>{}="24",
\ar@{->} "13"; (20,-6) *\cir<2pt>{}="04",
\ar@{->} "33"; "24",
\ar@{->} "43"; "24",
\ar@{->} "24";(25,-3) *\cir<2pt>{}="15",
\ar@{->} "04";"15",
\ar@{->} "24";(25,3) *\cir<2pt>{}="35",
\ar@{->} "24";(25,6) *\cir<2pt>{}="45",
\ar@{->} "15"; (30,0) *\cir<2pt>{}="26",
\ar@{->} "15"; (30,-6) *\cir<2pt>{}="06",
\ar@{->} "35"; "26",
\ar@{->} "45"; "26",
\ar@{->} "26";(35,-3) *\cir<2pt>{}="17",
\ar@{->} "06";"17",
\ar@{->} "26";(35,3) *\cir<2pt>{}="37",
\ar@{->} "26";(35,6) *\cir<2pt>{}="47",
\ar@{->} "17"; (40,0) *\cir<2pt>{}="28",
\ar@{->} "17"; (40,-6) *\cir<2pt>{}="08",
\ar@{->} "37"; "28",
\ar@{->} "47"; "28",
\ar@{->} "28";(45,-3) *\cir<2pt>{}="19",
\ar@{->} "08";"19",
\ar@{->} "28";(45,3) *\cir<2pt>{}="39",
\ar@{->} "28";(45,6) *\cir<2pt>{}="49",
\ar@{->} "19"; (50,0) *\cir<2pt>{}="210",
\ar@{->} "19"; (50,-6) *\cir<2pt>{}="010",
\ar@{->} "39"; "210",
\ar@{->} "49"; "210",
\ar@{->} "210";(55,-3) *\cir<2pt>{}="111",
\ar@{->} "010";"111",
\ar@{->} "210";(55,3) *\cir<2pt>{}="311",
\ar@{->} "210";(55,6) *\cir<2pt>{}="411",
\ar@{->} "111"; (60,0) *\cir<2pt>{}="212",
\ar@{->} "111"; (60,-6) *\cir<2pt>{}="012",
\ar@{->} "311"; "212",
\ar@{->} "411"; "212",
\ar@{->} "212";(65,-3) *\cir<2pt>{}="113",
\ar@{->} "012";"113",
\ar@{->} "212";(65,3) *\cir<2pt>{}="313",
\ar@{->} "212";(65,6) *\cir<2pt>{}="413",
\ar@{->} "113"; (70,0) *\cir<2pt>{}="214",
\ar@{->} "113"; (70,-6) *\cir<2pt>{}="014",
\ar@{->} "313"; "214",
\ar@{->} "413"; "214",
\ar@{->} "214";(75,-3) *\cir<2pt>{}="115",
\ar@{->} "014";"115",
\ar@{->} "214";(75,3) *\cir<2pt>{}="315",
\ar@{->} "214";(75,6) *\cir<2pt>{}="415",
\ar@{->} "115"; (80,0) *\cir<2pt>{}="216",
\ar@{->} "115"; (80,-6) *\cir<2pt>{}="016",
\ar@{->} "315"; "216",
\ar@{->} "415"; "216",
\ar@{->} "216";(85,-3) *\cir<2pt>{}="117",
\ar@{->} "016";"117",
\ar@{->} "216";(85,3) *\cir<2pt>{}="317",
\ar@{->} "216";(85,6) *\cir<2pt>{}="417",
\ar@{->} "117"; (90,0)  *\cir<2pt>{}="218",
\ar@{->} "117"; (90,-6) *\cir<2pt>{}="018",
\ar@{->} "317"; "218",
\ar@{->} "417"; "218",
\ar@{->} "218";(95,-3) *\cir<2pt>{}="119",
\ar@{->} "018";"119",
\ar@{->} "218";(95,3)  *\cir<2pt>{}="319",
\ar@{->} "218";(95,6)  *\cir<2pt>{}="419",
\ar@{->} "119"; (100,0)  *++!L{\cdots} *\cir<2pt>{}="220",
\ar@{->} "119"; (100,-6) *++!L{\cdots.} *\cir<2pt>{}="020",
\ar@{->} "319"; "220",
\ar@{->} "419"; "220",
\end{xy}
$$

It was shown by Happel~\cite{Happel87, Happel88}
that the quiver $\Delta$ is isomorphic to 
the Auslander-Reiten quiver of the bounded derived category  
$\dD_{Q} \seq D^{b}(\C Q \modcat)$ of representations of the Dynkin quiver $Q$. 
In particular, there is a nice bijection $\cH_{Q}$
from the vertex set $\Delta_{0}$ to 
the set of isomorphism classes of indecomposable objects of $\dD_{Q}$.

An intimate connection between  
the Auslander-Reiten quiver of $\dD_{Q}$ and 
the category $\Cc$ was originally observed by
Hernandez-Leclerc~\cite{HL15}.
In that paper, it was shown that the integers $\tc_{ij}(\ell)$ can be expressed as
Euler-Poincar\'e characteristics
of suitable pairs of indecomposable objects of $\dD_{Q}$.
Using this interpretation, one can see that the following assertion 
is equivalent to Theorem~\ref{Thm:imain}.

\begin{Thm}[= Theorem~\ref{Thm:main2}]
\label{Thm:imain2}
For any $(i, p), (j, r) \in \Delta_{0}$, 
the pole order of the normalized $R$-matrix $R_{ij}(u)$ 
at $u = q^{r}/q^{p}$
 is equal to 
$
\dim \Ext_{\dD_{Q}}^{1}(\cH_{Q}(j,r), \cH_{Q}(i,p)).
$
\end{Thm}

This yields the following interesting corollary.

\begin{Cor}[= Corollary~\ref{Cor:main2}] 
For any $(i,p), (j,r) \in \Delta_{0}$, the following conditions are mutually equivalent:
\begin{itemize}
\item The tensor product $V_{i}(q^{p}) \otimes V_{j}(q^{r})$ is irreducible;
\item $V_{i}(q^{p}) \otimes V_{j}(q^{r}) \cong V_{j}(q^{r}) \otimes V_{i}(q^{p})$ as $U_{q}(L\g)$-modules;
\item $\Ext_{\dD_{Q}}^{1}(\cH_{Q}(i,p), \cH_{Q}(j,r)) = 0$ and 
$\Ext_{\dD_{Q}}^{1}(\cH_{Q}(j,r), \cH_{Q}(i,p)) = 0$.
\end{itemize}
\end{Cor}

\subsection{Graded quiver varieties}
In this paper, we give a unified proof of Theorem~\ref{Thm:imain2}
(and hence Theorem~\ref{Thm:imain})
without using a computer.
Instead, we use geometry of the graded quiver varieties.

The graded quiver varieties were originally defined
by Nakajima~\cite{Nakajima01} as suitable 
torus fixed loci of the usual Nakajima quiver varieties, which provide a useful geometric setting to study 
finite-dimensional $U_{q}(L\g)$-modules 
when $\g$ is of type $\mathsf{ADE}$.
Given a finite-dimensional $\Delta_{0}$-graded $\C$-vector space 
$W = \bigoplus_{x \in \Delta_{0}}W_{x}$,
one can associate the graded quiver variety
$\Mg_{0}(W)$,
which is an affine complex algebraic variety
equipped with an action of the group 
$G_{W} = \prod_{x \in \Delta_{0}}GL(W_{x})$.

Our proof of Theorem~\ref{Thm:imain2} is based on the following beautiful result obtained by Keller-Scherotzke~\cite{KS16}
in their categorical study of the graded quiver varieties.
It also generalizes an important result
by Hernandez-Leclerc~\cite[Section 9]{HL15}.
Let $\Gamma = (\Gamma_{0}, \Gamma_{1})$
be another infinite quiver with the vertex set $\Gamma_{0} \seq \Delta_{0}$, whose arrow set $\Gamma_{1}$ is given by
the following condition:
for each $x,y \in \Delta_{0}$
the number of arrows from $x$ to $y$ 
is equal to $\dim \Ext^{1}_{\dD_{Q}}(\cH_{Q}(x), \cH_{Q}(y))$.
With this notation,
Keller-Scherotzke's theorem tells us that, for each $\Delta_{0}$-graded vector space $W$,
there exists a $G_{W}$-equivariant closed embedding of varieties
\begin{equation}
\label{Eq:iKS}
\Mg_{0}(W) \hookrightarrow \rep_{W}(\Gamma),
\end{equation}
where $\rep_{W}(\Gamma)$ denotes 
the affine space parametrizing representations of 
the quiver $\Gamma$ realized on $W$. 

In the special case when 
$W = W_{(i,p)} \oplus W_{(j,r)}$ 
for some $(i,p), (j,r) \in \Delta_{0}$ with $p \le r$
and $\dim W_{(i,p)} = \dim W_{(j,r)} =1$,
the above embedding \eqref{Eq:iKS} becomes an isomorphism.
Namely, the graded quiver variety $\Mg_{0}(W)$
in this case
is just an affine space whose dimension is equal to 
$\dim \Ext^{1}_{\dD_{Q}}(\cH_{Q}(j,r), \cH_{Q}(i,p))$.
This simple situation enables us to prove
Theorem~\ref{Thm:imain2} 
by using Nakajima's theory~\cite{Nakajima01} and a standard technique in geometric representation theory.

\subsection{Generalized quantum affine Schur-Weyl duality}
In the paper~\cite{KKK18}, Kang-Kashiwara-Kim
gave a general construction of a monoidal functor $\mathscr{F}_{J}$, 
called the generalized quantum affine Schur-Weyl duality functor,
associated with a given family $\{ \sV_{j} \}_{j \in J}$ of real simple objects of $\Cc$.
It connects the category $\Cc$ with a category of modules over 
the symmetric quiver Hecke algebra $H_{J}$
associated with a quiver $\Gamma_{J}$
determined by the singularities of normalized $R$-matrices
between the simple objects in $\{ \sV_{j} \}_{j \in J}$. 
The quiver Hecke algebras are $\Z$-graded algebras 
introduced by
Khovanov-Lauda~\cite{KL09} 
and by Rouquier~\cite{Rouquier08} independently
to establish a categorification
of the half of the quantized enveloping algebra
associated with a general symmetrizable Kac-Moody algebra.
In this sense, 
the quiver Hecke algebra $H_{J}$
is a generalization of
the affine Hecke algebra of type $\mathsf{A}$, and
Kang-Kashiwara-Kim's construction can be thought of a generalization of the usual quantum affine Schur-Weyl duality between $U_{q}(L\mathfrak{sl}_{n})$ and the affine Hecke algebra of type $\mathsf{A}$.

In the subsequent works by   
Kang, Kashiwara, Kim, Oh, Park and 
Scrimshaw \cite{KKK15, KKKO15, KKKO16, KKO19, KO19, OS19, KKOP19},
many interesting examples of the functor $\mathscr{F}_{J}$ are constructed.
In these nice examples, the functor $\mathscr{F}_{J}$ induces an isomorphism of Grothendieck rings between a category of finite-dimensional $H_{J}$-modules (or its suitable modification),
and a certain monoidal subcategory $\Cc_{J}$ of $\Cc$.

In this paper,  
we give a geometric interpretation of 
the functor $\mathscr{F}_{J}$ 
whenever it arises from a family $\{ \sV_{j}\}_{j \in J}$
of fundamental modules of type $\mathsf{ADE}$.
More precisely, we realize the  bimodule corresponding to the functor $\mathscr{F}_J$
via the equivariant $K$-theory of the graded quiver varieties,
mimicking
Ginzburg-Reshetikhin-Vasserot's geometric realization 
of the usual quantum affine Schur-Weyl duality~\cite{GRV94}.
This is a generalization of 
the author's previous result~\cite{Fujita18}.
A key fact in our construction  is that the quiver $\Gamma_{J}$  defining the quiver Hecke algebra $H_{J}$ 
is identical to a full subquiver of the quiver $\Gamma$
that appeared in Keller-Scherotzke's theorem above.  
This is a direct consequence of Theorem~\ref{Thm:imain2}
and explains the appearance of the quiver Hecke algebra $H_{J}$ from a geometric point of view. 
 
\subsection{Type $\mathsf{A}$ subquivers and graded nilpotent orbits} 
As an example of the above construction, 
with a given subquiver $Q^{\prime}$ of a Dynkin quiver $Q$ which is isomorphic to a quiver of type $\mathsf{A}$ with monotone orientation,
we associate a specific family 
$\{ \sV_{j} \}_{j \in \Z}$ of fundamental modules
labeled by the set of integers $\Z$.
We prove that the associated quiver $\Gamma_{J}$
is of type $\mathsf{A}_{\infty}$ with monotone orientation, and 
the corresponding graded quiver varieties are isomorphic 
to unions of graded nilpotent orbits of type $\mathsf{A}$.
Moreover, we show that the associated functor $\mathscr{F}_{J}$
induces an isomorphism of Grothendieck rings 
between a certain localization $\mathcal{T}_{N}$ of the module category of $H_{J}$ constructed in~\cite{KKK18}
and the monoidal full subcategory $\Cc_{\dD_{Q^{\prime}}}$ of $\Cc$
generated by the fundamental modules $V_{i}(q^{p})$ such that 
$\cH_{Q}(i,p) \in \dD_{Q^{\prime}} \subset \dD_{Q}$.
In some special cases of type $\mathsf{AD}$,
the associated functors $\mathscr{F}_{J}$ coincide with the ones studied in \cite{KKK18, KKO19, KKOP19}.  
Recently, Kashiwara-Kim-Oh-Park~\cite{KKOP19} proved 
that the localized category $\mathcal{T}_{N}$ gives a monoidal categorification 
of a certain cluster algebra of infinite rank.  
Therefore, we conclude that 
our monoidal category $\Cc_{\dD_{Q^{\prime}}}$ always 
inherits the same cluster structure 
from the category $\mathcal{T}_{N}$ 
via the monoidal functor $\mathscr{F}_{J}$.  

\subsection{Remark}
Note that 
explicit computations of the denominators $d_{ij}(u)$
for the other non-symmetric affine types have been also accomplished in the separate works by
Akasaka-Kashiwara~\cite{AK97} for type $\mathsf{C}$,
by Oh~\cite{Oh15} for type $\mathsf{B}$ and for doubly-twisted type $\mathsf{AD}$,
and by Oh-Scrimshaw~\cite{OS19, OS19c} for all the remaining cases.
Unfortunately, our geometric approach using the graded quiver varieties
is applicable only 
to the cases of untwisted type $\mathsf{ADE}$ (i.e.~symmetric affine types).   
At this moment, it is unclear whether there is
an analogous geometric approach to compute
the denominators $d_{ij}(u)$ for the non-symmetric types.

\subsection{Organization}
This paper is organized as follows.
In Section~\ref{Sec:denominator},
we recall some known facts about 
the representation theory of the quantum loop algebras 
$U_{q}(L\g)$ of type $\mathsf{ADE}$
and state our main theorem.
In Section~\ref{Sec:quiver}, 
we present an interpretation of our denominator formula
in terms of representations of Dynkin quivers.
After reviewing the graded quiver varieties in Section~\ref{Sec:Nakajima},
we give a geometric proof of our denominator formula in Section~\ref{Ssec:pf}.
In Section~\ref{Ssec:simplepole}, we add a remark on the case when the normalized $R$-matrix has a simple pole.     
Section~\ref{Sec:SW} is devoted to a study of the generalized quantum affine Schur-Weyl duality. 
In Section~\ref{Ssec:ginterpret}, we give a geometric interpretation
of the functor $\mathscr{F}_{J}$
when it arises from a family of fundamental modules.
Finally, we study some examples where the graded quiver varieties are isomorphic to 
unions of graded nilpotent orbits of type $\mathsf{A}$ 
in Section~\ref{Ssec:gnilp}.   

\subsection{Acknowledgments} 
The author is grateful to Se-jin Oh
for his interest in this paper and 
for answering the author's questions on his papers. 
The author was supported by
Grant-in-Aid for JSPS 
Research Fellow (No.~18J10669), and by JSPS Overseas Research Fellowships during the revision.

\subsection{Overall convention}
Working over a base field $\mathbb{F}$,
we often write $\otimes$ (resp.~$\Hom$, $\dim$) 
instead of $\otimes_{\mathbb{F}}$ (resp.~$\Hom_{\mathbb{F}}$, $\dim_{\mathbb{F}}$) 
suppressing the symbol $\mathbb{F}$
for simplicity. 
For an algebra $A$ over a field $\mathbb{F}$,
we denote by $A \modcat$ 
the category of left $A$-modules
which are finite-dimensional over $\mathbb{F}$.
We denote by $A^{\mathrm{op}}$ (resp.~$A^{\times}$)
the opposite algebra (resp.~the multiplicative group of invertible elements) of $A$.

\section{A unified denominator formula}   
\label{Sec:denominator}
In this section, we recall some known facts
on representation theory 
of the quantum loop algebras of type $\mathsf{ADE}$
and state our main theorem.  

\subsection{Notation}
\label{Ssec:Notation}
  
Throughout this paper, 
we fix a finite-dimensional complex simple Lie algebra $\g$
of type $\mathsf{A}_{n}\, (n \in \Z_{\ge 1})$, 
$\mathsf{D}_{n}\, (n \in \Z_{\ge 4})$, 
or $\mathsf{E}_{n}\, (n=6,7,8)$.
Let $I\seq\{1, 2, \ldots, n\}$ be the set of Dynkin indices.
The Cartan matrix of $\g$ 
is denoted by $(c_{ij})_{i,j \in I}$.
We write $i \sim j$ if $c_{ij}=-1$. 

Let $\cP^{\vee} = \bigoplus_{i \in I} \Z h_{i}$ 
be the coroot lattice of $\g$.
The fundamental weights $\{ \varpi_{i}\}_{i \in I}$ 
form a basis of  the weight lattice 
$\cP = \Hom_{\Z}(\cP^{\vee}, \Z)$
which is dual to $\{h_{i} \}_{i \in I}$.
Let $\alpha_{i} = \sum_{j \in I} c_{ij} \varpi_{j}$ 
be the $i$-th simple root and 
$\cQ = \bigoplus_{i \in I} \Z \alpha_{i} \subset \cP$ 
be the root lattice.
We put $\cP^{+} = \sum_{i \in I}\Z_{\ge 0} \varpi_{i}$
and $\cQ^{+} = \sum_{i \in I} \Z_{\ge 0} \alpha_{i}$.
Denote by $(-,-)$ the symmetric bilinear form on $\cP \otimes_{\Z} \Q$
given by $(\alpha_{i}, \varpi_{j}) = \delta_{ij}$, or equivalently
$(\alpha_{i}, \alpha_{j}) = c_{ij}$.
Let $\mathsf{W}$ be the Weyl group of $\g$.
It is a finite group of 
linear transformations on $\cP$ 
generated by the simple reflections $\{r_{i}\}_{ i \in I}$
defined by
$r_{i}(\lambda) \seq \lambda - \lambda(h_{i}) \alpha_{i}$
for $\lambda \in \cP$. 
The set $\cR^{+}$ of positive roots is defined by 
$\cR^{+} = (\mathsf{W}\{ \alpha_{i}\}_{ i \in I}) \cap \cQ^{+}$.
Let $h \seq 2 |\mathsf{R}^{+}|/n$ be the Coxeter number of $\g$.

We fix an $I$-tuple $\epsilon = (\epsilon_{i})_{i \in I} \in \{0, 1\}^{I}$ such that
$\epsilon_{i} \neq \epsilon_{j}$ whenever $i \sim j$.
We refer to such an $\epsilon$ as a {\em parity function}. 
Note that we have only two possible choices of $\epsilon$ and 
the difference does not affect the main results of this paper.

\subsection{Quantum loop algebra}
\label{Ssec:qloop}

Let $q$ be an indeterminate and 
$\kk \seq \overline{\Q(q)}$ be the algebraic closure of the field $\Q(q)$
of rational functions in $q$ with rational coefficients
inside the ambient field $\bigcup_{m \in \Z_{>0}} \overline{\Q}(q^{1/m})$.

\begin{Def}
\label{Def:qloop}
The quantum loop algebra 
$U_{q}(L\g)$ 
associated with $\g$
is defined as a $\kk$-algebra with the generators:
$$
\left\{ x^{+}_{i,r} , x^{-}_{i,r} \mid i \in I, r \in \Z \right\}\cup 
\left\{ q^{y} \mid y \in \cP^{\vee} \right\} \cup
\left\{ h_{i, m} \mid i \in I, m \in \Z \setminus \{0\} \right\}$$
satisfying the following relations:
$$
q^{0} = 1, \quad
q^{y} q^{y^{\prime}} = q^{y+y^{\prime}}, \quad
[q^{y}, h_{i,m}] = [h_{i,m}, h_{j, l}] = 0, \quad
q^{y} x^{\pm}_{i,r} q^{-y} = q^{\pm \alpha_{i}(y)} x^{\pm}_{i,r}, 
$$
$$
(z-q^{\pm c_{ij}}w)\phi_{i}^{\varepsilon}(z) x_{j}^{\pm}(w)
= (q^{\pm c_{ij}}z - w)x_{j}^{\pm}(z) \phi_{i}^{\varepsilon}(w), 
$$
$$
(z - q^{\pm c_{ij}}w)x_{i}^{\pm}(z) x_{j}^{\pm}(w) 
=
(q^{\pm c_{ij}}z-w)
x_{j}^{\pm}(w) x_{i}^{\pm}(z),   
$$
$$
[x_{i}^{+}(z), x_{j}^{-}(w)] = 
\frac{\delta_{ij}}{q - q^{-1}} 
\left( \delta \left( \frac{w}{z} \right) \phi_{i}^{+}(w)
- \delta \left( \frac{z}{w} \right) \phi_{i}^{-}(z) \right), 
$$
\begin{multline}
\left\{ x_{i}^{\pm}(z_{1}) x_{i}^{\pm}(z_{2}) x_{j}^{\pm}(w)
-(q + q^{-1})x_{i}^{\pm}(z_{1}) x_{j}^{\pm}(w)x_{i}^{\pm}(z_{2})
+ x_{j}^{\pm}(w) x_{i}^{\pm}(z_{1}) x_{i}^{\pm}(z_{2}) 
\right\} \\
+
\{ z_{1} \leftrightarrow z_{2}\}
= 0
\qquad \text{if $i \sim j$},
\nonumber
\end{multline}
where $\varepsilon \in \{+, - \}$
and $\delta(z), x_{i}^{\pm}(z), \phi_{i}^{\pm}(z)$ 
are the formal series defined as follows:
$$
\delta(z) \seq \sum_{r=-\infty}^{\infty}z^{r},
\quad
x_{i}^{\pm}(z) \seq \sum_{r=-\infty}^{\infty}x_{i,r}^{\pm}z^{-r},
$$
$$
\phi_{i}^{\pm}(z) \seq q^{\pm h_{i}}
\exp \left( \pm (q-q^{-1}) 
\sum_{m=1}^{\infty}h_{i, \pm m} z^{\mp m} \right).
$$
In the last relation, the second term $\{ z_{1} \leftrightarrow z_{2}\}$
means the exchange of $z_{1}$ with $z_{2}$ in the first term.
\end{Def}

Let $\widehat{\g}$ be the (untwisted) affine Lie algebra
associated with $\g$. It is realized as 
$$
\widehat{\g} = L\g \oplus \C c \oplus \C d
$$ 
with a suitable Lie algebra structure,
where
$L \g \seq \g \otimes \C[z^{\pm 1}]$
is the loop algebra of $\g$,
$c$ is a central element and $d \seq z \frac{\mathrm{d}}{\mathrm{d}z}$ is
the degree operator. 
The derived subalgebra 
$
\widehat{\g}^{\prime} =
[\widehat{\g}, \widehat{\g}] = L\g \oplus \C c
$
is a central extension of the loop algebra $L\g$.
Let $U_{q}(\widehat{\g})$ be the quantized enveloping algebra of 
$\widehat{\g}$. This is a Hopf algebra over $\kk$
presented by the Chevalley type generators
$\{ e_{i}, f_{i} \mid i \in I \cup \{ 0 \} \} \cup 
\{ q^{y} \mid y \in 
\cP^{\vee} \oplus \Z c \oplus \Z d \}$
and the well-known relations.
The coproduct
$\triangle \colon U_{q}(\widehat{\g})
\to U_{q}(\widehat{\g}) \otimes U_{q}(\widehat{\g})$ is given
by:
$$
\triangle(e_{i}) = e_{i} \otimes q^{-h_{i}} + 1 \otimes e_{i}, \quad
\triangle(f_{i}) = f_{i} \otimes 1 + q^{h_{i}} \otimes f_{i}, \quad
\triangle(q^{y}) = q^{y} \otimes q^{y}  
$$
for $i \in I \cup \{0\}, y \in \cP^{\vee}\oplus \Z c \oplus \Z d$.
The subalgebra $U_{q}^{\prime}(\widehat{\g})$
generated by the generators
$\{ e_{i}, f_{i}, q^{\pm h_{i}} \mid i \in I \cup \{ 0\} \}$ 
is a Hopf subalgebra of $U_{q}(\widehat{\g})$, which
is regarded as a $q$-deformation of 
the universal enveloping algebra of $\widehat{\g}^{\prime}$.
By Beck~\cite{Beck94},
we have a $\kk$-algebra isomorphism
$U_{q}(L\g) \cong 
U_{q}^{\prime}(\widehat{\g}) / \langle q^{c} -1 \rangle$,
via which
the quantum loop algebra $U_{q}(L\g)$  
inherits a structure of Hopf algebra.
Actually this isomorphism depends on 
the choice of a function 
$o \colon I \to \{ \pm 1\}$ satisfying
$o(i) = -o(j)$ if $i \sim j$.
In this paper, we set $o(i) \seq (-1)^{\epsilon_{i}}$
by using the parity function $\epsilon$ 
we fixed in the last subsection.

\subsection{Simple and fundamental modules}
\label{Ssec:fundamental}

A $U_{q}(L\g)$-module is said to be 
of type $\mathbf{1}$ if, for each $i \in I$, the element $q^{h_{i}}$
acts on it as a semisimple linear operator whose eigenvalues belong to $q^{\Z}$.  
Let $\Cc$ denote the category of 
finite-dimensional $U_{q}(L\g)$-modules of type $\mathbf{1}$. 
The category $\Cc$ is a $\kk$-linear abelian monoidal category.

It is well-known that the simple modules of the category $\Cc$ are parametrized by 
so-called Drinfeld polynomials~\cite{CP95}, or equivalently by the dominant monomials~\cite{FR99}, 
which we recall here.  
Let $\Mm$ be the abelian (multiplicative) group freely generated by 
the symbols $\{Y_{i,a}\}_{(i,a) \in I \times \kk^{\times}}$ and
$\Mm^{+}$ be the submonoid of $\Mm$ generated by $\{Y_{i,a}\}_{(i,a) \in I \times \kk^{\times}}$.  
We refer to an element of $\Mm^{+}$ as a {\em dominant monomial}.

\begin{Thm}[{Chari-Pressley~\cite[Theorem 3.3]{CP95}}]
For each dominant monomial $m = \prod_{(i,a)} Y_{i,a}^{m_{i,a}}$, 
there exists a simple module $L(m) \in \Cc$ 
with a non-zero vector $v_{m} \in L(m)$ 
satisfying
$$
x_{i}^{+}(z)v_{m} = 0, \quad
\phi_{i}^{\pm}(z)v_{m} =
\left(
\prod_{a \in \kk^{\times}} 
\left(
\frac{q-q^{-1}az^{-1}}{1-az^{-1}}\right)^{m_{i,a}}
\right)^{\pm} v_{m}
$$
for each $i \in I$,
where $(-)^{\pm}$ denotes 
the formal expansion at $z^{\mp 1}=0$.
Such a vector $v_{m} \in L(m)$ is
unique up to $\kk^{\times}$.
Moreover, the correspondence 
$m \mapsto L(m)$
gives a bijection between the set $\Mm^{+}$
of dominant monomials and the set of
isomorphism classes of simple modules of $\Cc$. 
\end{Thm}

For each $(i, a) \in I \times \kk^{\times}$, we define an element $A_{i,a} \in \Mm$ by
$$
A_{i,a} \seq Y_{i, qa}Y_{i,q^{-1}a}\cdot\prod_{j\sim i} Y_{j,a}^{-1}.
$$
For $m, m^\prime \in \Mm$, we write $m \le m^{\prime}$
if $m^{\prime} m^{-1}$ is a monomial in $\{ A_{i,a} \}_{(i,a) \in I \times \kk^{\times}}$.
This defines a partial ordering on the set $\Mm^{+}$ of dominant monomials.

The simple modules
$L(Y_{i,a})$ corresponding to the degree $1$ dominant monomials $Y_{i,a} \in \Mm^{+}$, $(i,a) \in I \times \kk^{\times}$,
are called the {\em fundamental modules}.
The next theorem shows their importance in the monoidal category $\Cc$. 

\begin{Thm}[Frenkel-Reshetikhin~\cite{FR99}, Frenkel-Mukhin~\cite{FM01}, Nakajima~\cite{Nakajima01}]
\label{Thm:FR}
Let $K(\Cc)$ denote the Grothendieck ring of the monoidal abelian category $\Cc$.
\begin{enumerate}
\item \label{Thm:FR:1}
The ring $K(\Cc)$ is isomorphic to the polynomial ring $\Z[t_{i,a} \mid (i,a) \in I \times \kk^{\times}]$
in infinitely many variables, where the variable $t_{i,a}$ corresponds
to the class of the fundamental module $L(Y_{i,a})$.
In particular, $K(\Cc)$ is commutative;
\item \label{Thm:FR:2}
For each dominant monomial $m = \prod_{i,a} Y_{i,a}^{m_{i,a}} \in \Mm^{+}$, we have
$$
\prod_{i,a}[L(Y_{i,a})]^{m_{i,a}} = [L(m)] + \sum_{m^{\prime} \in \Mm^{+}, m^{\prime} \lneq m} c(m, m^{\prime}) [L(m^{\prime})]
$$ 
in the Grothendieck ring $K(\Cc)$, where $c(m, m^{\prime}) \in \Z_{\ge 0}.$
\end{enumerate}
\end{Thm}
\begin{proof}
\eqref{Thm:FR:1} is \cite[Corollary 2]{FR99}.
\eqref{Thm:FR:2} was originally conjectured by \cite{FR99} 
and proved by \cite[Theorem 4.1]{FM01} 
and \cite[Proposition 5.2]{Nakajima01} independently.
\end{proof}


\subsection{Normalized $R$-matrices and their denominators}

\begin{Def} \label{Def:affinization}
Let $M$ be a $U_{q}(L\g)$-module and
$z$ be a formal parameter.
We equip the free $\kk[z^{\pm 1}]$-module
$M[z^{\pm 1}] \seq M \otimes_{\kk} \kk[z^{\pm 1}]$
with a left $U_{q}(L\g)$-module structure by
\begin{align*}
x_{i, r}^{\pm}\cdot (v\otimes f(z)) &\seq x_{i,r}^{\pm}v \otimes z^{r} f(z), \\
q^{y} \cdot (v\otimes f(z)) &\seq q^{y}v \otimes f(z), \\
h_{i,m} \cdot (v\otimes f(z)) &\seq h_{i,m}v \otimes z^{m}f(z), 
\end{align*}
where $v \in M, f(z) \in \kk[z^{\pm 1}]$.
We refer to the resulting $U(L\g)\otimes_{\kk} \kk[z^{\pm 1}]$-module
$M[z^{\pm 1}]$
as the {\em affinization} of $M$.\footnote{In \cite{Kashiwara02}, the affinization is defined in terms of the Chevalley type generators of the algebra $U'_q(\widehat{\g})$.
One can easily see that it coincides with our affinization in Definition~\ref{Def:affinization} under the isomorphism $U_{q}(L\g) \cong 
U_{q}^{\prime}(\widehat{\g}) / \langle q^{c} -1 \rangle$ in \cite{Beck94}. } 
\end{Def}

To simplify the notation,
we denote
the {\em affinized fundamental module} and its generating vector by
$$
V_{i}[z^{\pm 1}] \seq L(Y_{i,1})[z^{\pm 1}], \quad
v_{i} \seq (v_{Y_{i,1}}) \otimes 1
$$
for each $i \in I$.
In addition, for any non-zero scalar $a \in \kk^{\times}$, we set
$$V_{i}(a) \seq V_{i}[z^{\pm 1}]/(z-a)V_{i}[z^{\pm 1}]$$
and denote by $v_{i}(a)$ the image of the vector $v_{i}$
under the canonical quotient map $V_{i}[z^{\pm 1}] \to V_{i}(a)$.
With this notation, we have
an isomorphism 
$V_{i}(a) \cong L(Y_{i,a})$
of $U_{q}(L\g)$-modules
via which the vector $v_{i}(a)$ corresponds to 
the vector $v_{Y_{i,a}}$.

\begin{Rem}
The affinized fundamental module $V_{i}[z^{\pm 1}]$
is known to be isomorphic to the following modules:
\begin{itemize}
\item the level zero extremal weight module of extremal weight $\varpi_{i}$
introduced by Kashiwara~\cite{Kashiwara94, Kashiwara02};
\item the global Weyl module of highest weight $\varpi_{i}$ 
introduced by Chari-Pressley~\cite{CP01};
\item the standard module associated with $\varpi_{i}$, 
realized via the equivariant $K$-theory of quiver varieties by Nakajima~\cite{Nakajima01}
(see Section~\ref{Ssec:Nakajimahom} below).
\end{itemize}
For a proof, see
\cite[Section 5]{Kashiwara02} and
\cite[Remark 2.15]{Nakajima04}. 
\end{Rem}

For each pair $(i,j) \in I^{2}$,
there is a unique $U_{q}(L\g) \otimes_{\kk} \kk[z_{1}^{\pm1}, z_{2}^{\pm1}]$-homomorphism
called the {\it normalized $R$-matrix} 
$$
R_{ij}
\colon V_{i}[z_{1}^{\pm 1}] \otimes V_{j}[z_{2}^{\pm 1}]
\to 
\kk(z_{2}/z_{1}) \otimes_{\kk[(z_{2}/z_{1})^{\pm 1}]}
\left( V_{j}[z_{2}^{\pm 1}] \otimes V_{i}[z_{1}^{\pm 1}] \right),
$$
satisfying the condition 
$R_{ij}
(v_{i}\otimes v_{j}) = v_{j} \otimes
v_{i}$ (see~\cite[Appendix A]{AK97} or~\cite[Section 8]{Kashiwara02}).
The {\em denominator} of the normalized $R$-matrix
$R_{ij}$ is a unique monic polynomial 
$d_{ij}(u) \in \kk[u]$ of the smallest degree 
among polynomials satisfying
$$
d_{ij}(z_{2}/z_{1}) R_{ij} \left(V_{i}[z_{1}^{\pm 1}] \otimes V_{j}[z_{2}^{\pm 1}]\right)
\subset 
1 \otimes
\left(V_{j}[z_{2}^{\pm 1}] \otimes V_{i}[z_{1}^{\pm 1}] \right).
$$

\begin{Rem}
\label{Rem:generalR}
In the same way, 
we can define the normalized $R$-matrix
$$
R_{M, M^{\prime}} \colon 
M[z_{1}^{\pm 1}] \otimes M^{\prime}[z_{2}^{\pm 1}]
\to \kk(z_{2}/z_{1}) \otimes_{\kk[(z_{2}/z_{1})^{\pm 1}]} 
\left( M^{\prime}[z_{2}^{\pm 1}] \otimes M[z_{1}^{\pm 1}] \right)
$$
and its denominator $d_{M, M^{\prime}}(u) \in \kk[u]$
for any simple modules $M, M^{\prime} \in \Cc$.
\end{Rem}

In the rest of this subsection, 
we recall some properties of the normalized $R$-matrices $R_{ij}$
and their denominators $d_{ij}(u)$ for future use.   
Let $a, b \in \kk^{\times}$ be non-zero scalars such that $d_{ij}(b/a) \neq 0$.
Then the normalized $R$-matrix $R_{ij}$ can be specialized to yield
a non-zero $U_{q}(L\g)$-homomorphism $R_{ij}(b/a) \colon V_{i}(a) \otimes V_{j}(b) \to V_{j}(b) \otimes V_{i}(a)$
which sends the vector $v_{i}(a) \otimes v_{j}(b)$ to the vector $v_{j}(b) \otimes v_{i}(a)$.

\begin{Thm}[\cite{AK97, Chari02, FM01, Kashiwara02, VV02}]
\label{Thm:tensor}
Let $i,j \in I$ and $a, b \in \kk^{\times}$.
\begin{enumerate}
\item \label{Thm:tensor1}
As a $U_{q}(L\g)$-module, $V_{i}(a) \otimes V_{j}(b)$ is generated
by the vector $v_{i}(a) \otimes v_{j}(b)$ if and only if $d_{ij}(b/a) \neq 0$. 
If this is the case, the module $V_{i}(a) \otimes V_{j}(b)$ has a simple head $\Ima(R_{ij}(b/a))$.
\item \label{Thm:tensor2}
Any non-zero $U_{q}(L\g)$-submodule of $V_{i}(a) \otimes V_{j}(b)$ 
contains the vector $v_{i}(a) \otimes v_{j}(b)$ if and only if $d_{ji}(a/b) \neq 0$.
If this is the case, the module $V_{i}(b)\otimes V_{i}(a)$ has a simple socle $\Ima(R_{ji}(a/b))$.
\end{enumerate}
In particular, the following conditions are mutually equivalent:
\begin{itemize}
\item The tensor product $V_{i}(a) \otimes V_{j}(b)$ is irreducible;
\item $V_{i}(a) \otimes V_{j}(b) \cong V_{j}(b) \otimes V_{i}(a)$ as $U_{q}(L\g)$-modules;
\item $d_{ij}(b/a) \neq 0$ and $d_{ji}(a/b) \neq 0$.
\end{itemize} 
\end{Thm}
\begin{proof}
This is a special case of Akasaka-Kashiwara's conjecture~\cite{AK97},
which was proved by Chari~\cite{Chari02}, Kashiwara~\cite{Kashiwara02} and 
Varagnolo-Vasserot~\cite{VV02} independently. 
The irreducibility of $ \Ima(R_{ij}(b/a))$ and $\Ima(R_{ji}(a/b))$
was proved in~\cite[Corollary 2.3]{AK97}. 
Note that Frenkel-Mukhin~\cite{FM01} also proved the last assertion.
\end{proof}

\begin{Thm}[Chari~\cite{Chari02}, Kashiwara~\cite{Kashiwara02}]
\label{Thm:CK}
Let $i,j \in I$ and $a \in \kk$.
If $d_{ij}(a)=0$, we have
$a \in \{ q^{k} \in \kk^{\times} \mid k + \epsilon_{i} + \epsilon_{j} \in 2 \Z, k > 0\}$.
\end{Thm}
\begin{proof}
Assume that $d_{ij}(a)=0$.
By \cite[Theorem 4.4]{Chari02} and Theorem~\ref{Thm:tensor}~(\ref{Thm:tensor1}) above,
we see that $a = q^{k}$ for some integer $k$ satisfying $k + \epsilon_{i} + \epsilon_{j} \in 2\Z$.
On the other hand, \cite[Proposition 9.3]{Kashiwara02} implies that $a \in \bigcup_{m \in \Z_{>0}} q^{1/m} \overline{\Q}[\![q^{1/m}]\!]$.
Therefore $k$ should be positive.  
\end{proof}

\begin{Rem}
In~\cite[Section 6]{Chari02}, Chari further computed 
all the zeros of $d_{ij}(u)$ by a type-by-type argument.
However, we do not use this fact in this paper.
\end{Rem}


\subsection{Main theorem} 
\label{Ssec:main}
Let $z$ be a formal parameter.
The {\it quantum Cartan matrix} $C(z) = (C_{ij}(z))_{i,j \in I}$ of $\g$ is defined by
$$
C_{ij}(z) \seq \begin{cases}
z+z^{-1} & (i=j);\\
c_{ij} & (i \neq j).
\end{cases}
$$ 
We regard $C(z)$ as an element of the group $GL_{n}( \Z(\!(z)\!) )$
and denote its inverse by $\tC(z) = (\tC_{ij}(z))_{i,j \in I}$.
The $(i,j)$-entry $\tC_{ij}(z) \in \Z(\!(z)\!)$ can be written as
$$
\tC_{ij}(z) = \sum_{\ell = 1}^{\infty} \tc_{ij}(\ell) z^{\ell}.
$$
In this way, we get a collection of integers $\{ \tc_{ij}(\ell) \}_{i,j \in I, \ell \ge 1}$.

Now we can state the main theorem of this paper.

\begin{Thm} 
\label{Thm:main}
For each pair $(i,j) \in I^{2}$, the denominator 
$d_{ij}(u) \in \kk[u]$ of the normalized $R$-matrix $R_{ij}$ is given by the following formula: 
\begin{equation} \label{Eq:denominator}
d_{ij}(u)=
\prod_{\ell=1}^{h-1}(u-q^{\ell+1})^{\tc_{ij}(\ell)},
\end{equation} 
where $h$ is the Coxeter number of $\g$. 
\end{Thm}

Theorem~\ref{Thm:main} is equivalent to Theorem~\ref{Thm:main2} below,
whose proof is given later in Section~\ref{Ssec:pf} 
using geometry of the graded quiver varieties. 

\begin{Rem}
The RHS of the formula \eqref{Eq:denominator} 
is actually a polynomial because we have
$\tc_{ij}(\ell) \in \Z_{\ge 0}$ for $1 \le \ell \le h -1$
by Lemma~\ref{Lem:tC}~\eqref{Lem:tC:positive} below.
\end{Rem}

\begin{Rem}
\label{Rem:convention2}
Note that our denominator $d_{ij}(u)$
is different from the denominator 
$d_{V(\varpi_{i}), V(\varpi_{j})}(u)$,
which has been written by the same symbol $d_{ij}(u)$  
in the works of Kashiwara and his collaborators 
(e.g.~\cite{AK97, Kashiwara02, KKK18}).
Here $V(\varpi_{i})$ denotes the $i$-th fundamental module 
in the sense of Kashiwara~\cite{Kashiwara94},
which has a global crystal basis. 
It was shown by Nakajima~\cite{Nakajima04}
that Kashiwara's fundamental module
$V(\varpi_{i})$ is isomorphic to our
fundamental module $V_{i}(a_{i})$
with $a_{i} \seq (-1)^{\epsilon_{i}}(-q)^{1-h}$.
Moreover, we see 
in Proposition~\ref{Prop:HL} below that $\tc_{ij}(\ell) \neq 0$ 
only if $\ell + \epsilon_i + \epsilon_j + 1 \in 2\Z$. 
Therefore,
our denominator formula~\eqref{Eq:denominator} is equivalent to  
the formula
\begin{equation}
\label{Eq:denominatorK}
d_{V(\varpi_{i}), V(\varpi_{j})}(u) = \prod_{\ell=1}^{h-1}(u-(-q)^{\ell+1})^{\tc_{ij}(\ell)}.
\end{equation}
By making the values $\tc_{ij}(\ell)$ explicit,
we can check that the formula \eqref{Eq:denominatorK} certainly recovers the
known type-by-type denominator formulas
obtained in~\cite{DO94, KKK18, OS19, OS19c}.
However we do not use this fact in this paper.
\end{Rem}


\section{An interpretation by quiver representations}
\label{Sec:quiver}

In this section, we give an interpretation
of our denominator formula~\eqref{Eq:denominator}
in terms of homological properties
of representations of a Dynkin quiver of type $\g$.
We keep the notation from the previous section.

\subsection{Convention}
\label{Ssec:ConventionQ}

First, we fix our convention on quivers and their representations. 
A quiver $Q=(Q_{0}, Q_{1})$ is an oriented graph, 
consisting of the set $Q_{0}$ of vertices
and the set $Q_{1}$ of arrows. 
Here the sets $Q_{0}$ and $Q_{1}$ can be infinite. 
For an arrow $a \in Q_{1}$,
let  $a^{\prime}, a^{\pprime} \in Q_{0}$ denote its origin and goal
respectively. 
We always assume that the set $\{ a \in Q_{1} \mid a^{\prime} = x, a^{\pprime} = y \}$
is finite for each $x, y \in Q_{0}$.

We equip the vector space $\C Q_{0} \seq \bigoplus_{x \in Q_{0}} \C e_{x}$ with a structure of 
$\C$-algebra by $e_{x} \cdot e_{y} = \delta_{xy} e_{x}$. 
This is non-unital if $Q_{0}$ is infinite.
We equip the vector space $\C Q_{1}\seq\bigoplus_{a \in Q_{1}}\C a$ 
with a structure of $\C Q_{0}$-bimodule by setting $a \cdot e_{x} = \delta_{a^{\prime}, x} a$ and 
$e_{x} \cdot a = \delta_{x, a^{\pprime}} a$ for $x \in Q_{0}, a \in Q_{1}$. 
The path algebra $\C Q$ of $Q$ is defined to be the tensor algebra 
$T_{\C Q_{0}}(\C Q_{1}) \seq \bigoplus_{d \ge 0} (\C Q_{1})^{\otimes d}$, 
where tensor products are taken over $\C Q_{0}$.
Given a quotient algebra $A = \C Q / \mathcal{I}$  by 
an ideal $\mathcal{I} \subset \bigoplus_{d \ge 1} (\C Q_{1})^{\otimes d}$,
we denote by $A \modcat$
the $\C$-linear abelian category of
finite-dimensional
left $A$-modules $M$
satisfying 
$M = \bigoplus_{x \in Q_{0}} e_{x}M$.
For each vertex $x \in Q_{0}$,
we denote by 
$S_{x}$ the simple object of $A \modcat$ 
associated with $x$, 
i.e.~satisfying $\dim (e_{y} S_{x}) = \delta_{xy}$.
For a finite-dimensional $Q_{0}$-graded $\C$-vector space $V = \bigoplus_{x \in Q_{0}}V_{x}$,
we denote by $\rep_{V}(A)$ the variety of representations of the algebra $A$
realized on $V$. By definition, this is the closed subvariety of the affine space
$\rep_{V}(Q) \seq \prod_{a \in Q_{1}} \Hom_{\C}(V_{a^{\prime}}, V_{a^{\pprime}})$  
consisting of points $(f_{a})_{a\in Q_{1}}$ such that all the polynomials in the linear maps $f_{a}$
corresponding to elements in $\mathcal{I}$ vanish.

\subsection{Dynkin quiver}
\label{Ssec:Dynkin}

In this subsection, we fix 
a Dynkin quiver $Q=(Q_{0}, Q_{1})$ of type $\g$, 
i.e.~$Q_{0} \seq I =\{ 1, \ldots n\}$ 
and the arrow set $Q_{1}$ satisfies the condition  
$c_{ij} = 2 \delta_{ij} - \# \{ a \in Q_{1} \mid \{ a^{\prime}, a^{\pprime} \} = \{ i,j\} \}$
for each $i,j \in I$.
We write $i \to j$ if there is an arrow 
$a \in Q_{1}$
such that $a^{\prime} = i, a^{\pprime} = j$.
For $M \in \C Q \modcat$, we define its dimension vector 
by $\vdim (M) \seq \sum_{i \in I} \dim(e_{i}M) \alpha_{i} \in \mathsf{Q}^{+}$. 

By Gabriel's theorem~\cite{Gabriel72}, 
for each $\alpha \in \cR^{+}$, 
there exists an indecomposable object
$M_\alpha \in \C Q \modcat$ such that $\vdim (M_\alpha) = \alpha$
uniquely up to isomorphism.
The correspondence
$\alpha \mapsto M_{\alpha}$
gives a bijection between the set $\cR^{+}$ 
of positive roots and
the set of isomorphism classes of indecomposable objects 
of $\C Q \modcat$.
In particular, we have $S_{i} = M_{\alpha_{i}}$ for each $i \in I$. 

Let $\dD_{Q}$ denote the bounded derived category 
$D^{b}(\C Q \modcat)$ of the abelian category $\C Q \modcat$. 
The category $\dD_{Q}$ is a $\C$-linear triangulated category 
with Krull-Schmidt property. 
The category $\C Q \modcat$ is naturally identified with a full subcategory of $\dD_{Q}$
consisting of complexes concentrated on the cohomological degree $0$. 
We denote by $X[k]$ the cohomological degree shift of $X \in \dD_{Q}$ by $k \in \Z$.
Then the set $\{ M_{\alpha}[k] \mid \alpha \in \mathsf{R}^{+}, k \in \Z\}$
forms a complete collection of indecomposable objects of $\dD_{Q}$ (see~\cite[Lemma 4.1]{Happel87}).
Extending the definition of $\vdim$,
for each $X \in \dD_{Q}$, we define
its dimension vector $\vdim (X) \in \cQ$ by 
$$
\vdim (X) \seq \sum_{k \in \Z} (-1)^{k} \vdim H^{k}(X),
$$
where $H^{k}(X) \in \C Q \modcat$ denotes the $k$-th cohomology of $X$.
For $X, Y \in \dD_{Q}$, we define the Euler-Poincar\'e characteristic $\langle X, Y \rangle \in \Z$ by
$$
\langle X, Y \rangle \seq \sum_{k \in \Z} (-1)^{k} \dim \Ext^{k}_{\dD_{Q}}(X, Y),
$$
where $\Ext_{\dD_{Q}}^{k}(X, Y) \seq \Hom_{\dD_{Q}}(X, Y[k])$. 

\subsection{Happel's equivalence}
\label{Ssec:Happel}

Let $Q$ be a Dynkin quiver of type $\g$.
In this subsection, 
we recall the description of 
the full subcategory $\ind(\dD_{Q}) \subset \dD_{Q}$ consisting 
of indecomposable objects in $\dD_{Q}$ 
due to Happel~\cite{Happel87, Happel88} . 

Let 
$\xi = (\xi_{i})_{i \in I} \in \Z^{I}$ be an $I$-tuple of integers 
such that $\xi_{i} - \epsilon_{i} \in 2\Z$ and 
$\xi_{i} = \xi_{j} +1$ if $i \to j$.
Such an $I$-tuple $\xi$ is called a {\em height function} of $Q$ and 
determined up to a simultaneous shift by an even integer.
Choose a total ordering $I=\{i_{1}, i_{2}, \ldots , i_{n} \}$
satisfying $\xi_{i_1} \ge \xi_{i_2} \ge \cdots \ge \xi_{i_n}$
and consider the Coxeter element 
$\tau \seq r_{i_1}r_{i_2} \cdots r_{i_n} \in \mathsf{W}$.
The element $\tau$
depends only on $Q$ (independent from the choice of the above total ordering of $I$).
By an abuse of notation, we use the same symbol $\tau$ for 
the corresponding Coxeter functor, which is an auto-equivalence of $\dD_{Q}$.
Under this convention, we have
$
\vdim(\tau X) = \tau \vdim(X)
$
for any $X \in \dD_{Q}$.
For an indecomposable object $X \in \ind(\dD_{Q})$, 
its Coxeter transformation $\tau X$ coincides with
the Auslander-Reiten translation of $X$ (see~\cite[Lemma VII.5.8]{ASS06} for example).

For each $i \in I$, we define a positive root $\gamma_{i}$ 
to be the sum of simple roots $\alpha_{j}$
labeled by the vertices $j$ such that there exists an oriented path
in $Q$ from $j$ to $i$. 
Then the corresponding indecomposable representation $I_{i} \seq M_{\gamma_{i}}$ 
is an injective hull of the simple representation $S_{i}$ in $\C Q \modcat$. 
Note that we have 
$
\langle X, I_{i} \rangle = (\vdim(X), \varpi_{i})
$
for any $X \in \dD_{Q}$ and $i \in I$.

\begin{Def}
\label{Def:Delta}
We define 
an infinite quiver 
$\Delta = (\Delta_{0}, \Delta_{1})$
by
\begin{align*}
\Delta_{0} &\seq \{(i,p) \in I \times \Z \mid p- \epsilon_{i} \in 2\Z \}, \\
\Delta_{1} &\seq \{(i,p) \to (j, p+1) \mid 
(i, p), (j, p+1) \in \Delta_{0}, \; i \sim j
\}.
\end{align*}
\end{Def}

Let $\C(\Delta)$ denote
the $\C$-linear category whose set of objects is $\Delta_{0}$
and whose morphisms are generated by $\Delta_{1}$ satisfying 
the so-called {\em mesh relations}, i.e.~the sum of all paths from $(i, p)$ to $(i, p+2)$ 
vanishes for each $(i,p) \in \Delta_{0}$.
Note that the quiver $\Delta$ and the category $\C(\Delta)$
are independent from the choice of the Dynkin quiver $Q$ (depends only on $\g$). 
 
\begin{Thm}[Happel~\cite{Happel87, Happel88}]
\label{Thm:Happel}
For a Dynkin quiver $Q$ of type $\g$ with a height function $\xi$, 
there is an equivalence of $\C$-linear categories
$$
\cH_{Q} \colon \C(\Delta) \simeq \ind(\dD_{Q})
$$
satisfying
$
\cH_{Q}(i, p) = \tau^{(\xi_{i} - p)/2}(I_{i})
$
for each $(i, p) \in \Delta_{0}$.
\end{Thm}
\begin{proof}
See \cite[Proposition 4.6]{Happel87} or \cite[Theorem 5.6]{Happel88}.
\end{proof}

\begin{Rem}
Although the equivalence $\cH_{Q}$ depends on the choice of the height function $\xi$,
this choice does not affect on the results in the present paper essentially and
hence we suppress it from the notation. 
In addition, 
the Euler-Poincar\'e characteristic 
$\langle \cH_{Q}(i,p), \cH_{Q}(j,r) \rangle$
does not depend on the choice of the Dynkin quiver $Q$
because,
for any two Dynkin quivers $Q$ and $Q^{\prime}$ of type $\g$,
we have a natural isomorphism
\begin{equation}
\label{Eq:QQ}
\Ext^{k}_{\dD_{Q}}(\cH_{Q}(i,p), \cH_{Q}(j,r)) 
\cong 
\Ext^{k}_{\dD_{Q^{\prime}}}(\cH_{Q^{\prime}}(i,p), \cH_{Q^{\prime}}(j,r))
\end{equation}
for any $(i,p), (j,r) \in \Delta_{0}$ and $k \in \Z$. 
\end{Rem}

\begin{Rem}
As explained in \cite[Section 6.5]{Gabriel80}, 
we have
\begin{equation}
\label{Eq:Nakayama}
\cH_{Q}(i, p)[1] = \cH_{Q}(i^{*}, p+h)
\end{equation}
for any $(i, p) \in \Delta_{0}$. 
Here $i \mapsto i^{*}$ is the involution on $I$ given by $\mathsf{w}_{0}(\alpha_{i}) = - \alpha_{i^{*}}$, 
where $\mathsf{w}_{0}$ denotes the longest element of the Weyl group $\mathsf{W}$.
\end{Rem}

\subsection{Quiver interpretation of quantum Cartan matrix}
\label{Ssec:tC}
In this subsection, we give an interpretation 
of the integers $\{\tc_{ij}(\ell)\}_{i,j \in I, \ell \ge 1}$ defined in Section~\ref{Ssec:main}
in terms of representations of a Dynkin quiver.    
Our discussion is based on
the following observation due to Hernandez-Leclerc. 

\begin{Prop}[Hernandez-Leclerc~\cite{HL15}]
\label{Prop:HL}
Take a Dynkin quiver $Q$ of type $\g$ together with a height function $\xi$
as in the previous subsection. 
Then, for any $i, j \in I$ and $\ell \in \Z_{\ge 1}$, we have
$$
\tc_{ij}(\ell) = 
\begin{cases}
\left(\tau^{(\ell + \xi_{i} - \xi_{j} -1)/2} (\gamma_{i}), \varpi_{j} \right) & \text{if $\ell + \epsilon_{i} + \epsilon_{j} +1 \in 2\Z$}; \\
0 & \text{otherwise}.
\end{cases}
$$
\end{Prop}
\begin{proof}
This is \cite[Proposition 2.1]{HL15}. 
Note that the condition $\ell + \epsilon_{i} + \epsilon_{j} +1 \in 2\Z$ here
is equivalent to the condition $\ell + \xi_{i} - \xi_{j} -1 \in 2\Z$ therein. 
\end{proof}

Thanks to Proposition~\ref{Prop:HL}, once we depict the Auslander-Reiten quiver of $\dD_{Q}$, 
we can easily compute the explicit values of the integers $\{\tc_{ij}(\ell)\}_{i,j \in I, \ell \ge 1}$.
See~\cite[Example 2.2]{HL15} for an example of such a computation.

\begin{Cor}
\label{Cor:HL}
For $(i, p), (j, r) \in \Delta_{0}$ with $r \ge p$, we have
$$
\left\langle \cH_{Q}(i, p), \cH_{Q}(j, r) \right\rangle = \tc_{ij}(r-p+1)
$$
for any Dynkin quiver $Q$ of type $\g$.
\end{Cor}
\begin{proof}
We compute as:
\begin{align*}
\left\langle \cH_{Q}(i, p), \cH_{Q}(j, r) \right\rangle &= 
\left\langle \tau^{(\xi_{i}-p)/2}(I_{i}), \tau^{(\xi_{j} -r)/2}(I_{j}) \right\rangle \\
&= \left\langle \tau^{((r-p+1) + \xi_{i} -\xi_{j} -1)/2}(I_{i}), I_{j} \right\rangle \\
&= \left( \tau^{((r-p+1) + \xi_{i} -\xi_{j} -1)/2}(\gamma_{i}), \varpi_{j} \right).
\end{align*}
Since $r-p+1 \ge 1$ by assumption, the RHS is equal to $\tc_{ij}(r-p+1)$ 
by Proposition~\ref{Prop:HL}.
\end{proof}

Here we record some basic properties of the integers $\{\tc_{ij}(\ell)\}_{i,j \in I, \ell \ge 1}$.

\begin{Lem}
\label{Lem:tC}
The integers $\{\tc_{ij}(\ell)\}_{i,j \in I, \ell \ge 1}$
satisfy the following properties:
\begin{enumerate}
\item \label{Lem:tC:symm} $\tc_{ij}(\ell) = \tc_{ji}(\ell);$
\item \label{Lem:tC:auto} $\tc_{ij}(\ell) = \tc_{\sigma(i), \sigma(j)}(\ell)$ for any automorphism $\sigma$ of the Dynkin diagram;
\item \label{Lem:tC:period} $\tc_{ij}(\ell) = \tc_{ij}(\ell + 2h);$
\item \label{Lem:tC:ARdual}$\tc_{ij}(\ell) = - \tc_{ij}(2 h - \ell)$ for $1 \le \ell \le 2 h-1;$
\item \label{Lem:tC:shift} $\tc_{ij}(\ell) = \tc_{ji^{*}}(h - \ell)$ for $1 \le \ell \le h -1;$
\item \label{Lem:tC:vanish} $\tc_{ij}(k h) = 0$ for any $k \in \Z_{\ge 1};$
\item \label{Lem:tC:positive} $\tc_{ij}(\ell) \ge 0$ if $1 \le \ell \le h-1;$
\item \label{Lem:tC:negative} $\tc_{ij}(\ell) \le 0$ if $h+1 \le \ell \le 2 h -1.$
\end{enumerate}
\end{Lem}
\begin{proof}
\eqref{Lem:tC:symm} and \eqref{Lem:tC:auto} are immediate from the definition. 

Let us take $Q$ and $\xi$ as in Proposition~\ref{Prop:HL}.
\eqref{Lem:tC:period} is a direct consequence of Proposition~\ref{Prop:HL} and the well-known fact $\tau^{h} = 1$.

To prove \eqref{Lem:tC:ARdual}, we may assume $\ell + \epsilon_{i} + \epsilon_{j} +1 \in 2\Z$.
Then we can pick $(i, p), (j, r) \in \Delta_{0}$ such that $\ell = r-p+1$.
By Corollary~\ref{Cor:HL} and \eqref{Eq:Nakayama}, we have
$$\tc_{ij}(\ell) = \langle \cH_{Q}(i, p), \cH_{Q}(j, r) \rangle 
= \langle \cH_{Q}(i, p), \cH_{Q}(j, r-2h) \rangle$$
for any Dynkin quiver $Q$. 
Using the Auslander-Reiten duality $\langle X, Y \rangle = - \langle Y, \tau X \rangle$, $X, Y \in \dD_{Q}$,
the RHS is further computed as:
\begin{align*}
\langle \cH_{Q}(i, p), \cH_{Q}(j, r-2h) \rangle
&= - \langle \cH_{Q}(j, r - 2h), \tau \cH_{Q}(i, p) \rangle \\
&= - \langle \cH_{Q}(j, r - 2h), \cH_{Q}(i, p-2) \rangle.
\end{align*}
Since $(p-2)-(r-2h) = 2 h - 1 -  \ell \ge 0$ by assumption, 
the RHS of the last equation is equal to $-\tc_{ij}(2 h - \ell)$ again by Corollary~\ref{Cor:HL}. 
This proves \eqref{Lem:tC:ARdual}. 

Let us prove \eqref{Lem:tC:shift}. 
As before, we may assume $\ell + \epsilon_{i} + \epsilon_{j} +1 \in 2\Z$.
For a Dynkin quiver $Q$ and $(i, p), (j, r) \in \Delta_{0}$ with $r-p+1 = \ell$, we have
$\tc_{ij}(\ell) = \langle \cH_{Q}(i, p), \cH_{Q}(j, r) \rangle$ 
by Corollary~\ref{Cor:HL}.
Using $\langle X, Y \rangle = - \langle Y, \tau X \rangle
= \langle Y, \tau X[1] \rangle
$ and \eqref{Eq:Nakayama}, we further compute as:
$$
\langle \cH_{Q}(i, p), \cH_{Q}(j, r) \rangle 
= \langle \cH_{Q}(j, r), \tau \cH_{Q}(i, p)[1] \rangle 
= \langle \cH_{Q}(j, r), \cH_{Q}(i^{*}, p + h - 2) \rangle.   
$$ 
Since $(p+h-2)-r = (h-1) - \ell \ge 0$ by assumption, 
we get $\langle \cH_{Q}(j, r), \cH_{Q}(i^{*}, p + h - 2) \rangle = \tc_{ji^{*}}(h - \ell)$
again by Corollary~\ref{Cor:HL}. This proves \eqref{Lem:tC:shift}.

To prove \eqref{Lem:tC:vanish}, it suffices to check 
that $\tc_{ij}(h) = \tc_{ij}(2 h) = 0$
thanks to \eqref{Lem:tC:period}.
Specializing $\ell = h$ in \eqref{Lem:tC:ARdual}, 
we get $\tc_{ij}(h) = - \tc_{ij}(h)$ and hence $\tc_{ij}(h)=0$.
Let us verify $\tc_{ij}(2 h)=0$.    
When $\epsilon_{i} = \epsilon_{j}$, the number 
$2h + \epsilon_{i} + \epsilon_{j} +1$ is always odd. 
Therefore we have $\tc_{ij}(2 h)=0$ by Proposition~\ref{Prop:HL}.
When $\epsilon_{i} \neq \epsilon_{j}$,
let us choose $Q$ with the {\em sink-source} orientation 
such that $i$ is a source and $j$ is a sink. 
Namely, we choose $Q$ and its height function $\xi$ 
so that we have $\xi_{i} = \xi_{j} + 1$, and 
$\xi_{k} = \xi_{i}$ (resp.~$\xi_{k} = \xi_{j}$) 
if $\epsilon_{k} = \epsilon_{i}$ (resp.~$\epsilon_{k} = \epsilon_{j}$).
With such a choice, we have $\gamma_{i} = \alpha_{i}$ and 
$2h + \xi_{i} - \xi_{j} -1 = 2 h$.
Therefore, by Proposition~\ref{Prop:HL}, we get
$$
\tc_{ij}(2 h) = 
( 
\tau^{2 h/2}(\gamma_{i}), \varpi_{j}
)
=
(\alpha_{i}, \varpi_{j}) = 0.
$$

Let us prove \eqref{Lem:tC:positive}.
Assume $\tc_{ij}(\ell) \neq 0$. 
First we consider the case $\epsilon_{i} = \epsilon_{j}$. 
In this case, Proposition~\ref{Prop:HL}
implies that $\ell$ is odd. 
Let us take a Dynkin quiver $Q$ with a sink-source orientation.
In particular, we have $\xi_{i} = \xi_{j}$.
By the description of the Auslander-Reiten quiver of $\dD_{Q}$ 
in \cite[Proposition 6.5]{Gabriel80}, we see that 
the set
$$
\{ \tau^{(\ell - 1)/2}(\gamma_{i}) \mid 1 \le \ell \le h-1, \; \text{$\ell$ is odd} \}
= \{ \gamma_{i}, \tau(\gamma_{i}), \ldots, \tau^{\lfloor (h-2)/2 \rfloor}(\gamma_{i}) \}
$$ 
consists of positive roots. Therefore we have
$\tc_{ij}(\ell) = \left( \tau^{(\ell - 1)/2}(\gamma_{i}), \varpi_{j} \right) \ge 0$
for any $1 \le \ell \le h-1$.
For the other case $\epsilon_{i} \neq \epsilon_{j}$, 
Proposition~\ref{Prop:HL}
implies that $\ell$ is even. 
Let us take $Q$ with a sink-source orientation
with $i$ being a sink. Then $j$ is a source 
and hence $\ell + \xi_{i} - \xi_{j} -1 = \ell -2$.
Again we can see that the set
$$
\{ \tau^{(\ell - 2)/2}(\gamma_{i}) \mid 1 \le \ell \le h-1, \; \text{$\ell$ is even} \}
= \{ \gamma_{i}, \tau(\gamma_{i}), \ldots, \tau^{\lfloor (h-3)/2 \rfloor}(\gamma_{i}) \}
$$ 
consists of positive roots.
Therefore we get
$\tc_{ij}(\ell) = \left( \tau^{(\ell - 2)/2}(\gamma_{i}), \varpi_{j} \right) \ge 0$.

The last item \eqref{Lem:tC:negative} follows from \eqref{Lem:tC:ARdual} and \eqref{Lem:tC:positive}.
\end{proof}

Thanks to the above lemma,
we can recover all the integers $\{ \tc_{ij}(\ell) \mid \ell \ge 1\}$ for each $(i,j) \in I^{2}$
from the first $h-1$ integers $\{ \tc_{ij}(\ell) \mid 1 \le \ell \le h-1\}$,
for which we have the following simple representation-theoretic interpretation.

\begin{Prop} 
\label{Prop:Ext=tC}
Let $Q$ be a Dynkin quiver of type $\g$.
For any $(j, r), (i, p) \in \Delta_{0}$, we have
$$
\dim \Ext_{\dD_{Q}}^{1}(\cH_{Q}(j, r), \cH_{Q}(i,p)) = 
\begin{cases}
\tc_{ij}(r-p-1) & \text{if $1 \le r - p -1 \le h-1$}; \\
0 & \text{otherwise}.
\end{cases}
$$
\end{Prop}
\begin{proof}
Using \eqref{Eq:Nakayama}, we have 
\begin{align} \label{Eq:Ext=tC1}
\dim \Ext_{\dD_{Q}}^{1}(\cH_{Q}(j, r), \cH_{Q}(i, p))
&= \dim \Hom_{\dD_{Q}}(\cH_{Q}(j, r), \cH_{Q}(i, p)[1]) \\
&= \dim \Hom_{\dD_{Q}}(\cH_{Q}(j, r), \cH_{Q}(i^{*}, p+h)). \nonumber
\end{align}
On the other hand, using the Auslander-Reiten duality, we have
\begin{align} \label{Eq:Ext=tC2}
\dim \Ext_{\dD_{Q}}^{1}(\cH_{Q}(j, r), \cH_{Q}(i, p))
&= \dim \Hom_{\dD_{Q}}(\cH_{Q}(i, p), \tau \cH_{Q}(j, r))  \\
&= \dim \Hom_{\dD_{Q}}(\cH_{Q}(i, p), \cH_{Q}(j, r-2)). \nonumber
\end{align}
Now we assume that $\Ext_{\dD_{Q}}^{1}(\cH_{Q}(j, r), \cH_{Q}(i, p)) \neq 0$.
Then 
we have
$$
\Hom_{\dD_{Q}}(\cH_{Q}(j, r), \cH_{Q}(i^{*}, p+h)) \neq 0, \quad
\Hom_{\dD_{Q}}(\cH_{Q}(i, p), \cH_{Q}(j, r-2)) \neq 0
$$
by the above equations \eqref{Eq:Ext=tC1} and \eqref{Eq:Ext=tC2} respectively.
In view of Theorem~\ref{Thm:Happel}, 
they imply that $r \le p+h$ and $p \le r-2$,
or equivalently $1 \le r-p-1 \le h-1$.

Conversely, let us assume $1 \le r-p-1 \le h-1$. 
We continue \eqref{Eq:Ext=tC2} as:
\begin{align*}
\dim \Hom_{\dD_{Q}}(\cH_{Q}(i, p), \cH_{Q}(j, r-2)) 
&= \dim \Hom_{\dD_{Q}}(\tau^{(\xi_{i} - p)/2}(I_{i}), \tau^{(\xi_{j} - r + 2)/2}(I_{j})) \\
&= \dim \Hom_{\dD_{Q}}(\tau^{((r-p-1)+\xi_{i}-\xi_{j}-1)/2}(I_{i}), I_{j}). 
\end{align*}
Because of \eqref{Eq:QQ}, 
we may assume that our Dynkin quiver $Q$ has the sink-source orientation with 
the vertex $i$ being a source.
Then the object $\tau^{((r-p-1)+\xi_{i}-\xi_{j}-1)/2}(I_{i})$
remains inside $\C Q \modcat \subset \dD_{Q}$.
Therefore we have
$$
\dim \Hom_{\dD_{Q}}(\tau^{((r-p-1)+\xi_{i}-\xi_{j}-1)/2}(I_{i}), I_{j})
= \left\langle \tau^{((r-p-1)+\xi_{i}-\xi_{j}-1)/2}(I_{i}), I_{j} \right\rangle.
$$
The RHS is equal to $\tc_{ij}(r-p-1)$ by Corollary~\ref{Cor:HL}. 
\end{proof}

\subsection{Quiver interpretation of the denominator formula}
\label{Ssec:qinterpret}

Thanks to Theorem~\ref{Thm:CK} and Proposition~\ref{Prop:Ext=tC}, we see that
the following assertion is equivalent to our main theorem (=Theorem~\ref{Thm:main}).

\begin{Thm}
\label{Thm:main2}
Let $Q$ be a Dynkin quiver of type $\g$.
For any $(i, p), (j, r) \in \Delta_{0}$, 
the pole order of the normalized $R$-matrix $R_{ij}$ 
at $z_{2}/z_{1} = q^{r}/q^{p}$
(i.e.~the zero order of $d_{ij}(u)$
at $u = q^{r}/q^{p}$)
 is equal to 
$
\dim \Ext_{\dD_{Q}}^{1}(\cH_{Q}(j,r), \cH_{Q}(i,p)).
$
\end{Thm}

A proof of Theorem~\ref{Thm:main2} is given in Section~\ref{Ssec:pf} below. 

In what follows,
for each $(i,p) \in I \times \Z$,
we simplify the notation by setting
$$Y_{(i,p)} \seq Y_{i, q^{p}}.$$  
Recall that we have an identification $L(Y_{(i,p)}) = V_{i}(q^{p})$
for each $(i,p) \in I \times \Z$. 

\begin{Cor}
\label{Cor:main2}
Let $Q$ be a Dynkin quiver of type $\g$. 
For any $x, y \in \Delta_{0}$, the following conditions are mutually equivalent:
\begin{itemize}
\item The tensor product $L(Y_{x}) \otimes L(Y_{y})$ is irreducible;
\item $L(Y_{x}) \otimes L(Y_{y}) \cong L(Y_{y}) \otimes L(Y_{x})$ as $U_{q}(L\g)$-modules;
\item $\Ext_{\dD_{Q}}^{1}(\cH_{Q}(x), \cH_{Q}(y)) = 0$ and 
$\Ext_{\dD_{Q}}^{1}(\cH_{Q}(y), \cH_{Q}(x)) = 0$.
\end{itemize}
\end{Cor}
\begin{proof}
This follows from Theorem~\ref{Thm:tensor} and Theorem~\ref{Thm:main2}.
\end{proof}

\section{Graded quiver varieties}
\label{Sec:Nakajima}

In this section, we collect some known facts 
about the graded quiver varieties 
which we need in this paper.
We keep the notation from the previous sections.

\subsection{Notation on the equivariant $K$-theory}
\label{Ssec:Ktheory}
Let $G$ be a complex linear algebraic group.
In the present paper, 
a $G$-variety $X$ always means a quasi-projective 
complex algebraic variety
equipped with an algebraic action of the group $G$.  
We set $\mathrm{pt} \seq \mathop{\mathrm{Spec}} \C$ 
with the trivial $G$-action.
The equivariant $K$-group $K^{G}(X)$ is 
defined to be the Grothendieck group 
of the abelian category
of $G$-equivariant coherent sheaves on $X$,
which is a module over the
representation ring
$R(G)=K^{G}(\mathrm{pt})$.

Let $\mathbb{F}$ be a field of characteristic zero.
We put
$$K^{G}(X)_{\mathbb{F}} \seq K^{G}(X)\otimes_{\Z}\mathbb{F}, \quad
R(G)_{\mathbb{F}} \seq R(G)\otimes_{\Z}\mathbb{F}.$$
Let $\mathfrak{a} \subset R(G)_{\mathbb{F}}$ be the augmentation ideal,
i.e.~the ideal 
consisting of
virtual representations of dimension $0$.
We define the $\mathfrak{a}$-adic completions by
$$
\hK^{G}(X)_{\mathbb{F}} \seq 
\varprojlim_{k} K^{G}(X)_{\mathbb{F}}/\mathfrak{a}^{k} K^{G}(X)_{\mathbb{F}}, 
\quad
\hR(G)_{\mathbb{F}}\seq \varprojlim_{k} R(G)_{\mathbb{F}} / \mathfrak{a}^{k}.
$$ 
The completed 
$K$-group $\hK^{G}(X)_{\mathbb{F}}$
is a module over the algebra $\hR(G)_{\mathbb{F}}$.

\subsection{Convolution product}
\label{Ssec:convolution}

We recall the definition of the {\em convolution product}
for the equivariant $K$-groups (see~\cite[Chapter 5]{CG97} and 
\cite[Section 6, 8]{Nakajima01} for details).
Let $M_i$ be a non-singular $G$-variety for $i=1,2,3$.
We denote by $p_{ij} \colon M_1 \times M_2 \times M_3
\to M_i \times M_j$ the natural projection
for $(i,j) = (1,2), (2,3), (1,3)$.
Let $Z_{12} \subset M_{1} \times M_{2}$ and
$Z_{23} \subset M_{2} \times M_{3}$ be 
$G$-stable closed subvarieties such that
the morphism 
$
p_{13} \colon p_{12}^{-1}(Z_{12}) \cap p_{23}^{-1}(Z_{23}) 
\to Z_{13} 
\seq p_{13}(p_{12}^{-1}(Z_{12}) \cap p_{23}^{-1}(Z_{23}) )
$
is proper.
Then we define the convolution product
$
* \colon K^{G}(Z_{12}) \otimes_{R(G)} K^{G}(Z_{23}) 
\to K^{G}(Z_{13})
$
relative to $M_{1}, M_{2}, M_{3}$ by
\begin{equation}
\label{Eq:convolution}
\zeta * \eta \seq p_{13*} (p_{12}^{*}\zeta
\otimes^{\mathbb{L}}_{M_{1} \times M_{2} \times M_{3}} p_{23}^{*}\eta ),
\end{equation}
where $\zeta \in K^{G}(Z_{12}), \eta \in K^{G}(Z_{23})$.
This naturally induces the convolution product
on the completed $G$-equivariant $K$-groups
$\hK^{G}(Z_{12})_{\mathbb{F}} \otimes_{\hR(G)_{\mathbb{F}}}
\hK^{G}(Z_{23})_{\mathbb{F}} \to \hK^{G}(Z_{13})_{\mathbb{F}}
$.
Note that the convolution product $*$ depends on the ambient smooth spaces $M_{1}, M_{2}, M_{3}$.

\subsection{Quiver varieties}
\label{Ssec:QV}
In this subsection, 
we recall the definition of the (usual) Nakajima quiver varieties.
A basic reference is~\cite{Nakajima01}.
  
We fix $I$-graded finite-dimensional complex vector spaces
$
\bar{V} = \bigoplus_{i \in I} \bar{V}_{i},
\bar{W} = \bigoplus_{i \in I} \bar{W}_{i}
$
and consider the following space of linear maps:
$$
\mathbf{M}(\bar{V}, \bar{W}) \seq
\left(
\bigoplus_{i \sim j} \Hom (\bar{V}_{i}, \bar{V}_{j}) \right)
\oplus
\left( \bigoplus_{i \in I} \Hom (\bar{W}_{i}, \bar{V}_{i}) \right)
\oplus
\left(
\bigoplus_{i \in I} \Hom (\bar{V}_{i}, \bar{W}_{i}) \right)
$$
On the $\C$-vector space $\mathbf{M}(\bar{V}, \bar{W})$,
the groups
$G_{\bar{V}} \seq \prod_{i \in I} \mathop{GL}(\bar{V}_{i})$,
$G_{\bar{W}} \seq \prod_{i \in I} \mathop{GL}(\bar{W}_{i})$ 
act by conjugation and the $1$-dimensional torus
$\C^{\times}$ acts by the scalar multiplication. 
We write an element of $\mathbf{M}(\bar{V}, \bar{W})$ 
as a triple
$(B,a,b)$ of linear maps $B = \bigoplus B_{ij}$,
$a = \bigoplus a_{i}$ and $b = \bigoplus b_{i}$.
Let $\mu = \bigoplus_{i \in I} \mu_{i} 
\colon \mathbf{M}(\bar{V}, \bar{W}) 
\to \bigoplus_{i \in I} \mathfrak{gl}(\bar{V}_{i})$
be the map given by 
\begin{equation}
\label{Eq:moment}
\mu_{i}(B, a, b)
= a_{i} b_{i}
+\sum_{j \sim i} B_{ij}B_{ji}.
\nonumber
\end{equation}
A point $(B, a, b) \in \mu^{-1}(0)$ is said to be stable
if there exists no non-zero $I$-graded subspace 
$\bar{V}^{\prime} \subset \bar{V}$ such that 
$B(\bar{V}^{\prime}) \subset \bar{V}^{\prime}$ and 
$\bar{V}^{\prime} \subset \Ker b$.
Let $\mu^{-1}(0)^{\mathrm{st}}$ be 
the set of stable points, on which $G_{\bar{V}}$ acts freely. 
Then we consider a set-theoretic quotient
$$\M(\bar{V}, \bar{W}) \seq \mu^{-1}(0)^{\mathrm{st}} / G_{\bar{V}}.$$
It is known that this quotient has a structure of
a non-singular quasi-projective variety which 
is isomorphic to a quotient in the geometric invariant theory.
We also consider the categorical
quotient
$$\M_{0}(\bar{V}, \bar{W}) \seq \mu^{-1}(0)/\!/G_{\bar{V}}
= \mathrm{Spec}\, \C[\mu^{-1}(0)]^{G_{\bar{V}}},$$
together with a canonical projective morphism
$\M(\bar{V}, \bar{W}) \to \M_{0}(\bar{V}, \bar{W})$.
These quotients $\M(\bar{V}, \bar{W})$, $\M_{0}(\bar{V}, \bar{W})$
naturally inherit the actions of the group 
$\G_{\bar{W}} \seq G_{\bar{W}} \times \C^{\times}$,
which makes the canonical projective morphism
$\G_{\bar{W}}$-equivariant.  

For any two $I$-graded vector spaces $\bar{V}, \bar{V}^{\prime}$ such that
$\dim \bar{V}_{i} \le \dim \bar{V}^{\prime}_{i}$ for each $i \in I$,  
there is a natural closed embedding
$\M_{0} (\bar{V}, \bar{W})
\hookrightarrow 
\M_{0}(\bar{V}^{\prime}, \bar{W}).
$
With respect to these embeddings, 
the family $\{ \M_{0}(\bar{V}, \bar{W})\}_{\bar{V}}$
forms an inductive system, 
which stabilizes at some large $\bar{V}$.
We consider the union (inductive limit) and 
obtain the following combined 
$\G_{\bar{W}}$-equivariant morphism: 
$$
\pi \colon \M(\bar{W}) 
\seq \bigsqcup_{\bar{V}} \M(\bar{V}, \bar{W})
\to 
\M_{0}(\bar{W}) 
\seq \bigcup_{\bar{V}} \M_{0}(\bar{V}, \bar{W}).
$$  
We refer to these varieties $\M(\bar{W}), \M_{0}(\bar{W})$
as the {\em quiver varieties}.
Note that the component $\M(0, \bar{W})$
consists of a single point,
which we denote by $\hat{0}$.
We call its image $\pi(\hat{0}) = 0$ the origin of $\M_{0}(\bar{W})$.  

\subsection{Graded quiver varieties}
\label{Ssec:GQV}

Next we define the graded quiver varieties. 
Recall the infinite set
$
\Delta_{0} = \{(i,p) \in I \times \Z \mid p- \epsilon_{i} \in 2\Z \}
$
in Definition~\ref{Def:Delta}.
We fix a $\Delta_{0}$-graded finite-dimensional complex vector space
$W = \bigoplus_{(i,p) \in \Delta_{0}} W_{(i,p)}$.
Let $\bar{W} = \bigoplus_{i \in I} \bar{W}_{i}$ 
be the underlying $I$-graded vector space of $W$, 
i.e.~$\bar{W}_{i} \seq \bigoplus_{p} W_{(i,p)}$ for each $i \in I$. 
We define a $1$-dimensional algebraic subtorus $T_{W} \subset \G_{\bar{W}}$
by
\begin{equation}
\label{Eq:defT}
T_{W} \seq \left\{ \left(\mathop{\oplus}_{(i, p) \in \Delta_{0}} t^{p} \mathrm{id}_{W_{(i,p)}}, t \right) \in \G_{\bar{W}}
\; \middle| \;
t \in \C^{\times} \right\}.
\end{equation}
Note that the centralizer of $T_{W}$ inside $\G_{\bar{W}}$
is $\G_{W} \seq G_{W} \times \C^{\times}$, where 
$G_{W} \seq \prod_{(i,p) \in \Delta_{0}} GL(W_{(i,p)}) \subset G_{\bar{W}}$.
We consider the $T_{W}$-fixed loci:
$$
\pi^{\bullet} \seq \pi^{T_{W}} \colon \Mg(W) \seq \M(\bar{W})^{T_{W}} 
\to \Mg_{0}(W) \seq \M_{0}(\bar{W})^{T_{W}},
$$
and refer to these varieties 
$\Mg(W), \Mg_{0}(W)$ as the {\em graded quiver varieties}.
The centralizer $\G_{W}$
 naturally acts 
on the varieties $\Mg(W)$, $\Mg_{0}(W)$
and the proper morphism $\pi^{\bullet}$ is $\G_{W}$-equivariant. 
\subsection{Nakajima's homomorphism}
\label{Ssec:Nakajimahom}

Take a finite-dimensional $I$-graded $\C$-vector space $\bar{W}$
and consider the quiver varieties $\pi \colon \M(\bar{W}) \to \M_{0}(\bar{W})$.
We define
$$Z(\bar{W}) \seq \M(\bar{W}) \times_{\M_{0}(\bar{W})} \M(\bar{W}), 
\quad \mathfrak{L}(\bar{W}) \seq \pi^{-1}(0) = \M(\bar{W}) \times_{\M_{0}(\bar{W})} \{ 0 \}.$$
Applying the convolution construction in Section~\ref{Ssec:convolution},
we obtain an $R(\G_{\bar{W}})$-algebra 
$K^{\G_{\bar{W}}}(
Z(\bar{W}))$ and
a left module $K^{\G_{\bar{W}}}(\mathfrak{L}(\bar{W}))$ over it.

We set $A \seq R(\C^{\times})$ and 
regard $K^{\G_{\bar{W}}}\!(-)$ as an $A$-module
via the inclusion $A=R(\C^{\times}) \hookrightarrow R(\G_{\bar{W}})$
arising from the second projection 
$\G_{\bar{W}}=G_{\bar{W}} \times \C^{\times} \to \C^{\times}$.
Also, we  
regard the field $\kk=\overline{\Q(q)}$ as an $A$-algebra
via the homomorphism $A=R(\C^{\times}) \to \kk$ sending 
the class of the $1$-dimensional natural $\C^{\times}$-module to the parameter $q$.
After the base change, 
we obtain the
$\kk$-algebra
$K^{\G_{\bar{W}}}(
Z(\bar{W})) \otimes_{A} \kk$
and the left module $K^{\G_{\bar{W}}}(\mathfrak{L}(\bar{W}))\otimes_{A}\kk$ over it.

\begin{Thm}[Nakajima~\cite{Nakajima01}]
\label{Thm:Nakajima}
There exists a $\kk$-algebra homomorphism
$$
\Psi_{\bar{W}} \colon U_{q}(L\g) \to K^{\G_{\bar{W}}}(Z(\bar{W})) \otimes_{A} \kk
$$
via which the $K^{\G_{\bar{W}}}(Z(\bar{W})) \otimes_{A}\kk$-module $K^{\G(\bar{W})}(\mathfrak{L}(\bar{W}))\otimes_{A}\kk$
is regarded as a $U_{q}(L\g)$-module 
and isomorphic to the level-zero extremal weight module of extremal weight $\sum_{i \in I} (\dim \bar{W}_{i}) \varpi_{i}$
in the sense of Kashiwara ~\cite{Kashiwara94}.
\end{Thm}
\begin{proof}
See~\cite[Section 9]{Nakajima01} and \cite[Theorem 2]{Nakajima04}.
\end{proof}

Let us describe some more details in the special case when $\bar{W} = \bar{W}_{i}$
and $\dim \bar{W}_{i} = 1$ for some $i \in I$.
In this case, the extremal weight module of extremal weight $\varpi_{i}$
is isomorphic to the affinized fundamental module $V_{i}[z^{\pm1}]$.
By Theorem~\ref{Thm:Nakajima}, we have a $U_{q}(L\g)$-isomorphism
\begin{equation}
\label{Eq:greaVi}
K^{\G_{\bar{W}}}(\mathfrak{L}(\bar{W}))\otimes_{A}\kk
\cong V_{i}[z^{\pm 1}].
\end{equation}
Under this isomorphism, the vector $v_{i} \in V_{i}[z^{\pm 1}]$
corresponds to the class $[\Oo_{\{ \hat{0} \}}]$
of the structure sheaf of $\M(0, \bar{W}) = \{ \hat{0} \}$.  
The action of $R(\G_{\bar{W}})_{\kk}$ on the LHS is identified with
the action of $\kk[z^{\pm 1}]$ on the RHS 
via the isomorphism $\psi_{i} \colon R(\G_{\bar{W}})\otimes_{A}{\kk} \to \kk[z^{\pm 1}]$
which sends the class $[\bar{W}_{i}]$ of the natural representation 
given by the first projection $\G_{\bar{W}} = GL(\bar{W}_{i}) \times \C^{\times} \to GL(\bar{W}_{i})$
to the formal parameter $z$. 

\subsection{Completion}
\label{Ssec:completion}

As in Section~\ref{Ssec:GQV},
we fix
a finite-dimensional $\Delta_{0}$-graded vector space
$W = \bigoplus_{x \in\Delta_{0}} W_{x}$
and denote by
$\bar{W} = \bigoplus_{i \in I} \bar{W}_{i}$ its
underlying $I$-graded vector space.
Recall the $1$-dimensional subtorus $T_{W} \subset \G_{\bar{W}}$ 
and its centralizer $\G_{W} = G_{W} \times \C^{\times} \subset \G_{\bar{W}}$.
Note that the multiplication induces a group isomorphism
\begin{equation}
\label{Eq:GT}
G_{W} \times T_{W} \cong \G_{W}, \quad (g,t) \mapsto gt.
\end{equation}  
Let $\rr_{W}$ be 
the kernel of the restriction
$R(\G_{W})\otimes_{A} \kk
\to R(T_{W}) \otimes_{A}\kk = \kk$.
Note that the decomposition \eqref{Eq:GT}
yields an isomorphism 
$$
K^{\G_{W}}(X)\otimes_{A} \kk 
\cong K^{G_{W}}(X)_{\kk}
$$
for any $\G_{W}$-variety $X$ with
a trivial $T_W$-action.
In particular, we have an isomorphism 
$R(\G_{W})\otimes_{A} \kk
\cong R(G_{W})_{\kk}$ of $\kk$-algebras,
via which the maximal ideal $\rr_{W}
\subset R(\G_{W})\otimes_{A} \kk$
corresponds to the augmentation ideal $\mathfrak{a} \subset 
R(G_{W})_{\kk}$.  
Therefore we have an isomorphism 
\begin{equation}
\label{Eq:complG}
\left[
K^{\G_{W}}(X)\otimes_{A} \kk 
\right]_{\rr_{W}}^{\wedge}
\cong 
\hK^{G_{W}}(X)_{\kk},
\end{equation}
where $[ - ]_{\rr_{W}}^{\wedge}$
denotes the $\rr_{W}$-adic completion. 

We consider the fiber product
$$\Zg(W) \seq \Mg(W) \times_{\Mg_{0}(W)} \Mg(W) 
= Z(\bar{W})^{T_{W}}.$$
The completed equivariant $K$-group $\hK^{G_{W}}(\Zg(W))_{\kk}$ 
becomes a $\kk$-algebra via the convolution product.
We define the $\kk$-algebra homomorphism
$\widehat{\Psi}_{W} \colon U_{q}(L\g) \to \hK^{G_{W}}
(Z^{\bullet}(W))_{\kk}$ to be the following composition:
\begin{align*}
U_{q}(L\g) 
& \to
K^{\G_{\bar{W}}}(Z(\bar{W})) \otimes_{A} \kk
&& \text{(Nakajima's homomorphism $\Psi_{\bar{W}}$)}
\\
&\to
K^{\G_{W}}(Z(\bar{W})) \otimes_{A} \kk
&& \text{(restriction to $\G_{W} \subset \G_{\bar{W}}$)}
\\
& \to
\left[
K^{\G_{W}}(Z(\bar{W}))\otimes_{A} \kk 
\right]_{\rr_{W}}^{\wedge}
&&
\text{($\rr_{W}$-adic completion)}
\\
&\cong
\left[
K^{\G_{W}}(\Zg(W))\otimes_{A} \kk 
\right]_{\rr_{W}}^{\wedge}
&&
\text{(localization theorem)}
\\
&\cong 
\hK^{G_{W}}(\Zg(W))_{\kk}.
&& \text{(isomorphism \eqref{Eq:complG})}
\end{align*}
We refer to the homomorphism $\widehat{\Psi}_{W}$ 
as the {\em completed Nakajima homomorphism}.

Let us describe the special case when $W = W_{x}$ and $\dim W_{x} = 1$ for some $x = (i,p) \in \Delta_{0}$. 
In this case, we have $\Mg_{0}(W) = \{0\}$ and hence $\Mg(W)= \mathfrak{L}(\bar{W})^{T_W}$.
Note that the composition of
$R(G_{W})_{\kk} 
\cong R(\G_{W})\otimes_{A} \kk = R(\G_{\bar{W}})\otimes_{A} \kk$
arising from \eqref{Eq:GT}
and $\psi_{i}$
in the previous subsection 
yields an isomorphism 
$\psi_x \colon
R(G_{W})_{\kk} \cong \kk[z^{\pm 1}]. 
$
Since the group homomorphism $\G_W \cong G_W \times T_W \to G_W$ obtained by composing the inverse of \eqref{Eq:GT} and the natural projection is given by $(g,t) \mapsto gt^{-p}$ for $(g,t) \in \G_W = (\C^\times)^2$, the isomorphism $\psi_x$ sends the class $[W_{x}]$ of 
the natural representation of $G_{W}$
to the element $q^{-p}z$.
After the completion, we get an isomorphism
$\widehat{\psi}_x \colon \widehat{R}(G_{W})_{\kk} \cong \kk[\![ z-q^{p} ]\!]$.  
In the sequel, we identify them via $\widehat{\psi}_x$.
By completing the isomorphism \eqref{Eq:greaVi} $\rr_{W}$-adically,
we obtain the following.

\begin{Lem}
\label{Lem:greaVx}
We have an isomorphism of $U_{q}(L\g) \otimes_{\kk} \kk[\![ z-q^{p} ]\!]$-modules
$$
\hK^{G_{W}}(\Mg(W))_{\kk} \cong V_{i}[z^{\pm 1}] \otimes_{\kk[z^{\pm 1}]} \kk[\![ z-q^{p} ]\!],
$$
under which the vector $v_{i} \otimes 1$ in the RHS corresponds 
to the class $[\Oo_{\{ \hat{0} \}}]$ of the structure sheaf of $\{ \hat{0} \}$
in the LHS.  
\end{Lem}
\subsection{Keller-Scherotzke's theorem}
\label{Ssec:KS}
In this subsection, 
we recall a description of the affine graded quiver variety $\Mg_{0}(W)$
due to Keller-Scherotzke~\cite{KS16},
which plays a crucial role in this paper.
Recall the notation on quivers 
in Section~\ref{Ssec:ConventionQ}.  

\begin{Def}
We define an infinite quiver $\tDel=(\tDel_{0}, \tDel_{1})$ whose
set of vertex is $\tDel_{0} \seq I \times \Z$. 
Let $\Delta_{0}^{+}$ denote the complement of 
the subset $\Delta_{0}$ of $\tDel_{0}$:
$$
\tDel_{0} = \Delta_{0} \sqcup \Delta_{0}\!\!{}^{+}, \quad
\Delta_{0}^{+} = \{ (i,p+1) \in I \times \Z \mid (i,p) \in \Delta_{0} \}.
$$
The arrow set $\tDel_{1}$ consists of the following three kinds of arrows:
\begin{itemize}
\item
${a}_{i}(p) \colon (i, p) \to (i,p-1)$ for each $(i,p) \in \Delta_{0}$;
\item 
${b}_{i}(p) \colon (i, p) \to (i, p-1)$ for each $(i,p) \in \Delta_{0}^{+}$;
\item   
${B}_{ji}(p) \colon (i, p) \to (j, p-1)$ for each $(i,p) \in \Delta_{0}^{+}$ and $j \in I$ with $j \sim i$. 
\end{itemize}
Let $\mathfrak{I}$ be a two sided ideal of the path algebra $\C\tDel$
generated by the elements
$$
{a}_{i}(p-1) {b}_{i}(p)
+ \sum_{j \sim i}
{B}_{ij}(p-1) {B}_{ji}(p)  
$$
for $(i,p) \in \Delta_{0}^{+}$.
Then, we define the (non-unital) $\C$-algebras $\tLam$ and $\Lambda$ by 
$$
\tLam \seq \C \tDel / \mathfrak{I},\quad
\Lambda \seq \bigoplus_{x, y \in \Delta_{0}} e_{x} \tLam e_{y}.
$$
\end{Def}
For a $\Delta_{0}$-graded finite-dimensional $\C$-vector space 
$W = \bigoplus_{x \in \Delta_{0}}W_{x}$, we can consider the variety 
$\rep_{W}(\Lambda)$
of representations of the algebra $\Lambda$ realized on $W$.
We have the natural conjugation action of the group $G_{W}$ on the variety $\rep_{W}(\Lambda)$. 

\begin{Prop}[Leclerc-Plamondon~\cite{LP13}]
\label{Prop:LP}
There is an isomorphism of $G_{W}$-varieties:
$$
\Mg_{0}(W)
\cong
\rep_{W}(\Lambda).
$$
\end{Prop}
\begin{proof}
This is~\cite[Theorem 2.4]{LP13}. Note that 
the graded quiver variety $\Mg_{0}(W)$ defined therein 
is naturally isomorphic to our graded quiver variety defined as the $T_{W}$-fixed locus of
$\M_{0}(\bar{W})$ (see~\cite[Section 4.1]{Nakajima01} for details).  
\end{proof}

Let $C$ be a subset of the vertex set $\Delta_{0}$. 
We denote by $\Lambda_{C}$
the quotient of the algebra $\Lambda$
by the ideal generated by all the idempotents $e_{x}$ with $x \not \in C$. 
We consider the following condition $\mathrm{(R)}$ on the subset $C$:
\begin{itemize}
\item[$\mathrm{(R)}$] For each vertex $x \in \Delta_{0}$, there is a vertex $c \in C$
such that the space $\Hom_{\C(\Delta)}(x, c)$ of morphisms in the category $\C(\Delta)$
does not vanish.
\end{itemize}

\begin{Thm}[{Keller-Scherotzke~\cite[Corollary 3.10]{KS16}}]
\label{Thm:KS}
Assume that our subset $C \subset \Delta_{0}$ satisfies the above condition $\mathrm{(R)}$.
Then we have a canonical isomorphism 
$$
\Ext_{\Lambda_{C}}^{k}\left(S_{x}, S_{y}\right) \cong \Ext_{\dD_{Q}}^{k}(\cH_{Q}(x), \cH_{Q}(y))
$$
for any Dynkin quiver $Q$ of type $\g$, vertices $x, y \in C$ and $k \in \Z_{\ge 1}$.
\end{Thm}

\begin{Rem}
The condition $\mathrm{(R)}$
 is obviously satisfied when $C = \Delta_{0}$. 
\end{Rem}

\begin{Def}
\label{Def:Gamma}
We define an infinite quiver $\Gamma$ with $\Gamma_{0}=\Delta_{0}$
whose arrow set $\Gamma_{1}$ is determined by the following condition:
$$
\# \{ a \in \Gamma_{1} \mid a^{\prime} = x, a^{\pprime} = y \} 
= \dim \Ext_{\dD_{Q}}^{1}(\cH_{Q}(x), \cH_{Q}(y)) \quad \text{for each $x, y \in \Delta_{0}$},
$$
where $Q$ is a Dynkin quiver of type $\g$.
\end{Def}
Note that,
by Proposition~\ref{Prop:Ext=tC}, 
there is no arrow from $(i,p)$ to $(j,r)$ in the quiver $\Gamma$
unless $p>r$. 
In particular,  the quiver $\Gamma$ has neither loops nor oriented cycles.

\begin{Cor}
\label{Cor:KS}
For any $\Delta_{0}$-graded finite-dimensional complex vector space $W$,
there is a $G_{W}$-equivariant closed embedding 
$$
\Mg_{0}(W) \hookrightarrow \rep_{W}(\Gamma).
$$
Moreover, if $W$ is supported on just two vertices $(i,p), (j,r) \in \Delta_{0}$, 
i.e.~$W = W_{(i,p)} \oplus W_{(j,r)}$ with $r \ge p$, 
the graded quiver variety $\Mg_{0}(W)$
is $G_{W}$-equivariantly isomorphic to the affine space
$$
\Hom_{\C}(W_{(j,r)}, W_{(i,p)}) \otimes \Ext_{\dD_{Q}}^{1}(\cH_{Q}(j,r), \cH_{Q}(i,p)),
$$
where $G_{W}$ acts trivially on the second tensor factor.
\end{Cor}
\begin{proof}
By a general theory (see~\cite[Theorem 3.7]{ASS06} for example),
the algebra $\Lambda$ can be written as a quotient of the path algebra 
of a quiver $Q_{\Lambda}$ by an admissible ideal $\mathfrak{J} \subset \C Q_{\Lambda}$. 
By Theorem~\ref{Thm:KS} (in the case $C=\Delta_{0}$), we have $Q_{\Lambda} = \Gamma$ and hence
the variety 
$\rep_{W}(\Lambda)$
is a closed subvariety of the affine space
\begin{align*}
\rep_{W}(\Gamma) & = \prod_{a \in \Gamma_{1}} \Hom_{\C}(W_{a^{\prime}}, W_{a^{\pprime}}) \\
&= \prod_{(i,p), (j,r) \in \Delta_{0}, r>p} \Hom_{\C}(W_{(j,r)}, W_{(i,p)}) \otimes_{\C} \Ext_{\dD_{Q}}^{1}(\cH_{Q}(j,r), \cH_{Q}(i,p))
\end{align*}
Combining with Proposition~\ref{Prop:LP},  
we obtain a $G_{W}$-equivariant closed embedding $\Mg_{0}(W) = \rep_{W}(\Lambda) \hookrightarrow \rep_{W}(\Gamma)$.

If $W = W_{(i,p)} \oplus W_{(j,r)}$ with $r \ge p$, we have
$$
\rep_{W}(\Gamma) = \Hom_{\C}(W_{(j,r)}, W_{(i,p)}) \otimes \Ext_{\dD_{Q}}^{1}(\cH_{Q}(j,r), \cH_{Q}(i,p)).
$$
In addition, 
all the polynomials corresponding to elements of $\mathfrak{J}$ vanish
because $\mathfrak{J} \subset \bigoplus_{d \ge 2} (\C \Gamma_{1})^{\otimes d} \subset \C \Gamma$.
Therefore the closed embedding is an isomorphism in this case.
\end{proof}

\begin{Rem}
\label{Rem:KS}
By the same argument, we can show the following more general assertion.  
If a $\Delta_{0}$-graded vector space $W$ is supported on a subset $C \subset \Delta_{0}$
satisfying the condition $\mathrm{(R)}$, the graded quiver variety $\Mg_{0}(W)$
is identical to the space of representations of the full subquiver $\Gamma|_{C} \subset \Gamma$
satisfying some relations corresponding to elements of an admissible ideal $\mathfrak{J}_{C}$ of the path algebra $\C \Gamma|_{C}$.  
\end{Rem}
\subsection{Stratification}
\label{Ssec:strat}

Let $W$ be a $\Delta_{0}$-graded vector space as above and 
$V$ be a $\Delta_{0}^{+}$-graded vector space. 
We set
$$Y^{W} \seq \prod_{x \in \Delta_{0}} Y_{x}^{\dim W_{x}},
\quad A^{-V} \seq \prod_{y \in \Delta_{0}^{+}} A_{y}^{-\dim V_{y}},
$$
where $Y_{(i, p)} \seq Y_{i, q^{p}}, A_{(i,p)} \seq A_{i, q^{p}} \in \Mm$ for $(i,p) \in I \times \Z$.

Let $\rep_{W \oplus V}(\tLam)$ be the space of representations of the 
algebra $\tLam$ on the $\tDel_{0}$-graded vector space $W \oplus V$. 
We have a natural $G_{W}$-equivariant forgetful morphism
$\rep_{W \oplus V}(\tLam) \to \rep_{W}(\Lambda) = \Mg_{0}(W)$. 
We consider the subvariety of $\rep_{W \oplus V}(\tLam)$
consisting of modules $M$
whose stabilizer in the group $G_{V}$
is trivial,
and denote its image 
under the forgetful morphism by 
$\Mreg(V, W) \subset \Mg_{0}(W)$. 
Note that $\Mreg(V, W)$ can be empty.

\begin{Thm}[Nakajima~\cite{Nakajima01}]
\label{Thm:stratification}
The collection $\{ \Mreg(V, W)\}_{V}$
of locally closed smooth $G_{W}$-subvarieties  
gives a stratification of the variety $\Mg_{0}(W)$  
with finitely many (non-empty) strata, satisfying the following properties:
\begin{enumerate}
\item \label{Thm:stratification:nonempty}
$\Mreg(V, W) \neq \varnothing$ if and only if 
$Y^{W}A^{-V} \in \Mm^{+}$ and $c(Y^{W}, Y^{W}A^{-V}) \neq 0$
(see Theorem~\ref{Thm:FR}~\eqref{Thm:FR:2} for the notation).
If this is the case, we have 
$$ c(Y^{W}, Y^{W}A^{-V}) = \dim \iota^{!} IC(\Mreg(V, W), \kk),$$
where $IC(\Mreg(V, W), \kk)$ denotes the intersection cohomology complex associated with
the constant $\kk$-sheaf on $\Mreg(V, W)$ and $\iota \colon \{ 0 \} \hookrightarrow \Mg_{0}(W)$
denotes the inclusion of the origin; 
\item \label{Thm:stratification:IC}
Let $\mathscr{L}_{W}$ be the (derived) push-forward
of the constant sheaf $\underline{\kk}_{\Mg(W)}$ 
along the proper morphism $\pi^{\bullet} \colon \Mg(W) \to \Mg_{0}(W)$. 
Then it has a decomposition:
$$
\mathscr{L}_{W} \cong 
\bigoplus_{V} IC(\Mreg(V, W), \kk) \otimes_{\kk} L^{\bullet}_{V}, 
$$ 
where 
$L^{\bullet}_{V} \in D^{b}(\kk \modcat)$ is a finite-dimensional $\Z$-graded $\kk$-vector space.
Moreover, $L^{\bullet}_{V} \neq 0$ if and only if $\Mreg(V, W) \neq \varnothing$;
\item \label{Thm:stratification:ordering}
The closure inclusion $\Mreg(V, W) \subset \overline{\Mreg(V^{\prime}, W)}$
between non-empty strata implies
$Y^{W}A^{-V} \ge Y^{W}A^{-V^{\prime}}$.
\end{enumerate}
\begin{proof}
See \cite[Section 14.3]{Nakajima01} for \eqref{Thm:stratification:nonempty}, \eqref{Thm:stratification:IC},
and \cite[Section 3.3]{Nakajima01} for \eqref{Thm:stratification:ordering}. 
\end{proof}
\end{Thm}

The next theorem describes the stratification $\{ \Mreg(V, W) \}_{V}$
in terms of the algebra $\Lambda$. 

\begin{Thm}[Keller-Scherotzke~\cite{KS16}]
\label{Thm:KS2}
Let $Q$ be a Dynkin quiver of type $\g$.
There exists a canonical $\delta$-functor $\Phi_{Q} \colon \Lambda \modcat \to \dD_{Q}$
such that $\Phi_{Q}(S_{x}) \cong \cH_{Q}(x) \in \dD_{Q}$ 
for each $x \in \Delta_{0}$
and satisfies the following properties:
\begin{enumerate}
\item \label{Thm:KS2:1}
Two representations $M_{1}, M_{2} \in \rep_{W}(\Lambda) = \Mg_{0}(W)$
belong to a common stratum $\Mreg(V, W)$ if and only if 
we have $\Phi_{Q}(M_{1}) \cong \Phi_{Q}(M_{2})$;
\item \label{Thm:KS2:2}
If we write $\Phi_{Q}(M) \cong \bigoplus_{x \in \Delta_{0}} \cH_{Q}(x)^{\oplus m_{x}}$
for $M \in \Mreg(V,W)$, we have
$Y^{W} A^{- V} = \prod_{x \in \Delta_{0}} Y_{x}^{m_x} \in \Mm^{+}$.
\end{enumerate}
\end{Thm}
\begin{proof}
See \cite[Theorem 2.7]{KS16} for \eqref{Thm:KS2:1}, and 
\cite[Lemma 4.14]{KS16} for \eqref{Thm:KS2:2}. 
\end{proof}
We refer to the above $\delta$-functor $\Phi_{Q} \colon \Lambda \modcat \to \dD_{Q}$ 
as the {\em stratification functor}.
For a concrete construction of $\Phi_{Q}$,
see~\cite[Sections 4 and 5]{KS16}.

\section{A geometric proof of the denominator formula}
\label{Sec:proof}

In this section, we give a proof of Theorem~\ref{Thm:main2},
which is equivalent to our main theorem (= Theorem~\ref{Thm:main}).
We also describe a structure of 
the tensor product module $V_{i}(a) \otimes V_{j}(b)$ when $R_{ij}$ has a simple pole
at $z_{2} / z_{1} = b/a$ in Section~\ref{Ssec:simplepole} .  

\subsection{Proof of Theorem~\ref{Thm:main2}}
\label{Ssec:pf}

For $x= (i,p), y=(j,r) \in \Delta_{0}$, we set
\begin{align*}
V(x,y) &\seq V_{i}(q^{p}) \otimes V_{j}(q^{r}) = L(Y_{x}) \otimes L(Y_{y}), \\
\hV(x,y) &\seq \OO \otimes_{\kk[z_{1}^{\pm 1}, z_{2}^{\pm1}]} \left(V_{i}[z_{1}^{\pm 1}] \otimes V_{j}[z_{2}^{\pm 1}] \right), 
\end{align*}
where $\OO \seq \kk[\![ z_{1}-q^{p}, z_{2}-q^{r} ]\!]$.
We have $V(x,y) \cong \hV(x,y)/\mm \hV(x,y)$ as $U_{q}(L\g)$-modules, 
where $\mm \subset \OO$ is the maximal ideal. 
Since the module $V_{i}[z_{1}^{\pm 1}] \otimes V_{j}[z_{2}^{\pm 1}]$ is free over $\kk[z_{1}^{\pm 1}, z_{2}^{\pm1}]$,
the module $\hV(x,y)$ is free over $\OO$.
We denote the image of the vector $v_{i} \otimes v_{j}$ under the natural homomorphism
$V_{i}[z_{1}^{\pm 1}]\otimes V_{j}[z_{2}^{\pm 1}] \to V(x,y)$ (resp.~$V_{i}[z_{1}^{\pm 1}]\otimes V_{j}[z_{2}^{\pm 1}] \to \hV(x,y)$)
by $v_{x,y}$ (resp.~$\hat{v}_{x,y}$).

Let $\KK$ be the fraction field of $\OO$.
We set 
$$\hV(x,y)_{\KK} \seq \KK \otimes_{\OO} \hV(x,y) = \KK \otimes_{\kk[z_{1}^{\pm 1}, z_{2}^{\pm 1}]} (V_{i}[z_{1}^{\pm 1}] \otimes V_{j}[z_{2}^{\pm 1}]).$$
We can naturally regard $\hV(x,y)$ as a submodule of $\hV(x,y)_{\KK}$.
The normalized $R$-matrix $R_{ij} \colon V_{i}[z_{1}^{\pm 1}] \otimes V_{j}[z_{2}^{\pm 1}] 
\to \kk(z_{2}/z_{1})\otimes_{\kk[(z_{2}/z_{1})^{\pm 1}]} \left(V_{j}[z_{2}^{\pm 1}] \otimes V_{i}[z_{1}^{\pm 1}]\right)$ induces 
a unique $U_{q}(L\g)\otimes \KK$-homomorphism
$$
\widehat{R}_{x,y} \colon \hV(x,y)_{\KK} \to \hV(y,x)_{\KK}
$$
characterized by the property $\widehat{R}_{x,y} (\hat{v}_{x,y}) = \hat{v}_{y,x}$. 
Since the $U_{q}(L\g)\otimes \KK$-modules $\hV(x,y)_{\KK}$
and $\hV(y,x)_{\KK}$ are irreducible (see~\cite[Lemma 8.1]{Kashiwara02} for example), the homomorphism $\widehat{R}_{x,y}$
is an isomorphism, and we have
\begin{equation}
\label{Eq:uniqueness}
\Hom_{U_{q}(L\g)\otimes \KK}\left(\hV(x,y)_{\KK}, \hV(y,x)_{\KK} \right)
= \KK \widehat{R}_{x,y}. 
\end{equation}

Let $d_{x,y} \seq \dim \Ext^{1}_{\dD_{Q}}(\cH_{Q}(y), \cH_{Q}(x))$, where $Q$ is a Dynkin quiver 
of type $\g$. 
If $r \le p$, we have $d_{x,y} = 0$ by Proposition~\ref{Prop:Ext=tC}. 
On the other hand, we know that $d_{ij}(q^{r}/q^{p}) \neq 0$ for $r \le p$ by Theorem~\ref{Thm:CK}.   
Therefore, to prove Theorem~\ref{Thm:main2}, we may assume that $r > p$.
Then, it suffices to verify the following two properties:
\begin{align}
(z_{2}/z_{1} -q^{r}/q^{p})^{d_{x,y}} \widehat{R}_{x,y}\left(\hV(x,y)\right) & \subset \hV(y,x), \label{Eq:Property1} \\
(z_{2}/z_{1} -q^{r}/q^{p})^{d_{x,y}} \widehat{R}_{x,y}\left(\hV(x,y)\right) & \not \subset \mm \hV(y,x). \label{Eq:Property2}
\end{align}
We prove these properties by using geometry of the graded quiver varieties. 

Let $W = W_{x} \oplus W_{y}$ be the $\Delta_{0}$-graded vector space 
supported on $\{ x, y \} \subset \Delta_{0}$
satisfying $\dim W_{x} = \dim W_{y} = 1$.
In this case, we have $G_{W} = GL(W_{x}) \times GL(W_{y}) = (\C^{\times})^{2}$.
In what follows, we identify the completed representation ring $\widehat{R}(G_{W})_{\kk}$
with the ring $\OO$ by the isomorphism
\begin{equation}
\label{Eq:RO}
\widehat{R}(G_{W})_{\kk} \cong \OO, \quad [W_{x}] \leftrightarrow q^{-p}z_{1}, \; [W_{y}] \leftrightarrow q^{-r}z_{2},
\end{equation}
where $[W_{x}]$ and $[W_{y}]$ denote the classes of the natural $1$-dimensional representations 
of $G_{W}=GL(W_{x}) \times GL(W_{y})$.
By Corollary~\ref{Cor:KS}, the graded quiver variety  
$\Mg_{0}(W)$ is identified with the 
affine space $E = \C^{d_{x,y}}$ of dimension $d_{x,y}$ as a $G_{W}$-variety.
Here the action of the group $G_{W}=(\C^{\times})^{2}$ on $E$ is given by 
$(s_{1}, s_{2}) \cdot e = s_{1}s_{2}^{-1} e$, 
where $(s_{1}, s_{2}) \in (\C^{\times})^{2}$ and
$e \in E$. Let $\iota \colon \{ 0 \} \hookrightarrow E$ denote
the inclusion of the origin. 
From the morphisms $\pi^{\bullet} \colon \Mg(W) \to \Mg_{0}(W) = E$, 
$\mathrm{id} \colon E \to E$ and $\iota \colon \{ 0 \} \to E$, 
we make the fiber products $\Mg(W) \times_{E} E \subset \Mg(W) \times E$
and $\Mg(W) \times_{E} \{ 0 \} \subset \Mg(W) \times \{ 0 \}$.
The convolution product makes the completed equivariant $K$-groups
$\hK^{G_{W}}(\Mg(W) \times_{E} E)_{\kk}$ and 
$\hK^{G_{W}}(\Mg(W) \times_{E} \{ 0 \})_{\kk}$
into left $\hK^{G_{W}}(\Zg(W))_{\kk}$-modules.
Via the completed Nakajima homomorphism 
$\widehat{\Psi}_{W} \colon U_{q}(L\g) \to \hK^{G_{W}}(\Zg(W))_{\kk}$,
they are regarded as left $U_{q}(L\g)$-modules.

\begin{Lem}
\label{Lem:Grea}
There are isomorphisms of $U_{q}(L\g) \otimes \OO$-modules
\begin{align}
\hK^{G_{W}}(\Mg(W) \times_{E} E)_{\kk} &\cong \hV(x,y), \label{Eq:KV1}  \\
\hK^{G_{W}}(\Mg(W) \times_{E} \{ 0 \})_{\kk} &\cong \hV(y,x), \label{Eq:KV2}
\end{align}
under which the class $[\Oo_{\{\hat{0}\}}]$ of the structure sheaf of $\{ \hat{0} \} \subset \Mg(W)$ 
corresponds to 
the vectors $\hat{v}_{x,y}$ and $\hat{v}_{y,x}$ respectively.
\end{Lem}
\begin{proof}
Define the $1$-parameter subgroups 
$\lambda_{x,y} \colon \C^{\times} \to G_{W}$
and $\lambda_{y,x} \colon \C^{\times} \to G_{W}$
by $\lambda_{x,y}(t) \seq (t,1)$ and $\lambda_{y,x}(t) \seq (1,t)$ respectively.
Since
$\lambda_{x,y}(t) \cdot e = t e$ and $\lambda_{y,x}(t) \cdot e = t^{-1} e$
for any point $e \in E=\Mg_{0}(W)$, we have
\begin{align*}
\Mg(W) \times_{E} E &= \left\{ m \in \Mg(W) \; \middle| \; \lim_{t \to 0} \lambda_{x,y} (t) \pi^{\bullet}(m) = 0 \right\}, \\
\Mg(W) \times_{E} \{ 0 \} &= \left\{ m \in \Mg(W) \; \middle| \; \lim_{t \to 0} \lambda_{y,x} (t) \pi^{\bullet}(m) = 0 \right\}.
\end{align*}
From these descriptions, we see that 
they coincide with the $T_{W}$-fixed loci of 
the tensor product varieties introduced by Nakajima~\cite{Nakajima01t}.
More precisely, they are $\tZ(W_{x};W_{y})^{T_{W}}$
and $\tZ(W_{y};W_{x})^{T_{W}}$ respectively in the notation loc.~cit. 
Therefore, we can proceed the similar argument as the proof of~\cite[Theorem 6.12]{Nakajima01t}
in the $\rr_{W}$-adically completed setting to obtain 
the following isomorphisms of $U_{q}(L\g) \otimes \OO$-modules:
\begin{align*}
\hK^{G_{W}}(\Mg(W) \times_{E} E)_{\kk} & 
\cong \hK^{G_{W_{x}}}(\Mg(W_{x}))_{\kk} \hat{\otimes} \hK^{G_{W_{y}}}(\Mg(W_{y}))_{\kk},  \\
\hK^{G_{W}}(\Mg(W) \times_{E} \{0\})_{\kk} & \cong
\hK^{G_{W_{y}}}(\Mg(W_{y}))_{\kk} \hat{\otimes} \hK^{G_{W_{x}}}(\Mg(W_{x}))_{\kk}, 
\end{align*}
where $K\hat{\otimes}K^{\prime}$ denotes the completion
of $K\otimes_{\kk}K^{\prime}$.  
On the other hand, there are isomorphisms
$\hK^{G_{W_{x}}}(\Mg(W_{x}))_{\kk} \cong V_{i}[z_{1}^{\pm 1}] \otimes_{\kk[z_{1}^{\pm 1}]}\kk[\![z_{1}-q^{p}]\!]$
and $\hK^{G_{W_{y}}}(\Mg(W_{y}))_{\kk} \cong V_{j}[z_{2}^{\pm 1}]\otimes_{\kk[z_{2}^{\pm 1}]}\kk[\![ z_{2} - q^{r}]\!]$
by Lemma~\ref{Lem:greaVx}.
Thus we obtain the desired isomorphisms \eqref{Eq:KV1} and \eqref{Eq:KV2}. 
\end{proof} 

Now we consider the completed equivariant $K$-group
$\hK^{G_{W}}(E \times_{E} \{ 0 \})_{\kk}$.
Since $E \times_{E} \{ 0 \} = \{ 0 \}$, 
this is a free $\OO$-module of rank $1$ 
generated by the class $[\Oo_{\{ 0 \}}]$ of the structure sheaf.  
The convolution product with the class $[\Oo_{\{ 0 \}}]$ from the right 
$$
(-) * [\Oo_{\{ 0 \}}] \colon \hK^{G_{W}}(\Mg(W) \times_{E} E)_{\kk} \to \hK^{G_{W}}(\Mg(W) \times_{E} \{ 0 \})_{\kk} 
$$
is identified via the isomorphisms \eqref{Eq:KV1} and \eqref{Eq:KV2} with a $U_{q}(L\g) \otimes \OO$-homomorphism
$$
\mathbf{r} \colon \hV(x,y) \to \hV(y,x). 
$$ 
By the base change $\OO \to \KK$ (resp.~$\OO \to \kk$), 
the homomorphism $\mathbf{r}$ gives rise to
$\mathbf{r}_{\KK} \in \Hom_{U_{q}(L\g) \otimes \KK}(\hV(x,y)_{\KK}, \hV(y,x)_{\KK})$
(resp.~$\bar{\mathbf{r}} \in \Hom_{U_{q}(L\g)}(V(x,y), V(y,x))$).

The following two lemmas prove the properties \eqref{Eq:Property1} and \eqref{Eq:Property2} respectively, 
and hence complete the proof of Theorem~\ref{Thm:main2}. 

\begin{Lem}
Up to $\kk^{\times}$-multiplication, the homomorphism $\mathbf{r}_{\KK}$ is equal to 
the homomorphism $(z_{2}/z_{1} -q^{r}/q^{p})^{d_{x,y}} \widehat{R}_{x,y}$.
In particular, the property \eqref{Eq:Property1} holds.
\end{Lem}
\begin{proof}
In this proof, we identify $\hV(x,y)$ and $\hV(y,x)$ with the completed equivariant $K$-groups via the isomorphisms \eqref{Eq:KV1} and \eqref{Eq:KV2} respectively. 
Let us compute the operator $\mathbf{r} = (-) * [\Oo_{\{ 0 \}}]$
following the definition of the convolution product \eqref{Eq:convolution}.
For any $\zeta \in \hK^{G_{W}}(\Mg(W) \times_{E} E)$, we have
\begin{align*}
\mathbf{r}(\zeta) &= p_{13*}(p_{12}^{*}\zeta \otimes^{\mathbb{L}}_{\Mg(W) \times E \times \{ 0 \}} p_{23}^{*} [\Oo_{\{0\}}]) \\
&= p^{\prime}_{1*}(\zeta \otimes^{\mathbb{L}}_{\Mg(W) \times E} p^{\prime*}_{2}[\Oo_{\{0\}}]),
\end{align*}
where $p^{\prime}_{1} \colon \Mg(W) \times E \to \Mg(W)$ and
$p^{\prime}_{2} \colon \Mg(W) \times E \to E$ are the natural projections.
By the Koszul resolution, we have 
$$
[\Oo_{\{ 0 \}}] = \left(\sum_{k=1}^{d_{x,y}} (-1)^{k}\left[\wedge^{k}T^{*}_{0}E\right]\right) [\Oo_{E}]
= \left(1-q^{-r}z_{2}/q^{-p}z_{1}\right)^{d_{x,y}} [\Oo_{E}],
$$
where we regard $[\wedge^{k}T^{*}_{0}E] = [\wedge^{k}E^*] \in R(G_{W})$ and used the identification \eqref{Eq:RO}. 
Thus we obtain  
$\mathbf{r}(\zeta) = \left(1-q^{-r}z_{2}/q^{-p}z_{1} \right)^{d_{x,y}} p^{\prime}_{1*}(\zeta).$
In the special case $\zeta = [\Oo_{\{\hat{0} \}}]= \hat{v}_{x,y}$, 
we have 
$p_{1*}^{\prime} [\Oo_{\{ \hat{0} \}}] = [\Oo_{\{ \hat{0} \}}]$
and hence 
$\mathbf{r}(\hat{v}_{x,y}) = \left(1-q^{-r}z_{2}/q^{-p}z_{1} \right)^{d_{x,y}} \hat{v}_{y,x}$. 
Thanks to \eqref{Eq:uniqueness},
we conclude that $\mathbf{r}_{\KK} = \left(1-q^{-r}z_{2}/q^{-p}z_{1} \right)^{d_{x,y}} \widehat{R}_{x,y}$. 
\end{proof}

\begin{Lem}
The homomorphism $\bar{\mathbf{r}}$ is non-zero. 
In particular, the property \eqref{Eq:Property2} holds.
\end{Lem}
\begin{proof}
Applying the base change $\OO \to \kk$ to the isomorphisms 
 \eqref{Eq:KV1} and \eqref{Eq:KV2} in Lemma~\ref{Lem:Grea},
we obtain
$$
K(\Mg(W) \times_{E} E)_{\kk} \cong V(x,y), \quad
K(\Mg(W) \times_{E} \{ 0 \})_{\kk} \cong V(y, x).  
$$
Here we used the freeness of the equivariant $K$-groups of the quiver varieties (see~\cite[Theorem 7.3.5]{Nakajima01}). 
Under these isomorphisms, the homomorphism $\bar{\mathbf{r}} \colon V(x,y) \to V(y,x)$ is identified with 
the convolution operation
$$
(-)* [\Oo_{\{0\}}] \colon K(\Mg(W) \times_{E} E)_{\kk} \to K(\Mg(W) \times_{E} \{ 0 \})_{\kk}.
$$
Here the class $[\Oo_{\{ 0 \}}] \neq 0$ is regarded as an element of 
$K(E \times_{E} \{ 0 \})_{\kk} = K(\mathrm{pt})_{\kk}$.

Let $D^{b}_{c}(E)$ denote the bounded derived category of constructible complexes of $\kk$-vector spaces on $E$.
For $\mathscr{F}, \mathscr{G} \in D^{b}_{c}(E)$, we denote by
$\Ext^{*}(\mathscr{F}, \mathscr{G})$ the direct sum of 
the spaces $\Hom_{D^{b}_{c}(E)}(\mathscr{F}, \mathscr{G}[k])$ over $k \in \Z$.   
By the Chern character map (with a certain modification, see~\cite[Section 5.11]{CG97}) and 
a standard isomorphism explained in~\cite[Section 8.6]{CG97}, we obtain the following commutative diagram:
$$
\xy
\xymatrix{
K(\Mg(W) \times_{E} E)_{\kk} \otimes K(E \times_{E} \{0\})_{\kk}
\ar[r]^-{*}
\ar[d]^{\cong}
&
K(\Mg(W) \times_{E} \{ 0 \})_{\kk}
\ar[d]^{\cong}
\\
\Ext^{*}(\ud{\kk}_{E}, \mathscr{L}^{\bullet}_{W})
\otimes
\Ext^{*}(\ud{\kk}_{\{0\}}, \ud{\kk}_{E})
\ar[r]^-{\circ} 
&
\Ext^{*}(\ud{\kk}_{\{0\}}, \mathscr{L}^{\bullet}_{W}),
}
\endxy
$$ 
where $\mathscr{L}^{\bullet}_{W}$ denotes the (derived) proper push-forward along $\pi^{\bullet}$
of the constant sheaf on $\Mg(W)$.
The lower horizontal arrow denotes the Yoneda product 
$(\varphi_{1}, \varphi_{2}) \mapsto \varphi_{1} \circ \varphi_{2}$.
The vertical arrows are isomorphisms thanks to~\cite[Theorem 7.4.1]{Nakajima01}.
Note that the $\kk$-vector space 
$\Ext^{*}(\ud{\kk}_{\{0\}}, \ud{\kk}_{E}) = \Hom_{D^{b}_{c}(E)}\left(\iota_{!}\iota^{!}\ud{\kk}_{E}, \ud{\kk}_{E}\right)$
is $1$-dimensional and spanned by the adjoint morphism $\eta \colon \iota_{!}\iota^{!}\ud{\kk}_{E} \to \ud{\kk}_{E}$.
Under the vertical isomorphism 
$K(E \times_{E} \{0\})_{\kk} \cong \Ext^{*}(\ud{\kk}_{\{0\}}, \ud{\kk}_{E})$,
the class $[\Oo_{\{ 0 \}}]$ corresponds to the adjoint morphism $\eta$ up to $\kk^{\times}$. 
Since $E = \Mg_{0}(W)$ is an affine space, there exists a unique open dense stratum $\Mreg(V, W) \subset E$
and the corresponding intersection cohomology complex $IC(\Mreg(V, W), \kk)$ is just a shift of the constant sheaf $\ud{\kk}_{E}$. 
By Theorem~\ref{Thm:stratification}~\eqref{Thm:stratification:IC}, 
we see that $\mathscr{L}^{\bullet}_{W}$ contains 
a certain shift of the constant sheaf $\ud{\kk}_{E}$ as a direct summand. 
This implies that the Yoneda product
$$
(-) \circ \eta \colon \Ext^{*}\left(\ud{\kk}_{E}, \mathscr{L}^{\bullet}_{W}\right) 
\to \Ext^{*}(\ud{\kk}_{\{0\}}, \mathscr{L}^{\bullet}_{W}) 
$$  
is non-zero because $\mathrm{id}_{\ud{\kk}_{E}} \circ \eta = \eta \neq 0$. Thus the homomorphism $\bar{\mathbf{r}} = (-) * [\Oo_{\{ 0 \}}]$
is also non-zero.
\end{proof}

\subsection{A remark on the case of simple pole} 
\label{Ssec:simplepole}

Let $x = (i,p), y = (j, r) \in \Delta_{0}$ and
assume that the normalized $R$-matrix $R_{ij}$
has a simple pole at $z_{2}/z_{1} = q^{r}/q^{p}$.
By Theorem~\ref{Thm:main2}, this assumption is equivalent to the condition  
$\dim \Ext^{1}_{\dD_{Q}}(\cH_{Q}(y), \cH_{Q}(x)) = 1$
for a Dynkin quiver $Q$ of type $\g$.
Therefore there exists the following non-split exact triangle in $\dD_{Q}$:
$$
\cH_{Q}(x) \to \bigoplus_{w \in \Delta_{0}} \cH_{Q}(w)^{\oplus \mu_w} \to \cH_{Q}(y) \xrightarrow{+1},
$$
where the multiplicities $(\mu_{w})_{w \in \Delta_{0}}$ 
are uniquely determined by the pair $(x, y)$ and independent from the choice of $Q$.
Then we define a dominant monomial
$
m_{[y,x]} \in \Mm^{+} 
$
by
\begin{equation}
\label{Eq:m[y,x]}
m_{[y,x]} \seq \prod_{w \in \Delta_{0}} Y_{w}^{\mu_{w}}.
\end{equation}

\begin{Prop}
\label{Prop:Dorey}
Let $x = (i,p), y = (j, r) \in \Delta_{0}$ and
assume that the normalized $R$-matrix $R_{ij}$
has a simple pole at $z_{2}/z_{1} = q^{r}/q^{p}$.
With the above notation, we have the following non-split short exact sequences
in $\Cc$:
\begin{align*}
0 \to L(Y_{x}Y_{y}) \to L(Y_{x}) \otimes L(Y_{y}) \to L(m_{[y,x]}) \to 0, \\
0 \to L(m_{[y, x]}) \to L(Y_{y}) \otimes L(Y_{x}) \to L(Y_{x}Y_{y}) \to 0. 
\end{align*}
\end{Prop}
\begin{proof}
As in Section~\ref{Ssec:pf} above, we consider  
the graded quiver variety $E\seq\Mg_{0}(W)$ associated with 
a $\Delta_{0}$-graded vector space 
$W = W_{x} \oplus W_{y}$ such that $\dim W_{x} = \dim W_{y} = 1$.
By Corollary~\ref{Cor:KS}, our assumption implies 
that $E$ is just a $1$-dimensional affine space.  
Since the action of $(s_{1}, s_{2}) \in G_{W} = (\C^{\times})^{2}$ on $E$ 
is given by the multiplication of $s_{1}s_{2}^{-1} \in \C^{\times}$,
there are only two $G_{W}$-orbits $\{0\}$ and $E \setminus \{ 0 \}$. 
On the other hand, $E = \Mg_{0}(W)$ is stratified 
by the $G_{W}$-stable subvarieties 
$\Mreg(V, W)$ 
and we know that $\{ 0 \}$ coincides with the stratum $\Mreg(0, W)$. 
Therefore there exists a unique $V$ such that $E \setminus \{ 0 \} = \Mreg(V, W)$. 
Since $IC(\Mreg(V, W), \kk) = \ud{\kk}_{E}[1]$,
we have $c(Y_{x}Y_{y}, Y_{x}Y_{y}A^{-V}) = \dim \iota^{!} (\ud{\kk}_{E}[1]) = 1$
by Theorem~\ref{Thm:stratification}~\eqref{Thm:stratification:nonempty}
and hence
$$
[L(Y_{x}) \otimes L(Y_{y})] = [L(Y_{x}Y_{y})] + [L(Y_{x}Y_{y}A^{-V})]
$$
in the Grothendieck ring $K(\Cc)$. Then, in view of Theorem~\ref{Thm:tensor}, we have
\begin{align*}
0 \to L(Y_{x}Y_{y}) \to L(Y_{x}) \otimes L(Y_{y}) \to L(Y_{x}Y_{y}A^{-V}) \to 0, \\
0 \to L(Y_{x}Y_{y}A^{-V}) \to L(Y_{y}) \otimes L(Y_{x}) \to L(Y_{x}Y_{y}) \to 0,
\end{align*}
which are exact and non-split.
It remains to show $Y_{x}Y_{y}A^{-V} = m_{[y,x]}$.
Recall that $E = \rep_{W}(\Lambda)$ and pick 
a $\Lambda$-module $M$ corresponding to a point of $E \setminus \{0\}$.
Then there exists a non-split short exact sequence in $\Lambda \modcat$:
$$
0 \to S_{x} \to M \to S_{y} \to 0.
$$
By applying the stratifying functor $\Phi_{Q} \colon \Lambda \modcat \to \dD_{Q}$ 
in Theorem~\ref{Thm:KS2},
we obtain an exact triangle in $\dD_{Q}$:
\begin{equation}
\label{Eq:xMy}
\cH_{Q}(x) \to \Phi_{Q}(M) \to \cH_{Q}(y) \xrightarrow{+1}.
\end{equation}
By Theorem~\ref{Thm:KS2}~\eqref{Thm:KS2:1},  
we see that the isomorphism class of $\Phi_{Q}(M)$ 
does not depend on the choice of $M \in E \setminus \{ 0 \}$ and the exact triangle \eqref{Eq:xMy} 
does not split. 
Thus we get $Y_{x}Y_{y}A^{-V} = m_{[y,x]}$ by Theorem~\ref{Thm:KS2}~\eqref{Thm:KS2:2} and the definition \eqref{Eq:m[y,x]}.
\end{proof}
\section{Generalized quantum affine Schur-Weyl duality}
\label{Sec:SW}

In this section, as an application of our discussion so far,
we give a geometric interpretation
of the generalized quantum affine Schur-Weyl duality functor
when it arises from a family of fundamental modules.   

\subsection{KKK-functors}
\label{Ssec:KKK}
First, we shall outline the original construction by Kang-Kashiwara-Kim.
Let $\{ (\sV_{j}, \sa_{j}) \}_{j \in J}$ be a family indexed by an arbitrary set $J$,
consisting of pairs of a real simple module $\sV_{j} \in \Cc$
(i.e.~a simple object in $\Cc$ whose tensor square remains simple)
and a non-zero scalar $\sa_{j} \in \kk^{\times}$.  
Recall that we have the normalized $R$-matrix $R_{\sV_{i}, \sV_{j}}$
and its denominator $d_{\sV_{i}, \sV_{j}}(u) \in \kk[u]$
for each $(i,j) \in J^{2}$ (cf.~Remark~\ref{Rem:generalR}).  

\begin{Def}
\label{Def:GammaJ}
Given a family $\{ (\sV_{j}, \sa_{j}) \}_{j \in J}$ as above,
we define a quiver $\Gamma_{J}$ with
$(\Gamma_{J})_{0} \seq J$
whose arrow set $(\Gamma_{J})_{1}$ is determined by the following condition:
$$
\# \{ a \in (\Gamma_{J})_{1} \mid a^{\prime} = j, a^{\pprime} = i \}
= (\text{zero order of $d_{\sV_{i}, \sV_{j}}(u)$ at $u = \sa_{j}/\sa_{i}$}),
$$
for each $(i,j) \in J^{2}$.
\end{Def}

The quiver $\Gamma_{J}$ has no loops since each $\sV_j$ is real.
Let $\g_{J}$ be the Kac-Moody algebra associated with the underlying graph of $\Gamma_{J}$ 
and $\{ \alpha^{J}_{j} \}_{j \in J}$ its set of simple roots. 
We put $\cQ_{J}^{+} \seq \sum_{j \in J} \Z_{\ge 0} \alpha^{J}_{j}$. 
For each $\beta \in \cQ_{J}^{+}$,  
we denote by $H_{J}(\beta)$
the corresponding quiver Hecke algebra.
This is a $\Z$-graded $\kk$-algebra
defined by generators and relations
(see~\cite[Section 1.2]{KKK18} for instance).
Let $H_{J}(\beta) \gmod$ denote the category
of finite-dimensional graded $H_{J}(\beta)$-modules.
The direct sum 
$H_{J} \gmod \seq \bigoplus_{\beta \in \cQ_{J}^{+}} H_{J}(\beta) \gmod$
carries a structure of a $\kk$-linear monoidal category 
with respect to the so-called \textit{convolution product},
which is an analogue of the parabolic induction 
for the affine Hecke algebras.

In the above setting, 
Kang-Kashiwara-Kim~\cite{KKK18}
constructed a bimodule
\begin{equation}
\label{Eq:KKK_bimodule}
 U_{q}(L\g) \quad \curvearrowright 
\quad \widehat{\sV}^{\otimes \beta}
\quad \curvearrowleft 
\quad  \widehat{H}_{J}(\beta)
\end{equation}
with some good properties.
Here $\widehat{H}_{J}(\beta)$ denotes the completion of $H_{J}(\beta)$
along the $\Z$-grading.
As a left $U_{q}(L\g)$-module, $\widehat{\sV}^{\otimes \beta}$
is a direct sum of suitable tensor products of the completed modules
$\sV_{j}[z^{\pm 1}] \otimes_{\kk[z^{\pm 1}]} \kk[\![ z - \sa_{j} ]\!]$ 
for various $j \in J$.
The right action of $\widehat{H}_{J}(\beta)$ is given by an explicit formula 
involving the normalized $R$-matrices 
$R_{\sV_{i}, \sV_{j}}$.
See~\cite[Section 3]{KKK18} for details.
The assignment 
$M \mapsto \widehat{\sV}^{\otimes \beta} \otimes_{\widehat{H}_{J}(\beta)} M$
combined with the forgetful functor $H_{J}(\beta) \gmod \to \widehat{H}_{J}(\beta) \modcat$
yields a $\kk$-linear functor $H_{J}(\beta) \gmod \to \Cc$.
Summing up over $\beta \in \cQ_{J}^{+}$, we obtain a $\kk$-linear monoidal functor 
$$
\mathscr{F}_{J} \colon H_{J} \gmod \to \Cc,
$$
which we refer to as the \textit{generalized quantum affine Schur-Weyl duality functor},
or simply the \textit{KKK-functor} associated with the family $\{ (\sV_{j}, \sa_{j}) \}_{j \in J}$. 
 
\subsection{A geometric interpretation}
\label{Ssec:ginterpret}
In this subsection,
we give a geometric interpretation 
of the KKK-functor when it arises from a family of fundamental modules.
Namely, we restrict ourselves to the case when 
$\sV_{j} \in \{ V_{i}(1) \mid i \in I \}$ for every $j \in J$.
Furthermore, we focus on the case when
the associated quiver $\Gamma_{J}$ is connected.
Then, in view of our denominator formula \eqref{Eq:denominatorK},
we may assume that there exists
an injective map 
$
\sx \colon J \hookrightarrow \Delta_{0}
$
which determines the family $\{ (\sV_{j}, \sa_{j}) \}_{j \in J}$ by 
$$ (\sV_{j}, \sa_{j}) = (V_{\sx_{1}(j)}(1), q^{\sx_{2}(j)})$$
for every $j \in J$,
where we write $\sx(j) = (\sx_{1}(j), \sx_{2}(j)) \in I \times \Z$.

The next lemma is a key for our construction. 
Recall the quiver $\Gamma$ from Definition~\ref{Def:Gamma}.

\begin{Lem}
\label{Lem:key}
Under the above assumption, 
the quiver $\Gamma_{J}$ in Definition~\ref{Def:GammaJ}
is identical to the full subquiver $\Gamma |_{\sx(J)}$
of the quiver $\Gamma$ 
whose vertex set is the image $\sx(J)$ 
of the injective map $\sx \colon J \hookrightarrow \Delta_{0}= \Gamma_{0}$.
\end{Lem}
\begin{proof}
This is a consequence of Theorem~\ref{Thm:main2}.
\end{proof}

In what follows, we often identify
a $J$-graded vector space $D = \bigoplus_{j \in J} D_{j}$ with 
the $\Delta_{0}$-graded vector space $\bigoplus_{x \in \Delta_{0}}D_{x}$
defined by
$$
D_{x} \seq \begin{cases}
D_{j} & \text{if $x = \sx(j)$ with $j \in J$}; \\
0 & \text{if $x \not \in \sx(J)$}. 
\end{cases}
$$
Under this convention, we have
$
\rep_{D}(\Gamma) = \rep_{D}(\Gamma_{J})
$
by Lemma~\ref{Lem:key} above. 

Now we fix $\beta = \sum_{j \in J} d_{j} \alpha_{j}^{J} \in \cQ_{J}^{+}$ and
a $J$-graded complex vector space $D_{\beta} = \bigoplus_{j \in J} (D_{\beta})_{j}$ such that
$\dim (D_{\beta})_{j} = d_{j}$ for each $j \in J$.
To simplify the notation,
we set $G_{\beta} \seq G_{D_{\beta}}$
and $E_{\beta} \seq \rep_{D_{\beta}}(\Gamma_{J})$.   
Let us consider the following two non-singular $G_{\beta}$-varieties:
\begin{align*}
\B_{\beta} &= 
\{F^{\bullet} = (D_{\beta}=F^{0} \supsetneq F^{1} \supsetneq \cdots 
\supsetneq F^{d}=0) \mid \text{$F^{k}$ is a $J$-graded 
subspace of $D_{\beta}$} \}, \\
\F_{\beta} &=
\{(F^{\bullet}, X) \in \B_{\beta} \times E_{\beta}
\mid X(F^{k}) \subset F^{k} \; \text{for any $1\le k \le d$} \},
\end{align*}
where $d \seq \sum_{j \in J} d_{j} = \dim D_{\beta}$.
We denote by 
$\mu \colon \F_{\beta} \to E_{\beta}$ 
the second projection $\mu(F^{\bullet}, X) = X$.
This is 
a $G_{\beta}$-equivariant proper morphism since 
$\B_{\beta}$ is a projective variety.
Combined with Corollary~\ref{Cor:KS},
we have obtained the following diagram
\begin{equation}
\label{Eq:geom_diagram}
\Mg(D_{\beta}) \xrightarrow{\pi^{\bullet}} \Mg_{0}(D_{\beta}) \hookrightarrow 
E_{\beta} \xleftarrow{\mu} \F_{\beta}
\end{equation}
consisting of $G_{\beta}$-equivariant proper morphisms.
Applying the convolution construction,
we get a bimodule
\begin{equation}
\label{Eq:geom_bimodule}
K^{G_{\beta}}(Z^{\bullet}(D_{\beta}))_{\kk} 
\quad \curvearrowright
\quad K^{G_{\beta}}(\Mg(D_{\beta}) \times_{E_{\beta}} \F_{\beta})_{\kk}
\quad \curvearrowleft 
\quad K^{G_{\beta}}(\mathcal{Z}_{\beta})_{\kk}, 
\end{equation}
where we set
$Z^{\bullet}(D_{\beta}) \seq \Mg(D_{\beta}) \times_{E_{\beta}} \Mg(D_{\beta})$
and $\mathcal{Z}_{\beta} \seq \F_{\beta} \times_{E_{\beta}} \F_{\beta}$.

For each $\beta \in \cQ_{J}^{+}$,
we denote by
$\Ext_{G_{\beta}}^{k}(\mathscr{F}, \mathscr{G})$
the $k$-th $\Ext$-space in the $G_{\beta}$-equivariant bounded derived category
of complexes of $\kk$-sheaves on $E_{\beta}$
and define
$$\Ext_{G_{\beta}}^{*}(-, -) \seq \bigoplus_{k \in \Z}\Ext_{G_{\beta}}^{k}(-, -),
\quad
\Ext_{G_{\beta}}^{*}(-, -)^{\wedge} \seq \prod_{k \in \Z}\Ext_{G_{\beta}}^{k}(-, -).
$$
Let $\mathscr{L}_{\beta}$ denote
the push-forward of the constant perverse $\kk$-sheaf on 
the smooth variety $\F_{\beta}$
along the $G_{\beta}$-equivariant proper morphism
$\mu_{\beta} \colon \F_{\beta} \to E_{\beta}$.
The following theorem establishes a geometric realization of
the quiver Hecke algebra $H_{J}(\beta)$.

\begin{Thm}[{Varagnolo-Vasserot~\cite[Theorem 3.6]{VV11}}]
There is an isomorphism
\begin{equation}
\label{Eq:VV}
H_{J}(\beta) \cong \Ext^{*}_{G_{\beta}}(\mathscr{L}_{\beta}, \mathscr{L}_{\beta})
\end{equation}
of $\Z$-graded $\kk$-algebras.
\end{Thm}

After completing the above isomorphism \eqref{Eq:VV}, 
we obtain 
\begin{equation}
\label{Eq:cVV}
\hH_{J}(\beta) \cong \Ext^{*}_{G_{\beta}}(\mathscr{L}_{\beta}, \mathscr{L}_{\beta})^{\wedge} 
\cong \hK^{G_{\beta}}(\mathcal{Z}_{\beta})_{\kk},
\end{equation}
where the second isomorphism is given by 
the equivariant Chern character map (see~\cite[Corollary 3.9]{Fujita18} for details).

Now we give a geometric interpretation of the bimodule $\widehat{\sV}^{\otimes \beta}$.
\begin{Thm}
\label{Thm:geom}
With the above notation,
there exists an isomorphism
$$
\widehat{\sV}^{\otimes \beta} \cong \hK^{G_{\beta}}(\Mg(D_{\beta})\times_{E_{\beta}}\F_{\beta})_{\kk}
$$
which makes the following diagram commute
$$
\xy
\xymatrix{
U_{q}(L\g) 
\ar[r]
\ar[d]^-{\widehat{\Psi}_{D_{\beta}}}
&
\End\left(\widehat{\sV}^{\otimes \beta}\right)
\ar[d]^-{\cong}
&
\widehat{H}_{J}(\beta)^{\mathrm{op}}
\ar[l]
\ar[d]^-{\cong}_-{\eqref{Eq:cVV}}
\\
\hK^{G_{\beta}}(Z^{\bullet}(D_{\beta}))_{\kk}
\ar[r]
&
\End \left( \hK^{G_{\beta}}(\Mg(D_{\beta}) \times_{E_{\beta}} 
\F_{\beta})_{\kk}
\right)
&
\hK^{G_{\beta}}(\mathcal{Z}_{\beta})_{\kk}^{\mathrm{op}},
\ar[l]
}
\endxy
$$
where the first and second rows
denote the structure homomorphisms of the bimodules \eqref{Eq:KKK_bimodule} and \eqref{Eq:geom_bimodule} respectively.  
\end{Thm}
We omit a proof because
one can prove the assertion along the same lines as 
the proof of~\cite[Theorem 1.1]{Fujita18}, 
which is a spacial case (see Example~\ref{Ex:CQ} below). 

\begin{Ex}
\label{Ex:CQ}
Let $Q$ be a Dynkin quiver of type $\g$.
We take $J = I$ and  
define an injective map $\sx \colon J=I \hookrightarrow \Delta_{0}$
by 
$
\sx(i) \seq \cH_{Q}^{-1}(S_{i})
$
for each $i \in I$.
This is the case
Kang-Kashiwara-Kim considered in~\cite{KKK15}.
Then, we have $\Gamma_{J} = Q$ and
hence $\g_{J} = \g$.
Moreover
the embedding $\Mg_{0}(D) \hookrightarrow \rep_{D}(Q)$
in Corollary~\ref{Cor:KS}
becomes an isomorphism for any $I$-graded vector space $D$
(see~\cite[Theorem 9.11]{HL15}). 
Theorem~\ref{Thm:geom} for this special case
was established in~\cite{Fujita18}.
In this case,
we can further prove (see~\cite{Fujita17, Fujita18}) that the corresponding KKK-functor 
$\mathscr{F}_{J}$
induces an equivalence of monoidal categories
$
\bigoplus_{\beta \in \cQ^{+}}\widehat{H}_{J}(\beta) \modcat \simeq \Cc_{Q}, 
$
where $\Cc_{Q}$ is the monoidal full subcategory of $\Cc$
introduced by Hernandez-Leclerc~\cite{HL15},
consisting of modules whose composition factors 
are isomorphic to $L(m)$ for some dominant monomial $m$
in variables $Y_{x}$ labeled by $x \in \Delta_{0}$ such that $\cH_{Q}(x) \in \C Q \modcat \subset \dD_{Q}$.
\end{Ex}

\subsection{Type $\mathsf{A}$ subquivers and graded nilpotent orbits}
\label{Ssec:gnilp}

In this subsection, we study some examples of the KKK-functors
when the corresponding graded quiver varieties $\Mg_{0}(D_{\beta})$
are isomorphic to graded nilpotent orbits of type $\mathsf{A}$. 

Let $Q$ be a Dynkin quiver of type $\g$ and 
fix an integer $N$ such that $1 \le N-1 \le n$.
We assume that the full subquiver $Q^{\prime}$ of $Q$
supported on the subset
$
I^{\prime} \seq \{ 1, 2, \ldots, N-1\} \subset I=\{1, 2, \ldots, n\}
$
is of type $\mathsf{A}_{N-1}$ with a monotone orientation, i.e.
$$
Q^{\prime}=
\left(
\begin{xy}
\def\objectstyle{\scriptstyle}
\ar@{->} *+!D{1} *\cir<2pt>{};
(7,0) *+!D{2} *\cir<2pt>{}="A", 
\ar@{->} "A"; (14,0) *+!D{3} *\cir<2pt>{}="B",
\ar@{->} "B"; (20,0),
\ar@{.} (21,0); (23,0),
\ar@{->} (24,0); (30,0) *+!D{N-1} *\cir<2pt>{}
\end{xy}
\hspace{6pt}
\right)
\subset Q.
$$
Note that we have a natural fully faithful embedding
$\varepsilon \colon \dD_{Q^{\prime}} \hookrightarrow \dD_{Q}$
of triangulated categories.
We fix a height function $\xi$ as in Section~\ref{Ssec:Dynkin}.
The restriction of $\xi$ to the subset $I^{\prime}$
gives a height function for $Q^{\prime}$.
With these choices, we have the equivalences
$\cH_{Q} \colon \C(\Delta) \simeq \ind(\dD_{Q})$ 
and $\cH_{Q^{\prime}} \colon \C(\Delta^{\prime}) \simeq \ind(\dD_{Q^{\prime}})$
in Theorem~\ref{Thm:Happel},
where $\Delta^{\prime}$ denotes the counterpart of $\Delta$ for the subquiver $Q^{\prime}$.

Let $J \seq \Z$. With the above notation, 
we define an injective map $\sx \colon J \hookrightarrow \Delta_{0}$ by
$\sx(j) \seq (\cH_{Q}^{-1} \circ \varepsilon \circ \cH_{Q^{\prime}}) (1, \xi_{1} - 2j +2)$
for each $j \in J$, or equivalently, we define
\begin{equation}
\label{Eq:defx}
\sx(j) \seq \begin{cases}
\cH_{Q}^{-1} \left( S_{i} [-2k] \right) & \text{if $j = i + kN, 1 \le i < N, k\in \Z$};\\
\cH_{Q}^{-1} \left( M_{\theta} [-2k+1] \right) & \text{if $j = kN, k \in \Z$},
\end{cases}  
\end{equation} 
where $\theta \seq \sum_{i=1}^{N-1} \alpha_{i} \in \cR^{+}$.

\begin{Lem}
\label{Lem:Ainfty}
The quiver $\Gamma_{J}$ associated with the injective map $\sx \colon J \hookrightarrow \Delta_{0}$
given by \eqref{Eq:defx} is equal to the quiver of type $\mathsf{A}_{\infty}$ with a monotone orientation, 
namely
$$
\Gamma_{J} = \left(
\begin{xy}
\def\objectstyle{\scriptstyle}
\ar@{->} (0,0) *+!D{-2} *\cir<2pt>{}="A";(7,0) *+!D{-1} *\cir<2pt>{}="B",
\ar@{->} "B";(14,0) *+!D{0} *\cir<2pt>{}="C",
\ar@{->} "C";(21,0) *+!D{1} *\cir<2pt>{}="D",
\ar@{->} "D";(28,0) *+!D{2} *\cir<2pt>{}="E",
\ar@{->} (-6,0); "A",
\ar@{.} (-9,0);(-7,0),
\ar@{->} "E";(34,0),
\ar@{.} (35,0);(37,0)
\end{xy}
\hspace{6pt}
\right).
$$
\end{Lem}
\begin{proof}
By Lemma~\ref{Lem:key}, it suffices to prove that
\begin{equation}
\label{Eq:whattoprove}
\dim \Ext^{1}_{\dD_{Q}}\left(\cH_{Q}(\sx(j)), \cH_{Q}(\sx(j^{\prime}))\right) = 
\begin{cases}
1 & \text{if $j^{\prime} = j+1$}; \\
0 & \text{otherwise}.
\end{cases}
\end{equation}
Using the fully faithful embedding $\varepsilon \colon \dD_{Q^{\prime}} \hookrightarrow \dD_{Q}$,
we can reduce the situation to the special case $Q^{\prime} = Q$. 
In this case, 
we can easily check \eqref{Eq:whattoprove}
since 
the $\C Q$-module $M_{\theta}$
is a projective cover of the simple module $S_{1}$ and 
at the same time it is an injective hull of the simple module $S_{N-1}$.
\end{proof}

By Lemma~\ref{Lem:Ainfty},
the Kac-Moody algebra $\g_{J}$ is of type $\mathsf{A}_{\infty}$.
Let $\cR_{J}^{+}$ denote the set of positive roots of $\g_{J}$,
which is given by 
$$
\cR_{J}^{+} = \{ \alpha(j; \ell) \in \cQ_{J}^{+} \mid
j \in J, \ell \in \Z_{\ge 1}\}, 
\quad \text{where $\alpha(j; \ell) \seq \sum_{k=0}^{\ell -1} \alpha^{J}_{j+k}$}.
$$ 
We also consider the subsets  
$$
\cR_{J, N}^{+}\seq \{ \alpha(j; N) \mid j \in J\}, 
\quad
\cR_{J, \le N}^{+} \seq \{ \alpha(j; \ell) \mid
j \in J, 1 \le \ell \le N \}.
$$
For a fixed element $\beta = \sum_{j \in J} d_{j} \alpha_{j}^{J} \in \cQ_{J}^{+}$, we define 
a finite set 
$$
\KP(\beta) \seq\left\{ \nu = (\nu_{\alpha}) \in (\Z_{\ge 0})^{\cR_{J}^{+}} \; \middle| \; 
\textstyle \sum_{\alpha \in \cR_{J}^{+}} \nu_{\alpha} \alpha = \beta \right\}.
$$
An element $\nu$ of $\KP(\beta)$
is called a {\em Kostant partition} of $\beta$.
We also consider the subset 
$$
\KP_{\le N}(\beta) \seq \{ \nu \in \KP(\beta) \mid \text{$\nu_{\alpha} =0$ unless $\alpha \in \cR_{J, \le N}^{+}$}\}.
$$
Let $D_{\beta} = \bigoplus_{j \in J} (D_{\beta})_{j}$ be a $J$-graded vector space
such that $\dim (D_{\beta})_{j} = d_{j}$ for each $j\in J$ as in the previous subsection.
We set 
$$ E_{\beta} \seq \rep_{D_{\beta}}(\Gamma_{J}) = \prod_{j \in J}\Hom_{\C}((D_{\beta})_{j}, (D_{\beta})_{j+1})$$
and regard it as a $G_{\beta}$-stable closed subvariety of $\mathfrak{gl}(D_{\beta})$. 
A $G_{\beta}$-orbit in $E_{\beta}$ 
can be realized as a component of a certain $\C^{\times}$-fixed locus of a nilpotent orbit of $\mathfrak{gl}(D_{\beta})$ 
and hence called a {\em graded nilpotent orbit}.
By Gabriel's theorem~\cite{Gabriel72}, the set of $G_{\beta}$-orbits in $E_{\beta}$
is in bijection with the set $\KP(\beta)$.
For an element $\nu \in \KP(\beta)$, 
the corresponding $G_{\beta}$-orbit $\fO_{\nu}$ 
contains the $\C \Gamma_{J}$-module
$\bigoplus_{\alpha \in \cR_{J}^{+}} (M^{J}_{\alpha})^{\oplus \nu_{\alpha}}$,
where
$M^{J}_{\alpha}$ denotes the unique indecomposable $\C \Gamma_{J}$-module
of dimension vector $\alpha \in \cR_{J}^{+}$.  

Let $A^{\prime}$ be the quotient of the path algebra 
$\C \Gamma_{J}$ by the ideal generated by all the paths of length $\ge N$.
We consider the $G_{\beta}$-stable closed subvariety of $E_{\beta}$ 
$$
E_{\beta}^{\prime} \seq \rep_{D_{\beta}}(A^{\prime}) = \{ X \in E_{\beta} \mid X^{N} = 0 \}. 
$$
We can naturally regard the category $A^{\prime} \modcat$ as the full subcategory 
of $\C \Gamma_{J} \modcat$ consisting of modules isomorphic to 
direct sums of $M^{J}_{\alpha}$ for various $\alpha \in \cR^{+}_{J, \le N}$.
Thus the variety $E_{\beta}^{\prime}$ is a union of $G_{\beta}$-orbits
$\fO_{\nu}$ with $\nu \in \KP_{\le N}(\beta)$.

As explained in~\cite[Chapter II.2.6(a)]{Happel88}, 
the algebra $A^{\prime}$ coincides with the repetitive algebra of $\C Q^{\prime}$.
Since it is self-injective, the category $A^{\prime} \modcat$ is a Frobenius category 
and the set $\{ M^{J}_{\alpha} \mid \alpha \in \cR^{+}_{N} \}$ forms 
a complete collection of indecomposable projective $A^{\prime}$-modules.
We denote by $A^{\prime}\udmod$
the stable category of $A^{\prime} \modcat$.

\begin{Thm}[Happel \cite{Happel88}]
There exists a $\delta$-functor 
$\Phi^{\prime} \colon A^{\prime} \modcat \to \dD_{Q^{\prime}}$
which satisfies
$$
\Phi^{\prime}(M^{J}_{\alpha(j ; \ell)}) \cong 
\begin{cases}
\cH_{Q^{\prime}}(\ell, \xi_{\ell} - 2j +2 ) & \text{if $1 \le \ell < N$};\\
0 & \text{if $\ell = N$}
\end{cases}
$$
for each $j \in J$ and $1\le \ell \le N$, 
and induces a triangle equivalence $A^{\prime}\udmod \simeq \dD_{Q^{\prime}}$. 
\end{Thm}
\begin{proof}
Apply the general theory \cite[Theorem II.4.9]{Happel88}
of the repetitive algebras.
\end{proof}

To each $\alpha \in \cR_{J, \le N}^{+}$, we assign 
an element $\sx(\alpha) \in \Delta_{0} \sqcup \{ 0 \}$ by
$$
\sx(\alpha) \seq (\cH_{Q}^{-1} \circ \varepsilon \circ \Phi^{\prime})(M^{J}_{\alpha}).
$$
Note that we have
$\sx(\alpha^{J}_{j}) = \sx(j)$ for each $j \in J$ by definition.

Now we state the main theorem of this subsection.
\begin{Thm}
\label{Thm:gnilp}
For any $\beta \in \cQ_{J}^{+}$,
the following assertions hold:
\begin{enumerate}
\item \label{Thm:gnilp:isom}
The closed embedding $\Mg_{0}(D_{\beta}) \hookrightarrow E_{\beta}$
of Corollary~\ref{Cor:KS} 
induces an isomorphism of $G_{\beta}$-varieties
\begin{equation}
\label{Eq:gnilp}
\Mg_{0}(D_{\beta}) \cong E_{\beta}^{\prime}.
\end{equation}
\item \label{Thm:gnilp:strata}
Under the isomorphism~\eqref{Eq:gnilp},
each non-empty stratum $\Mreg(V, D_{\beta})$
coincides with a single $G_{\beta}$-orbit $\fO_{\nu}$
associated with the Kostant partition $\nu \in \KP_{\le N}(\beta)$
determined by the relation
\begin{equation}
\label{Eq:YA=mnu}
Y^{D_{\beta}} A^{-V} = m_{\nu} \seq \prod_{\alpha \in \cR_{J, \le N}^{+}} Y_{\sx(\alpha)}^{\nu_{\alpha}},
\end{equation}
where we set $Y_{\sx(\alpha)} = Y_{0} \seq 1$ for $\alpha \in \cR_{J, N}^{+}$.
\end{enumerate}
\end{Thm}

For a proof of Theorem~\ref{Thm:gnilp}~\eqref{Thm:gnilp:isom},
we need the following lemma. 

\begin{Lem}
\label{Lem:technical}
We define a subset $C \subset \Delta_{0}$ by
\begin{equation}
\label{Eq:defC}
C \seq \sx(J) \sqcup \{ \cH_{Q}^{-1}(S_{i}[k]) \mid i \in I \setminus I^{\prime}, k \in \Z \}.
\end{equation}
Then the following assertions hold.
\begin{enumerate}
\item \label{Lem:technical:C}
The subset $C$ satisfies the condition $\mathrm{(R)}$ in Theorem~\ref{Thm:KS}. 
\item \label{Lem:technical:path}
For any $i, j \in J$, we have
\begin{equation}
\label{Eq:path}
\dim \left(e_{\sx(i)} \cdot \C \Gamma|_{C} \cdot e_{\sx(j)} \right) =
\begin{cases}
1 & i \ge j; \\
0 & i < j. 
\end{cases} 
\end{equation}
\end{enumerate}
\end{Lem}
\begin{proof}
Let $I \setminus I^{\prime} = I_{1} \sqcup \cdots \sqcup I_{b}$
$(b \in \Z_{\ge 0})$
be a decomposition such that 
the full subquiver $Q|_{I_{k}}$ is a connected component 
of $Q|_{I \setminus I^{\prime}}$ for each $1 \le k \le b$.
Since the Dynkin graph is a tree, 
there exist unique $i_{k} \in I^{\prime}$
and $j_{k} \in I_{k}$ satisfying $i_{k} \sim j_{k}$ for each $1 \le k \le b$.
After reordering if necessary, we may assume that 
there exists $0 \le b_{1} \le b$
such that we have $i_{k} \leftarrow j_{k}$ for $1 \le k \le b_{1}$
and $i_{k} \to j_{k}$ for $b_{1} < k \le b$.
We put $I^{\circ} \seq \bigsqcup_{1 \le k \le b_{1}} I_{k}$
and $I^{\bullet} \seq \bigsqcup_{b_{1} < k \le b} I_{k}$. 

To verify the assertion (1),
it suffices to prove that for any $K \in \ind \dD_{Q}$
there exists $L \in \cH_{Q}(C)$
such that $\Hom_{\dD_{Q}}(K, L) \neq 0$.
Since the set $\cH_{Q}(C)$ is stable under even degree shifts,
we may assume that $K \cong M_{\alpha}$ or $K \cong M_{\alpha}[1]$
for some $\alpha \in \cR^{+}$.
When $K \cong M_{\alpha}$, 
we just take a simple quotient $M_{\alpha} \twoheadrightarrow S_{i}$
and find $\Hom_{\dD_{Q}}(K, L) \neq 0$ with $L = S_{i} \in \cH_{Q}(C)$. 
When $K \cong M_{\alpha}[1]$, we encounter the following two cases.

{Case 1}: $(\alpha, \varpi_{j}) \neq 0$ for some $j \in I^{\circ}$.
In this case, the subspace 
$M^{\prime} \seq \bigoplus_{i \in I^{\prime} \sqcup I^{\bullet}} e_{i}M_{\alpha}$
is a submodule of $M_{\alpha}$ and the quotient $M_{\alpha} / M^{\prime}$
is non-zero. Therefore there is $j \in I^{\circ}$ such that
$\Hom_{\C Q}(M_{\alpha} / M^{\prime}, S_{j}) \neq 0$.
This implies that $\Hom_{\dD_{Q}}(K, L) \neq 0$ with $L = S_{j}[1] \in \cH_{Q}(C)$.

{Case 2}: $(\alpha, \varpi_{j}) = 0$ for all $j \in I^{\circ}$.
In this case, $M_{\alpha}$ can be regarded as a representation of 
the full subquiver $Q|_{I^{\prime} \sqcup I^{\bullet}}$.
First, we assume $(\alpha, \varpi_{N-1}) \neq 0$. 
Then we have
$\Hom_{\C Q}(M_{\alpha}, M_{\theta}) \neq 0$
because $M_{\theta}$ is an injective hull of 
the simple module $S_{N-1}$
in the category $(\C Q|_{I^{\prime} \sqcup I^{\bullet}}) \modcat$. 
Thus we obtain $\Hom_{\dD_{Q}}(K, L) \neq 0$ with $L = M_{\theta}[1] \in \cH_{Q}(C)$.
Next, we assume $(\alpha, \varpi_{N-1}) = 0$ and $(\alpha, \varpi_{i}) \neq 0$
for some $1 \le i < N-1$. Let us take such an $i$ as large as possible.
Then we see that $\alpha^{\prime} \seq \alpha + \alpha_{i+1}$ is a positive root and there is
 a non-trivial extension $0 \to S_{i+1} \to M_{\alpha^{\prime}} \to M_{\alpha} \to 0$.
Therefore we obtain 
$\Ext_{\C Q}^{1}(M_{\alpha}, S_{i+1}) \neq 0$. Noting that $i+1 \in I^{\prime}$,
we get
$\Hom_{\dD_{Q}}(K, L) \neq 0$ with $L = S_{i+1}[2] \in \cH_{Q}(C)$.
Finally, we assume $(\alpha, \varpi_{i}) = 0$ for all $i \in I^{\prime}$.
Then $M_{\alpha}$ is supported on $I^{\bullet}$
and hence $\Hom_{\dD_{Q}}(K, L) \neq 0$ with $L=S_{j}[1]$ for some $j \in I^{\bullet}$.
    
Next we shall prove the assertion (2).
We note that the LHS of \eqref{Eq:path}
is equal to the number of (oriented) paths from $\sx(j)$ to $\sx(i)$ 
in the quiver $\Gamma|_{C}$.
Since $\Gamma_{J} = \Gamma|_{\sx(J)} \subset \Gamma|_{C}$,
there is at least one path from $\sx(j)$ to $\sx(i)$ when
$i \ge j$.
On the other hand, we know that the quiver $\Gamma$  
has neither loops nor oriented cycles by Proposition~\ref{Prop:Ext=tC}.
In particular, there are no paths from $\sx(j)$ to $\sx(i)$ in $\Gamma|_{C}$ when $i < j$.
Thus, we only have to prove that there are no two different paths
from $\sx(j)$ to $\sx(i)$ when $j < i$.
To see this, we divide the arrows of the quiver $\Gamma |_{C}$ 
into the following seven types:  
\begin{itemize}
\item[(i)] the arrows of the subquiver $\Gamma_{J} = \Gamma|_{\sx(J)} \subset \Gamma|_{C}$;
\item[(ii)] $\cH_{Q}^{-1}(S_{j_k}[2 \ell]) \to \cH_{Q}^{-1}(S_{i_k}[2\ell])$ for any $1 \le k  \le b_{1}, \ell \in \Z$; 
\item[(iii)] $\cH_{Q}^{-1}(S_{i_k}[2\ell]) \to \cH_{Q}^{-1}(S_{j_k}[2\ell])$ for any $b_{1} < k \le b, \ell \in \Z$;
\item[(iv)] $\cH_{Q}^{-1}(S_{j}[\ell]) \to \cH_{Q}^{-1}(S_{j}[\ell - 1])$ for any $j \in I \setminus I^{\prime}, \ell \in \Z$;
\item[(v)] $\cH_{Q}^{-1}(S_{j}[\ell]) \to \cH_{Q}^{-1}(S_{j^{\prime}}[\ell])$ for some $j, j^{\prime} \in I_{k}, 1 \le k \le b, \ell \in \Z$; 
\item[(vi)] $\cH_{Q}^{-1}(S_{j_k}[2\ell + 1]) \to \cH_{Q}^{-1}(M_{\theta}[2 \ell +1])$ for any $1 \le k \le b_{1}, \ell \in \Z$;
\item[(vii)] $\cH_{Q}^{-1}(M_{\theta}[2\ell + 1]) \to \cH_{Q}^{-1}(S_{j_k}[2 \ell +1])$ for any $b_{1} < k \le b, \ell \in \Z$.
\end{itemize}
A path in $\Gamma|_{C}$ 
going out from $\sx(J)$ should contain an arrow of type (iii) or (vii),
and hence go through a vertex belonging to the set
$S^{\bullet} \seq \{ \cH_{Q}^{-1}(S_{j}[\ell]) \mid j \in I^{\bullet}, \ell \in \Z \}$.
On the other hand, a path in $\Gamma|_{C}$ coming into $\sx(J)$
should contain an arrow of type (ii) or (vi), 
and hence go through a vertex belonging to the set 
$S^{\circ} \seq \{ \cH_{Q}^{-1}(S_{j}[\ell]) \mid j \in I^{\circ}, \ell \in \Z \}$.    
However, there are no paths in $\Gamma|_C$ from a vertex of $S^{\bullet}$
to a vertex of $S^{\circ}$. Therefore
there are no paths in $\Gamma|_{C}$
connecting two different vertices of $\sx(J)$ other than the paths in $\Gamma_{J}$.  
\end{proof}

\begin{proof}[Proof of Theorem~\ref{Thm:gnilp}~\eqref{Thm:gnilp:isom}]
Thanks to Lemma~\ref{Lem:technical}~\eqref{Lem:technical:C},
we can apply Theorem~\ref{Thm:KS} and Remark~\ref{Rem:KS} to the algebra $\Lambda_{C}$
associated with the subset $C \subset \Delta_{0}$ given by \eqref{Eq:defC}.
Thus, there exists an admissible ideal $\mathfrak{J}_{C}$
of the path algebra $\C \Gamma|_{C}$
such that $\Lambda_{C} \cong (\C \Gamma|_{C}) / \mathfrak{J}_{C}$
and hence $\Mg_{0}(D_{\beta}) = \rep_{D_{\beta}}(\Lambda_{C}) \cong \rep_{D_{\beta}}( (\C \Gamma|_{C}) / \mathfrak{J}_{C} )$.
Now, we need to prove
$\rep_{D_{\beta}}( (\C \Gamma|_{C}) / \mathfrak{J}_{C} ) = E_{\beta}^{\prime}$.
Thanks to Lemma~\ref{Lem:technical}~\eqref{Lem:technical:path},
it suffices to show that 
$e_{\sx(j+N)} (\mathfrak{J}_{C}) e_{\sx(j)} \neq 0$
and 
$e_{\sx(j + i)} (\mathfrak{J}_{C}) e_{\sx(j)}  = 0$
for any $j \in J$ and $1 \le i < N$.
By Theorem~\ref{Thm:KS}, these conditions can be verified 
by checking the following two homological properties:
\begin{align}
\Ext_{\dD_{Q}}^{2}\left(\cH_{Q}(\sx(j)), \cH_{Q}(\sx(j+N))\right) & \neq 0, \label{Eq:Ext2N} \\
\Ext_{\dD_{Q}}^{2}(\cH_{Q}(\sx(j)), \cH_{Q}(\sx(j+i))) & = 0 \label{Eq:Ext2i}
\end{align} 
for all $j \in J$ and $1 \le i < N$.
The property \eqref{Eq:Ext2N} follows because we have
$\cH_{Q}(\sx(j+N)) = \cH_{Q}(\sx(j))[-2]$ by definition.
We can prove the property \eqref{Eq:Ext2i} easily 
by using the fact that $M_{\theta}$ is both projective and injective
in the subcategory $\C Q^{\prime} \modcat \subset \C Q \modcat$.
\end{proof}

For a proof of Theorem~\ref{Thm:gnilp}~\eqref{Thm:gnilp:strata},
we need the following lemma.

\begin{Lem}
\label{Lem:DoreyA}
Let $\ell^{\prime}, \ell^{\pprime}$ be two positive integers 
such that $\ell \seq \ell^{\prime} + \ell^{\pprime} \le N$.
Fix $j \in J$ and set 
$\alpha^{\prime} \seq \alpha(j; \ell^{\prime}), 
\alpha^{\pprime} \seq \alpha(j+\ell^{\prime}; \ell^{\pprime}),
\alpha \seq \alpha(j; \ell^{\prime} + \ell^{\pprime}) = \alpha^{\prime} + \alpha^{\pprime}$.
Then, there is an injective $U_{q}(L\g)$-homomorphism
$$
L(Y_{\sx(\alpha)}) \hookrightarrow L(Y_{\sx(\alpha^{\prime})}) \otimes L(Y_{\sx(\alpha^{\pprime})}).
$$
\end{Lem}
\begin{proof}
Under the assumption, 
the functor $\varepsilon \circ \Phi^{\prime} \colon A^{\prime} \modcat \to \dD_{Q}$
sends a non-split short exact sequence
$
0 \to M^{J}_{\alpha^{\pprime}} \to M^{J}_{\alpha} \to M^{J}_{\alpha^{\prime}} \to 0
$
in $A^{\prime} \modcat$
to a non-split exact triangle 
$
\cH_{Q}(\sx(\alpha^{\pprime})) \to \cH_{Q}(\sx(\alpha)) \to \cH_{Q}(\sx(\alpha^{\prime})) \xrightarrow{+1}
$
in $\dD_{Q}$,
where we understand $\cH_{Q}(\sx(\alpha)) = 0$ when $\alpha \in \cR_{J, N}^{+}$,
or equivalently $\ell = N$.
Moreover, it induces an isomorphism of $1$-dimensional vector spaces:
$$
\Ext^{1}_{A^{\prime}}(M^{J}_{\alpha^{\prime}}, M^{J}_{\alpha^{\pprime}})  
\cong \Ext^{1}_{\dD_{Q}}(\cH_{Q}(\sx(\alpha^{\prime})), \cH_{Q}(\sx(\alpha^{\pprime}))).
$$
Applying Proposition~\ref{Prop:Dorey}, 
we obtain a short exact sequence in $\Cc$:
$$0 \to
L(m_{[\sx(\alpha^{\prime}), \sx(\alpha^{\pprime})]}) \to
L(Y_{\sx(\alpha^{\prime})}) \otimes L(Y_{\sx(\alpha^{\pprime})})
\to L(Y_{\sx(\alpha^{\prime})} Y_{\sx(\alpha^{\pprime})})
\to 0
$$
with $m_{[\sx(\alpha^{\prime}), \sx(\alpha^{\pprime})]}=Y_{\sx(\alpha)}$. 
\end{proof}

\begin{proof}[Proof of Theorem~\ref{Thm:gnilp}~\eqref{Thm:gnilp:strata}]
First, we note that the assignment
$\KP_{\le N}(\beta) \ni \nu \mapsto m_{\nu} \in \Mm^{+}$
is injective for each fixed element $\beta \in \cQ_{J}^{+}$.

Since a non-empty stratum $\Mreg(V, D_{\beta})$
is $G_{\beta}$-stable, 
it is a union of $G_{\beta}$-orbits.
In particular, the number of non-empty strata $\Mreg(V, D_{\beta})$
is less than or equal to 
the number of $G_{\beta}$-orbits in $E_{\beta}^{\prime}$,
which is
$\# \KP_{\le N}(\beta)$.
On the other hand,
we apply Lemma~\ref{Lem:DoreyA} repeatedly 
to find
that $c(Y^{D_{\beta}}, m_{\nu}) \neq 0$
for all $\nu \in \KP_{\le N}(\beta)$.   
By Theorem~\ref{Thm:stratification}~\eqref{Thm:stratification:nonempty},
this implies that the number of non-empty strata 
is not less than $\# \KP_{\le N}(\beta)$.
Therefore each non-empty stratum consists of a single $G_{\beta}$-orbit.

We shall prove the relation~\eqref{Eq:YA=mnu}.
Recall the stratifying functor 
$\Phi_{Q} \colon \Lambda \modcat \to \dD_{Q}$
in Theorem~\ref{Thm:KS2}.
Thanks to Theorem~\ref{Thm:gnilp}~\eqref{Thm:gnilp:isom},
we can identify the category $A^{\prime} \modcat$
with the full subcategory of $\Lambda \modcat$
consisting of modules supported on the subset 
$\sx(J) \subset \Delta_{0}$.
By Theorem~\ref{Thm:KS2}~\eqref{Thm:KS2:2},
the relation~\eqref{Eq:YA=mnu} holds if and only if
there is  an isomorphism
$$
\bigoplus_{\alpha \in \cR_{J, \le N}^{+}} \Phi_{Q}(M_{\alpha}^{J})^{\oplus \nu_{\alpha}}
\cong \bigoplus_{\alpha \in \cR_{J, \le N}^{+}}(\varepsilon \circ \Phi^{\prime})(M_{\alpha}^{J})^{\oplus \nu_{\alpha}}.
$$
Thus, it suffices to prove the relation~\eqref{Eq:YA=mnu}
for the special case when $\beta$ is a positive root $\alpha \in \cR_{J, \le N}^{+}$ and $\nu$
is the Kostant partition $\delta(\alpha) \in \KP_{\le N}(\alpha)$ 
given by $\delta(\alpha)_{\alpha^{\prime}} = \delta_{\alpha, \alpha^{\prime}}$ 
for each $\alpha^{\prime} \in \cR_{J, \le N}^{+}$.
Now we concentrate on this special case.
As in the previous paragraph, we have $c(Y^{D_{\alpha}}, Y_{\sx(\alpha)}) \neq 0$
and hence there is a non-empty stratum $\Mreg(V, D_{\alpha})$ such that
$Y^{D_{\alpha}}A^{-V} = Y_{\sx(\alpha)}$.
Note that $Y_{\sx(\alpha)}$ is a minimal element of $\Mm^{+}$
with respect to the partial ordering $\le$,
which implies that
$\Mreg(V, D_{\alpha})$ is a maximal stratum with respect to the closure ordering
by Theorem~\ref{Thm:stratification}~\eqref{Thm:stratification:ordering}.
On the other hand, 
$E^{\prime}_{\alpha}$ is an affine space
and 
$\fO_{\delta(\alpha)}$
is the unique open dense $G_{\alpha}$-orbit of $E^{\prime}_{\alpha}$.
Therefore we have
$\Mreg(V, D_{\alpha}) = \fO_{\delta(\alpha)}$.
\end{proof}

Let $\mathscr{L}_{\beta}$ be
the push-forward of the constant perverse $\kk$-sheaf
along the proper morphism
$\mu_{\beta} \colon \F_{\beta} \to E_{\beta}$
as in the previous subsection.
By the decomposition theorem, 
we have
\begin{equation}
\label{Eq:decF}
\mathscr{L}_{\beta} \cong \bigoplus_{\nu \in \KP(\beta)} IC(\fO_{\nu}, \kk) \otimes_{\kk} L_{\nu},
\end{equation}
where each $L_{\nu} \in D^{b}(\kk \modcat)$ 
is a finite-dimensional $\Z$-graded $\kk$-vector space   
which is self-dual.
Via the isomorphism \eqref{Eq:VV},
each vector space $L_{\nu}$ is equipped with a structure of
graded $H_{J}(\beta)$-module.
It is known that 
we have $L_{\nu} \neq 0$ for all $\nu \in \KP(\beta)$
and 
the set $\{ L_{\nu} \mid \nu \in \KP(\beta) \}$ 
forms a complete collection of self-dual simple objects
in the category $H_{J}(\beta) \gmod$
(see~\cite[Corollary 2.8]{Kato14} for instance). 

On the $U_{q}(L\g)$-side, we have the following homomorphisms of $\kk$-algebras
$$
U_{q}(L\g) \xrightarrow{\widehat{\Psi}_{D_{\beta}}}
\hK^{G_{\beta}}(Z^{\bullet}(D_{\beta}))_{\kk}
\cong 
\Ext_{G_{\beta}}^{*}(\mathscr{L}^{\bullet}_{\beta}, \mathscr{L}^{\bullet}_{\beta})^{\wedge},
$$
where $\mathscr{L}^{\bullet}_{\beta}$ is the push-forward of the constant $\kk$-sheaf on $\Mg(D_{\beta})$
along the $G_{\beta}$-equivariant proper morphism 
$\pi^{\bullet} \colon \Mg(D_{\beta}) \to \Mg_{0}(D_{\beta}) = E_{\beta}^{\prime}$
(see~\cite[Corollary 3.16]{Fujita18}).
By Theorem~\ref{Thm:stratification}~\eqref{Thm:stratification:IC}
and Theorem~\ref{Thm:gnilp}~\eqref{Thm:gnilp:strata},
the complex $\mathscr{L}^{\bullet}_{\beta}$ decomposes as: 
\begin{equation}
\label{Eq:decM}
\mathscr{L}^{\bullet}_{\beta} \cong \bigoplus_{\nu \in \KP_{\le N}(\beta)}
IC(\fO_{\nu}, \kk) \otimes_{\kk} L^{\bullet}_{\nu},
\end{equation}
where each $L^{\bullet}_{\nu}$ is a non-zero finite-dimensional 
$\Z$-graded $\kk$-vector space, 
which has a natural structure of 
a simple module over the algebra 
$\Ext_{G_{\beta}}^{*}(\mathscr{L}^{\bullet}_{\beta}, \mathscr{L}^{\bullet}_{\beta})^{\wedge}$.
Moreover, by~\cite[Theorem 14.3.2(3)]{Nakajima01} and the relation \eqref{Eq:YA=mnu},
we have $(\widehat{\Psi}_{D_{\beta}})^{*}(L^{\bullet}_{\nu}) \cong L(m_{\nu})$ as 
$U_{q}(L\g)$-modules.

Let us consider a natural bimodule given by the Yoneda products: 
\begin{equation}
\label{Eq:bimodExt}
\Ext_{G_{\beta}}^{*}(\mathscr{L}^{\bullet}_{\beta}, \mathscr{L}^{\bullet}_{\beta}) 
\quad \curvearrowright
\quad \Ext_{G_{\beta}}^{*}(\mathscr{L}_{\beta}, \mathscr{L}^{\bullet}_{\beta})
\quad \curvearrowleft
\quad \Ext_{G_{\beta}}^{*}(\mathscr{L}_{\beta}, \mathscr{L}_{\beta}).
\end{equation} 
Comparing the decompositions \eqref{Eq:decF} and \eqref{Eq:decM},
we see that the functor 
$$
\Ext_{G_{\beta}}^{*}(\mathscr{L}_{\beta}, \mathscr{L}_{\beta}) \gmod
\to 
\Ext_{G_{\beta}}^{*}(\mathscr{L}^{\bullet}_{\beta}, \mathscr{L}^{\bullet}_{\beta}) \gmod
$$
induced from the bimodule \eqref{Eq:bimodExt}
sends the simple module $L_{\nu}$ 
to the simple module $L^{\bullet}_{\nu}$ if $\nu \in \KP_{\le N}(\beta)$,
or zero otherwise
(see~\cite[Theorem 6.8]{GRV94} for a detailed explanation).
Moreover, after the completion,
the bimodule~\eqref{Eq:bimodExt} 
gets identified with the bimodule~\eqref{Eq:geom_bimodule},
i.e.~we have the following commutative diagram
$$
\xy
\xymatrix{
\hK^{G_{\beta}}(Z^{\bullet}(D_{\beta}))_{\kk}
\ar[r]
\ar[d]^-{\cong}
&
\End \left(
\hK^{G_{\beta}}(\Mg(D_{\beta}) \times_{E_{\beta}} \F_{\beta})_{\kk}
\right)
\ar[d]^-{\cong}
&
\hK^{G_{\beta}}(\cZ_{\beta})_{\kk}^{\mathrm{op}}
\ar[l]
\ar[d]^-{\cong}
\\
\Ext_{G_{\beta}}^{*}(\mathscr{L}_{\beta}^{\bullet},
\mathscr{L}_{\beta}^{\bullet})^{\wedge}
\ar[r]
&
\End \left(
\Ext_{G_{\beta}}^{*}(\mathscr{L}_{\beta},
\mathscr{L}^{\bullet}_{\beta})^{\wedge} \right)
&
\Ext_{G_{\beta}}^{*}(\mathscr{L}_{\beta},
\mathscr{L}_{\beta})^{\wedge \mathrm{op}}.
\ar[l]
}
\endxy
$$   
Combining the above discussion 
with Theorem~\ref{Thm:geom},
we obtain the following.

\begin{Cor}
\label{Cor:Fsimple}
The KKK-functor 
$\mathscr{F}_{J} \colon H_{J} \gmod \to \Cc$
associated with the injective map $\sx \colon J \hookrightarrow \Delta_{0}$
given by \eqref{Eq:defx}satisfies
$$
\mathscr{F}_{J}(L_{\nu}) \cong \begin{cases}
L(m_{\nu}) & \text{if $\nu \in \KP_{\le N}(\beta)$}; \\
0 & \text{otherwise},
\end{cases} 
$$
for each $\beta \in \cQ^{+}_{J}$ and $\nu \in \KP(\beta)$.
\end{Cor} 

We define a category $\Cc_{\dD_{Q^{\prime}}}$
as the full subcategory of $\Cc$
consisting of modules whose composition factors 
are isomorphic to $L(m)$ for some dominant monomial $m$
in variables $Y_{x}$ labeled by $x \in \Delta_{0}$ such that
$\cH_{Q}(x) \in \dD_{Q^{\prime}} \subset \dD_{Q}$.
Note that the above KKK-functor $\mathscr{F}_{J}$ 
is exact by \cite[Theorem 3.8]{KKK18}.
By Corollary~\ref{Cor:Fsimple}, 
the image of the 
functor $\mathscr{F}_{J}$ 
is contained in the category $\Cc_{\dD_{Q^{\prime}}}$.

In~\cite[Section 4]{KKK18},
Kang-Kashiwara-Kim introduced
a certain localization $\mathcal{T}_{N}$
of the category $H_{J} \gmod$
for each $N \in \Z_{\ge 1}$.
This is a $\Z$-graded $\kk$-linear abelian monoidal category
equipped with a canonical quotient functor $\Omega \colon H_{J} \gmod \to \mathcal{T}_{N}$
characterized by the following universal property.
Suppose that
$\mathcal{A}$ is a $\kk$-linear abelian monoidal category 
and
$F \colon H_{J} \gmod \to \mathcal{A}$
is a $\kk$-linear exact monoidal functor 
such that
\begin{itemize}
\item[(i)] $F(L_{\nu}) = 0$ for any $\nu \in \KP(\beta) \setminus \KP_{\le N}(\beta)$;
\item[(ii)] $F(L_{\delta(\alpha)}) \cong \mathbf{1}_{\mathcal{A}}$ for any $\alpha \in \cR_{J, N}^{+}$.
Here $\mathbf{1}_{\mathcal{A}}$ denotes the unit object of $\mathcal{A}$;
\item[(iii)] $F$ satisfies a certain commutativity condition 
on the tensoring operations with the modules $\{ L_{\delta(\alpha)} \mid \alpha \in \cR_{J, N}^{+} \}$
as in~\cite[Proposition A.12]{KKK18}.
\end{itemize}  
Then there exists a unique $\kk$-linear exact monoidal functor
$\widetilde{F} \colon \mathcal{T}_{N} \to \mathcal{A}$
such that we have $F \simeq \widetilde{F} \circ \Omega$.

Since the category $\mathcal{T}_{N}$ is $\Z$-graded,
its Grothendieck ring $K(\mathcal{T}_{N})$
has a structure of $\Z[v^{\pm 1}]$-algebra,
where the multiplication of $v$ is given by
the grading shift functor. 
We denote by $K(\mathcal{T}_{N})|_{v=1}$ 
the specialization $K(\mathcal{T}_{N})/ (v-1)K(\mathcal{T}_{N})$.

\begin{Cor}
\label{Cor:Gringisom}
After a suitable modification of the isomorphism  
\eqref{Eq:cVV}
for each $\beta \in \cQ_{J}^{+}$,
the above KKK-functor 
$\mathscr{F}_{J} \colon H_{J} \gmod \to \Cc_{\dD_{Q^{\prime}}}$
factors through the 
localized category $\mathcal{T}_{N}$
and yields a ring isomorphism
$$
K(\mathcal{T}_{N})|_{v=1} \cong K(\Cc_{\dD_{Q^{\prime}}}). 
$$
\end{Cor}
\begin{proof}[Sketch of proof]
The functor $\mathscr{F}_{J} \colon H_{J} \gmod \to \Cc_{\dD_{Q^{\prime}}}$
satisfies the above conditions (i) and (ii) by Corollary~\ref{Cor:Fsimple}.
By the similar argument as in~\cite[Theorem 2.6.8]{KKO19},
we can modify the isomorphism 
$\hH_{J}(\beta) \cong \hK^{G_{\beta}}(\mathcal{Z}_{\beta})_{\kk}$ 
for each $\beta \in \cQ_{J}^{+}$ to make the functor 
$\mathscr{F}_{J}$ satisfy the condition (iii) as well.
Thus, by the universal property,
the functor $\mathscr{F}_{J}$ factors through the localization $\mathcal{T}_{N}$. 
By~\cite[Proposition 4.31]{KKK18}, the set
$\{ \Omega(L_{\nu}) \mid \nu \in \KP_{\le N-1}(\beta), \beta \in \cQ_{J}^{+} \}$
forms a complete collection of self-dual simple objects of $\mathcal{T}_{N}$
and hence their classes 
give a $\Z$-basis of $K(\mathcal{T}_{N})|_{v=1}$.
On the other hand, the set $\{ L(m_{\nu}) \mid \nu \in \KP_{\le N-1}(\beta), \beta \in \cQ_{J}^{+} \}$
forms a complete collection of simple modules of $\Cc_{\dD_{Q^{\prime}}}$
and hence their classes give a $\Z$-basis of $K(\Cc_{\dD_{Q^{\prime}}})$. 
Now the desired ring isomorphism follows because the functor $\mathscr{F}_{J}$
induces a bijection between these two bases again by Corollary~\ref{Cor:Fsimple}.
\end{proof}

\begin{Rem}
In the recent paper~\cite{KKOP19} by Kashiwara-Kim-Oh-Park,
it was shown that the category $\mathcal{T}_{N}$
gives a monoidal categorification of a cluster algebra associated with
a certain infinite quiver.
Combined with Corollary~\ref{Cor:Gringisom},
we conclude that the category $\Cc_{\dD_{Q^{\prime}}}$
also gives a monoidal categorification of the same cluster algebra. 
\end{Rem}

We finish this subsection with exhibiting a couple of examples
of subquivers $Q^{\prime} \subset Q$ and 
the corresponding injective maps $\sx \colon J \hookrightarrow \Delta_{0}$.

\begin{Ex}[Type $\mathsf{A}_{n}$]
Consider the following case with $N=n+1$:
$$
Q^{\prime}=
\left(
\begin{xy}
\def\objectstyle{\scriptstyle}
\ar@{->} *+!D{1} *\cir<2pt>{};
(7,0) *+!D{2} *\cir<2pt>{}="A", 
\ar@{->} "A"; (14,0) *+!D{3} *\cir<2pt>{}="B",
\ar@{->} "B"; (20,0),
\ar@{.} (21,0); (23,0),
\ar@{->} (24,0); (30,0) *+!D{n} *\cir<2pt>{}
\end{xy}
\hspace{6pt}
\right)
=Q.
$$
For simplicity, 
we choose the height function $\xi$ with $\xi_{1}=-2$.
Then the corresponding injective map $\sx \colon J = \Z \hookrightarrow \Delta_{0}$ 
is explicitly given by
$$\sx(j)= (1, -2j) \quad  \text{for each $j \in \Z$}.$$
In this case, the associated functor $\mathscr{F}_{J}$ 
has been studied in detail by~\cite{KKK18, KKO19, KKOP19}.
Moreover, it can be seen as a suitable completion of 
the usual quantum affine Schur-Weyl duality
between the quantum loop algebra $U_{q}(L\mathfrak{sl}_{n+1})$
and the affine Hecke algebras of $GL$'s.
Moreover, our geometric interpretation 
can be obtained from
Ginzburg-Reshetikhin-Vasserot's geometric interpretation~\cite{GRV94}.
\end{Ex}

\begin{Ex}[Type $\mathsf{D}_{n}$]
Consider the following case with $N=n$:
$$
Q^{\prime}=
\left(
\begin{xy}
\def\objectstyle{\scriptstyle}
\ar@{->} *+!D{1} *\cir<2pt>{};
(7,0) *+!D{2} *\cir<2pt>{}="A", 
\ar@{->} "A"; (13,0),
\ar@{.} (14,0); (16,0),
\ar@{->} (17,0); (23,0) *+!D{n-2} *\cir<2pt>{}="B",
\ar@{->} "B"; (30,0) *+!D{n-1} *\cir<2pt>{}
\end{xy}
\hspace{6pt}
\right)
\subset Q
=\left(
\begin{xy}
\def\objectstyle{\scriptstyle}
\ar@{->} *+!D{1} *\cir<2pt>{};
(7,0) *+!D{2} *\cir<2pt>{}="A", 
\ar@{->} "A"; (13,0),
\ar@{.} (14,0); (16,0),
\ar@{->} (17,0); (23,0) *+!D{n-2} *\cir<2pt>{}="B",
\ar@{->} "B"; (29,2) *+!L{n-1} *\cir<2pt>{}
\ar@{->} "B"; (29,-2) *+!L{n} *\cir<2pt>{}
\end{xy}
\hspace{6pt}
\right).
$$
For simplicity, we choose the height function $\xi$ with $\xi_{1}=-2$.
Then the corresponding injective map $\sx \colon J=\Z \hookrightarrow \Delta_{0}$
is explicitly given by
\begin{align*}
\sx(i + kn) &= (1, -2i -2kh) \quad \text{if $1 \le i \le n-2$}, \\
\sx(n-1 + kn) &= ((n-1)^{*}, -3n+4-2kh), \\
\sx(kn) &=  ((n-1)^{*}, n-2 -2kh)
\end{align*}
where $k \in \Z$, and $h = 2n-2$ is the Coxeter number.
Note that $(n-1)^{*} = n-1$ 
when $n$ is even, 
and $(n-1)^{*} = n$ when $n$ is odd.
The associated functor $\mathscr{F}_{J}$ in this case
coincides with the one studied in~\cite[Section 6.2.4]{KKOP19}.
\end{Ex}

\begin{Ex}[Type $\mathsf{E}_{n}$]
For $n=6,7,8$, consider the following case with $N = n$:
$$
Q^{\prime}=
\left(
\begin{xy}
\def\objectstyle{\scriptstyle}
\ar@{->} *+!D{1} *\cir<2pt>{};
(7,0) *+!D{2} *\cir<2pt>{}="A", 
\ar@{->} "A"; (14,0) *+!D{3} *\cir<2pt>{}="B",
\ar@{->} "B"; (20,0),
\ar@{.} (21,0); (23,0),
\ar@{->} (24,0); (30,0) *+!D{n-1} *\cir<2pt>{}
\end{xy}
\hspace{6pt}
\right)
\subset Q
=\left(
\begin{xy}
\def\objectstyle{\scriptstyle}
\ar@{->} *+!D{1} *\cir<2pt>{};
(7,0) *+!D{2} *\cir<2pt>{}="A", 
\ar@{->} "A"; (14,0) *+!D{3} *\cir<2pt>{}="B",
\ar@{->} "B"; (20,0),
\ar@{.} (21,0); (23,0),
\ar@{->} (24,0); (30,0) *+!D{n-1} *\cir<2pt>{},
\ar@{->} "B"; (14,-5) *+!L{n} *\cir<2pt>{}
\end{xy}
\hspace{6pt}
\right).
$$
For simplicity, we choose the height function $\xi$ with $\xi_{1}=-2$.
Then the corresponding injective map $\sx \colon J=\Z \hookrightarrow \Delta_{0}$
is explicitly given by
$$
\sx(i + kn) = 
\begin{cases}
(1, -2i -2kh)  & \text{if $0 \le i \le 3$}; \\
((n-1)^{*}, n-h-2i-2kh) & \text{if $4 \le i \le n-1$,}
\end{cases}
$$
where $k \in \Z$, 
and
$h = 12, 18, 30$ is the Coxeter number
of type $\mathsf{E}_{6,7,8}$ respectively.
Note that $(n-1)^{*} = 1$ when $n=6$,
and $(n-1)^{*} = n-1$
when $n=7,8$.
\end{Ex}



\providecommand{\bysame}{\leavevmode\hbox to3em{\hrulefill}\thinspace}
\providecommand{\MR}{\relax\ifhmode\unskip\space\fi MR }
\providecommand{\MRhref}[2]{%
  \href{http://www.ams.org/mathscinet-getitem?mr=#1}{#2}
}
\providecommand{\href}[2]{#2}

\end{document}